\documentclass[11pt,oneside,english]{amsart}
\usepackage[latin9]{inputenc}
\usepackage{color}
\usepackage{babel}
\usepackage{amstext}
\usepackage{amsthm}
\usepackage{amssymb}
\usepackage[a4paper]{geometry}
\usepackage{setspace}
\usepackage[pdfusetitle,
 bookmarks=true,bookmarksnumbered=false,bookmarksopen=false,
 breaklinks=false,pdfborder={0 0 1},backref=false,colorlinks=false]
 {hyperref}

\makeatletter
\numberwithin{equation}{section}
\numberwithin{figure}{section}
\theoremstyle{plain}
\newtheorem{thm}{\protect\theoremname}[section]
\theoremstyle{plain}
\newtheorem{prop}[thm]{\protect\propositionname}
\theoremstyle{plain}
\newtheorem{conjecture}[thm]{\protect\conjecturename}
\theoremstyle{plain}
\newtheorem{lem}[thm]{\protect\lemmaname}
\theoremstyle{plain}
\newtheorem{cor}[thm]{\protect\corollaryname}
\theoremstyle{definition}
\newtheorem{defn}[thm]{\protect\definitionname}
\theoremstyle{remark}
\newtheorem*{acknowledgement*}{\protect\acknowledgementname}
\theoremstyle{remark}
\newtheorem*{org*}{\protect\orgname}
\theoremstyle{remark}
\newtheorem{rem}[thm]{\protect\remarkname}
\theoremstyle{definition}
\newtheorem{example}[thm]{\protect\examplename}
\theoremstyle{plain}
\newtheorem{fact}[thm]{\protect\factname}
\theoremstyle{plain}
\newtheorem{question}[thm]{\protect\questionname}

\usepackage{babel}
\usepackage{enumitem}

\providecommand{\corollaryname}{Corollary}
\providecommand{\definitionname}{Definition}
\providecommand{\examplename}{Example}
\providecommand{\lemmaname}{Lemma}
\providecommand{\propositionname}{Proposition}
\providecommand{\remarkname}{Remark}
\providecommand{\theoremname}{Theorem}

\usepackage{ytableau}

\providecommand{\factname}{Fact}

\providecommand{\questionname}{Question}

\makeatother
\providecommand{\acknowledgementname}{Acknowledgement}
\providecommand{\orgname}{Organization}
\providecommand{\conjecturename}{Conjecture}
\providecommand{\corollaryname}{Corollary}
\providecommand{\definitionname}{Definition}
\providecommand{\examplename}{Example}
\providecommand{\factname}{Fact}
\providecommand{\lemmaname}{Lemma}
\providecommand{\propositionname}{Proposition}
\providecommand{\questionname}{Question}
\providecommand{\remarkname}{Remark}
\providecommand{\theoremname}{Theorem}

\allowdisplaybreaks[1]

\begin{document}
\title{Anti-concentration of polynomials: $L^{p}$ balls and symmetric measures}
\author{Itay Glazer}
\address{Department of Mathematics, Technion - Israel Institute of Technology,
Haifa, Israel}
\email{glazer@technion.ac.il}
\urladdr{https://sites.google.com/view/itay-glazer}
\author{Dan Mikulincer}
\address{Department of Mathematics, University of Washington, Seattle, WA}
\email{danmiku@uw.edu}
\urladdr{https://sites.math.washington.edu/\textasciitilde danmiku}
\begin{abstract}
We begin with the observation, based on previous results, that dimension-free lower bounds on the variance of a polynomial under a log-concave measure yield dimension-free small-ball and Fourier decay estimates. Motivated by this, we establish variance bounds for polynomials on log-concave random vectors beyond the classical setting of product measures.
First, we consider the family of uniform measures on the $n$-dimensional isotropic $L^{p}$ balls.
We show that for a degree-$d$ homogeneous polynomial $f=\sum_{I}a_{I}x^{I}$,
with $\sum_{I}a_{I}^{2}=1$, the only obstruction to a dimension-free
lower bound on its variance occurs when $p=d$ is an even integer and the coefficients
of $f$ are close to those of $\frac{1}{\sqrt{n}}\left\Vert x\right\Vert _{p}^{p}$.
Second, we consider general isotropic log-concave measures that
are invariant under coordinate permutations and reflections, and determine the minimal variance for quadratic and cubic polynomials.
These variance bounds lead to new dimension-free anti-concentration results in both settings, addressing a natural extension of a question posed by Carbery and Wright.
\end{abstract}

\maketitle
\pagenumbering{arabic}

\global\long\def\N{\mathbb{N}}%
\global\long\def\R{\mathbb{\mathbb{R}}}%
\global\long\def\Z{\mathbb{Z}}%
\global\long\def\val{\mathbb{\mathrm{val}}}%
\global\long\def\Qp{\mathbb{Q}_{p}}%
\global\long\def\Zp{\mathbb{\mathbb{Z}}_{p}}%
\global\long\def\ac{\mathbb{\mathrm{ac}}}%
\global\long\def\C{\mathbb{\mathbb{C}}}%
\global\long\def\Q{\mathbb{\mathbb{Q}}}%
\global\long\def\supp{\mathbb{\mathrm{supp}}}%
\global\long\def\VF{\mathbb{\mathrm{VF}}}%
\global\long\def\RF{\mathbb{\mathrm{RF}}}%
\global\long\def\VG{\mathbb{\mathrm{VG}}}%
\global\long\def\spec{\operatorname{Spec}}%
\global\long\def\Ldp{\mathbb{\mathcal{L}_{\mathrm{DP}}}}%
\global\long\def\sgn{\mathrm{sgn}}%
\global\long\def\id{\mathrm{Id}}%
\global\long\def\Sym{\mathrm{Sym}}%
\global\long\def\Vol{\mathrm{Vol}}%
\global\long\def\cyc{\mathrm{cyc}}%
\global\long\def\U{\mathrm{U}}%
\global\long\def\SU{\mathrm{SU}}%
\global\long\def\Wg{\mathrm{Wg}}%
\global\long\def\E{\mathbb{E}}%
\global\long\def\Irr{\mathrm{Irr}}%
\global\long\def\P{\mathbb{P}}%
\global\long\def\bh{\mathbf{h}}%
\global\long\def\Span{\operatorname{Span}}%
\global\long\def\pr{\operatorname{pr}}%
\global\long\def\sgn{\operatorname{sgn}}%
\global\long\def\sG{\mathsf{G}}%
\global\long\def\sW{\mathsf{W}}%
\global\long\def\sX{\mathsf{X}}%
\global\long\def\sY{\mathsf{Y}}%
\global\long\def\sZ{\mathsf{Z}}%
\global\long\def\sH{\mathsf{H}}%
\global\long\def\sV{\mathsf{V}}%
\global\long\def\sT{\mathsf{T}}%
\global\long\def\v{\mathsf{v}}%
\global\long\def\d{\mathsf{d}}%
\global\long\def\tr{\operatorname{tr}}%
\global\long\def\lct{\operatorname{lct}}%
\global\long\def\fraku{\mathfrak{u}}%
\global\long\def\Rep{\mathrm{Rep}}%
\global\long\def\sgn{\mathrm{sgn}}%
\global\long\def\Ind{\mathrm{Ind}}%
\raggedbottom 

\section{Introduction}

Anti-concentration inequalities capture the idea that a non-constant
function of an appropriate high-dimensional random vector cannot place
too much mass in a small interval. For linear functionals,
such estimates are classical, going back to Littlewood-Offord theory
\cite{littlewood1938number,erdos1945lemma} and its subsequent refinements
\cite{rogozin1961increase,halasz1977estimates,tao2009inverse}. For nonlinear functions, and in particular polynomials, the picture is
more delicate. In this work, we focus on polynomials for which we clarify several notions of anti-concentration appearing in the literature and investigate their relationships.

Our starting point is the observation that for well-behaved high-dimensional measures, notably log-concave measures, these notions become quantitatively equivalent. Building on this, we establish
anti-concentration bounds that go beyond the traditionally
studied product measures. Our results apply to symmetric log-concave measures, with the strongest statements obtained for uniform measures on $L^p$ balls.

\subsection{Different notions of anti-concentration}
Let $\mu$ be a probability measure on $\R^n$, let $X\sim \mu$, and let $f:\R^n \to \R$. We consider three complementary frameworks for formalizing anti-concentration, corresponding to the following types of estimates:  

\begin{itemize}
\item \textbf{\emph{Variance bounds:}}
\begin{equation}
\mathrm{Var}_{\mu}(f):=\int f^{2}d\mu-\left(\int fd\mu\right)^{2}\ge c.\label{eq:variancebound}
\end{equation}
\item \textbf{\emph{Small-ball estimates:}} 
\begin{equation}
	\sup\limits_{a\in\R}\P\left(|f(X)-a|\le\varepsilon\right) \leq C\varepsilon^\alpha.\label{eq:smallball}
\end{equation}

\item \textbf{\emph{Fourier decay estimates:}}
\begin{equation}
\left|\int\limits_{\R^n}e^{\mathrm{i}tf}d\mu\right|:=\left|\mathcal{F}(f_{*}\mu)(t)\right| 
\leq C'\frac{1}{|t|^\beta}.\label{eq:FourierDecay}
\end{equation} 
\end{itemize}
Above $c,C,C',\alpha,\beta>0$ may depend on the data, \eqref{eq:smallball} should hold for arbitrarily small $\varepsilon>0$ and \eqref{eq:FourierDecay} should hold for large $|t|.$

When comparing these three notions, it is easy to observe that small-ball
estimates \eqref{eq:smallball} yield variance bounds, simply by choosing
$a=\int fd\mu$ and applying Chebyshev's inequality 
\begin{equation} \label{eq:cheby}
	\mathrm{Var}_{\mu}(f) \geq \varepsilon^2\P\left(\left|f(X)-a\right|\geq\varepsilon\right) = \varepsilon^2\left(1-\P\left(\left|f(X)-a\right|<\varepsilon\right)\right)
	.
\end{equation}
Additionally, Esseen's inequality \cite{esseen1966kolmogorv} allows to establish small-ball estimates through bounds on the Fourier transform, as in \eqref{eq:FourierDecay} and using  $|\mathcal{F}(f_*\mu)(t)|\leq 1$\footnote{In the special case $\beta = 1$ we need to add a logarithmic factor $\log(\frac{1}{\varepsilon})$},
\begin{equation} \label{eq:berry}
\sup\limits_{a\in\R}\P\left(|f(X)-a|\le\varepsilon\right) \leq \varepsilon\int_{-\frac{2\pi}{\varepsilon}}^{\frac{2\pi}{\varepsilon}}|\mathcal{F}(f_*\mu)(t)|dt \leq \varepsilon\int_{-\frac{2\pi}{\varepsilon}}^{\frac{2\pi}{\varepsilon}}\min\left(1,\frac{C'}{|t|^\beta}\right)dt \leq C''\varepsilon^{\min(1,\beta)}.
\end{equation}
In summary: $\text{\emph{Fourier decay estimates}}\implies\text{\emph{small-ball estimates }}\implies\text{\emph{variance bounds}}.$

On the other hand, in general, the reverse implications cannot hold with reasonable quantitative estimates.
For an atomic $\mu$, its small-ball probabilities will remain bounded away from $0$ even for small $\varepsilon$. As for Fourier estimates, take $X\sim\mu$
to be the uniform measure on the Cantor set. Bounds on the Hausdorff
dimension imply that $\mu$ has small-ball probabilities scaling as 
$\varepsilon^{\frac{\log(3)}{\log(2)}}$, while
a calculation shows that $\left|\mathcal{F}(\mu)\right|$ decays only logarithmically.

The upshot of the above discussion is that some structural assumptions
are needed in order to further relate the three notions. Log-concavity
is one such natural condition. We call $\mu$ \emph{log-concave} if it is
absolutely continuous and $\frac{d\mu}{dx}=e^{-\varphi}dx$, for some
convex $\varphi$. This condition rules out both types of pathologies
mentioned above, while still encompassing many examples of interest
that arise in high-dimensional probability and convex geometry \cite{artstein2015asymptotic,artstein2021asymptotic},
see $\mathsection$\ref{sec:disc} for further discussion. The starting point
of this current work is the following observation, based on \cite{kosov2018fractional,kosov2025oscillatory}, that under log-concavity
the above definitions turn out to be equivalent for polynomial functions. 
\begin{prop}
\label{thm:equivs} Let $X\sim\mu$ be a log-concave measure on $\R^{n}$
and let $f:\R^{n}\to\R$ be a polynomial of degree $d\geq 2$. Then the following
conditions are equivalent: 
\begin{enumerate}
\item \textbf{$f_{*}\mu$ has a variance bound}$~~\mathrm{Var}_{\mu}(f)\geq C_1$, for some $C_1>0$. 
\item \textbf{$f_{*}\mu$ has small-ball estimates}$~ ~\sup\limits_{a\in\R}\P\left(|f(X)-a|\leq\varepsilon\right)\leq C_2\varepsilon^{\frac{1}{d}}$, for some $C_2>0$. 
\item \textbf{The Fourier coefficients of $f_{*}\mu$ decay} at a rate of $|\mathcal{F}(f_{*}\mu)(t)|\leq\frac{C_3}{|t|^{\frac{1}{d}}}$, for some $C_3>0$.
\end{enumerate}
where the constants $C_1,C_2,C_3$ depend on each other and on $d$ only.  
\end{prop}
\begin{proof}
	$(3)\implies(2)\implies(1)$ follows from \eqref{eq:cheby} and \eqref{eq:berry}. It is left to show $(1) \implies (3)$, which follows from \cite[Theorem 3.5]{kosov2018fractional}, according to which for every smooth function $h:\R \to \R$ with  $\|h\|_{\infty} \leq 1$ and every degree-$d$ polynomial $f$,
	$$\left|\int h' \circ f d\mu\right| \leq \frac{C_d\|h'\|^{1-\frac{1}{d}}_{\infty}}{\mathrm{Var}_\mu(f)^{\frac{1}{2d}}},$$ whenever $\mu$ is log-concave.
	The implication follows by applying this bound to $h_1(x) = \pm\sin(tx)$ and $h_2(x)=\pm\cos(tx)$, since $e^{\mathrm{i}tx} = \cos(tx) + \mathrm{i}\sin(tx).$
\end{proof}

One of the key points in Proposition \ref{thm:equivs} is that the bounds
are \textbf{dimension-free}. This suggests a practical strategy
for establishing strong anti-concentration bounds; rather than attempting
to directly control the decay of the Fourier coefficients, a hard problem in general, one may
instead focus on proving a lower bound on the variance, an $L^{2}$
computation.  

As a warm-up, and to demonstrate the usefulness of Proposition \ref{thm:equivs}, we employ
it for product measures. Below we say that a measure $\mu$ is \emph{isotropic} if it is centered and its covariance matrix is the identity. Further, if $f(x) = \sum\limits_{|I|\leq d} \alpha_Ix^I$ is a degree-$d$ polynomial, where we use the standard multi-index notation, define $\mathrm{coeff}_{d}(f) := \sqrt{\sum_{|I|=d}\alpha_I^2}$. With this notation, the result \cite[Theorem 1]{GM22} states that
for any degree-$d$ polynomial $f$ and any isotropic log-concave
measure $X\sim\mu$ with \textbf{independent coordinates},
\begin{equation}
	\mathrm{Var}_{\mu}(f)\geq c_{d}\cdot\mathrm{coeff}_{d}^{2}(f).\label{eq:ourvarbound}
\end{equation}
Proposition \ref{thm:equivs} immediately implies: 
\begin{prop}
	\label{thm:indepence} For $\mu$ an isotropic log-concave product measure on $\R^n$ and $f:\R^n\to\R$ a degree-$d$ polynomial with $\mathrm{coeff}_d(f) = 1$, the Fourier coefficients
	of $f_{*}\mu$ decay as
	\[
	\left|\int\limits_{\R^n}e^{\mathrm{i}tf(x)}d\mu\right|\leq\frac{C_d}{|t|^{\frac{1}{d}}},
	\]
	where $C_d>0$ is a constant depending only on $d$.
\end{prop}
Proposition \ref{thm:indepence} appeared before in \cite{kosov2025oscillatory}
where it was also noted it provides a positive answer to a question posed
by Carbery and Wright about the uniform measure on $[-1,1]^{n}$. Crucially though, Proposition \ref{thm:equivs} makes no assumption on
independence, which suggests that Proposition \ref{thm:indepence} could
hold in much greater generality, which is our aim in this work.
\subsection{Main results: Fourier decay of log-concave measures}
To go beyond the independence assumption, we first consider $\mu_{n,p}$, the uniform measure on $B_{n,p}$, the $n$-dimensional isotropic $L^p$ ball. Concretely, 
$
B_{n,p}:=\{x\in\R^{n}:\ \|x\|_p\leq R_{n,p}\},
$
where $\|x\|_{p}^{p}=\sum|x_{i}|^{p}$ and $R_{n,p}$ is chosen to ensure the measure is isotropic, see \eqref{eq:isotropic radius}. According to \cite[Theorem 2]{GM22}
for any degree-$d$ homogeneous polynomial, 
\begin{align} \label{eq:l2norm}
\int f^{2}d\mu_{n,p}\geq C_{d}\cdot\mathrm{coeff}_d(f)^{2}.
\end{align}
However, \eqref{eq:l2norm} falls short of fitting the general machinery of Proposition
\ref{thm:equivs}; it bounds the $L^{2}$ norm, rather than the variance.
Unfortunately, this cannot be improved in general, since for $f_p(x)=\frac{1}{\sqrt{n}}\|x\|_{p}^{p}$,
$\mathrm{Var}_{\mu_{n,p}}(f_p)=o(1),$ as shown in \cite{GM22}.

In our first main result, we show that $f_p$ is essentially the only bad example. In light of Proposition \ref{thm:equivs} we are able to precisely characterize all homogeneous polynomials which admit dimension-independent decay of Fourier coefficients, in terms of the distance from the polynomial $f_p$ using the inner product induced by $\mathrm{coeff}_d$. That is, for $f(x) = \sum\limits_{|I|= d} \alpha_Ix^I, g(x) = \sum\limits_{|I|= d} \beta_Ix^I$ we write $\langle f,g\rangle = \sum\limits_I \alpha_I\beta_I$.
\begin{thm}
\label{thm:lpball} Let $n\in\N$, $p\geq1$, and let $f:\R^{n}\to\R$
be a degree-$d$ homogeneous polynomial with $\mathrm{coeff}_{d}(f)=1$.
Then, there exists a constant $c_{p,d}>0$, depending only on $p$
and $d$ such that $\mathrm{Var}_{\mu_{n,p}}(f)\geq c_{p,d}$ in the
following cases: 
\begin{enumerate}
\item Either $p\notin 2\N$ or $d\neq p$. 
\item If $d=p$ is even, and $|\langle f,f_p\rangle| \leq c_{p} $,  for some $c_{p}<1$, where $f_p(x)=\frac{1}{\sqrt{n}}\|x\|_{p}^{p}$. 
\end{enumerate}
Consequently, in those cases, for every $t \in \R$,
\[
\left|\int\limits_{\R^n}e^{\mathrm{i}tf(x)}d\mu_{n,p}\right|\leq\frac{C_{p,d}}{|t|^{\frac{1}{d}}},
\]
for some constant $C_{p,d}>0$, only depending on $p$ and
$d$. 
\end{thm}

To give an intuitive explanation of Theorem \ref{thm:lpball}, we recall 
that a uniform measure on a high-dimensional convex body is concentrated near its boundary.
It is therefore natural to expect that if the boundary is described by the zero set (or a level set)
of a polynomial, then this polynomial, when evaluated on the body, will concentrate tightly.
For $B_{n,p}$ this phenomenon only happens when $p$ is an even integer, in which case the boundary is
a level set of the degree-$p$ polynomial $f_p$. In contrast, any polynomial that is sufficiently different from $f_p$, either by having a different degree or, when $d=p$, by having $|\langle f, f_p\rangle|$ small, cannot fully exploit this boundary concentration and therefore retains a nontrivial amount of variance.\footnote{This heuristic does not immediately explain the behavior of $f_p^k$ for $k >1$. However, our proof also applies to these cases.} We emphasize that Theorem \ref{thm:lpball} requires the polynomial to be homogeneous. While Proposition \ref{thm:equivs} applies to arbitrary polynomials, extending Theorem \ref{thm:lpball} to the non-homogeneous setting would lead to a problem of a rather different flavor. In $\mathsection$\ref{sec:just} we explain why one should not expect a statement as clean as Theorem \ref{thm:lpball} in that setting.


Note that the measures in Proposition \ref{thm:indepence} and Theorem
\ref{thm:lpball} are highly symmetric; they are invariant to any
permutation of the coordinates, and $\mu_{n,p}$ is also invariant
to any coordinate reflection, or in other words \emph{$H_{n}$-invariant},
where $H_{n}$ is the group of signed permutations\footnote{$H_{n}$ is also called the \emph{hyperoctahedral group}, see $\mathsection$\ref{sec:Preliminaries-in-representation}
for a detailed discussion.}. In contrast, $\mathsection$\ref{sec:product} exhibits log-concave
measures without such symmetries for which even \eqref{eq:l2norm}
fails. This suggests that symmetry can be a useful condition
to help establish anti-concentration bounds. With this in mind, we specialize to $H_{n}$-invariant measures.

 Given such a measure $\mu$, $H_{n}$ acts on $L^{2}(\mu)$ through composition
$f\to f\circ\sigma^{-1}$ for $\sigma\in H_{n}$, and it preserves
the space of degree-$d$ polynomials. This space can be further decomposed
into $H_{n}$-invariant subspaces of polynomials, or \emph{$H_{n}$-representations}.
By applying tools from representation theory, we can now compute the
$L^{2}$-norm or variance of any polynomial. In $\mathsection$\ref{sec:Spectrum-of-unconditional,}
we give the following partial result in this general setting. 
\begin{thm}
\label{thm:symmetric} Let $\mu$ be an $H_n$-invariant measure on
$\R^{n}$ which is isotropic and log-concave, and let $f:\R^{n}\to\R$,
a degree-$d$ homogeneous polynomial with $\mathrm{coeff}_{d}(f)=1$. Then, $\mathrm{Var}_{\mu}(f)\geq c$
for $c>0$, a universal constant in the following cases: 
\begin{itemize}
\item either $d=3$, 
\item or $d=2$, and $\mathrm{Var}\left(\frac{1}{\sqrt{n}}\|x\|_{2}^{2}\right)\geq c$. 
\end{itemize}
Consequently, in those cases, there exists a constant $C>0$, such that for every $t\in \R$,
\[
\left|\int\limits_{\R^{n}}e^{\mathrm{i}tf}d\mu\right|\leq\frac{C}{|t|^{\frac{1}{d}}}.
\] 
\end{thm}
For quadratic polynomials Theorem \ref{thm:symmetric} essentially states that any bad example must be similar to the isotropic Euclidean ball, in the sense that the norm concentrates very tightly.
In contrast, for $d=3$, according to Theorem \ref{thm:symmetric}, there are no bad
examples.


\subsection{Discussion and further questions}
The statement of Theorem \ref{thm:symmetric} immediately raises the question about the behavior of higher degrees, $d > 3$. In principle, our approach should also be useful for $d>3.$
One major challenge in those cases is that we believe that already
for $d=4$, there should be many more bad cases, other than $B_{n,4}$,
making the classification much harder. To overcome this, it is natural to first consider the easier task of bounding the $L^2$-norm. Indeed, the general result from \eqref{eq:l2norm} for $L^p$-balls suggests that in this case there should be no bad examples, leading us to the following conjecture.
\begin{conjecture} \label{conj:l2bound}
	Let $d \in \N$. There exists a constant $C_d>0$, depending only on $d$, such that for every isotropic, log-concave and $H_n$-invariant measure $\mu$ on $\R^n$ and every degree-$d$ homogeneous polynomial $f:\R^n\to \R$,  
	$$\int\limits_{\R^n} f^2d\mu \geq C_d.$$
	Consequently, since homogeneous functions of odd degree are odd, in that case $$\mathrm{Var}_\mu(f) \geq C_d.$$
\end{conjecture}
Due to the difference between $L^2$-norm and variance, Conjecture \ref{conj:l2bound} predicts potential different behaviors for odd and even degrees. Let us expand on this difference, and explain the inherent difficulty in even degrees. First, when $d = 2$ the only polynomial which is both strictly convex and symmetric is,
up to scaling, $\|x\|_{2}^{2}$. It thus stands to reason that any bad example will be of the form $\mu(x)\propto e^{-\varphi(\|x\|_{2})}$, for $\varphi$ convex. This is indeed the case of $B_{n,2}$, as demonstrated by Theorem \ref{thm:symmetric}.
Further, the theorem also shows that $\varphi$ needs to be close to a convex indicator; if we consider $\varphi(x) = \frac{x^2}{2}$ we will instead obtain the standard Gaussian, which by Theorem \ref{thm:indepence} does not have the same pathological behavior.
However, already for $d = 4$, there are many more symmetric convex polynomials. Moreover, classifying these potential bad examples is a potentially hard task, since just checking whether an even degree polynomial is convex is known to be NP-hard.
In contrast, there are no convex polynomials of odd degrees, which is again consistent with Theorem \ref{thm:symmetric}.

Towards Conjecture \ref{conj:l2bound} we remark that our proof of Theorem \ref{thm:symmetric} allows us to describe the minimal value of $\int_{\R^n} f^2d\mu$ for any fixed measure $\mu$ in terms of the moments of the first $2d$ marginals of $\mu$. For a sequence of measures $\mu_n$, this leads to a numerical condition that guarantees $\liminf_n \int_{\R^n} f^2d\mu_{n}>0.$ This can be used to obtain lower bounds for many specific sequences of measures, although we could not establish it uniformly for all $H_n$-invariant log-concave measures.

Curiously, the proof of Theorem \ref{thm:symmetric} also suggests a possible candidate for the value of the constant $C_d$ in Conjecture \ref{conj:l2bound}. It turns out that up to $O(\frac{1}{n})$ terms, when $d = 2,3$, $C_d$ is attained for $\mu_{\mathrm{cube}}$, the uniform measure on the isotropic cube $\left[-\sqrt{3},\sqrt{3}\right]^n$, which is, according to a computation,
$$C_2 =\frac{4}{5}+O\left(\frac{1}{n}\right), \qquad C_3 =\frac{108}{175}+O\left(\frac{1}{n}\right).$$
Extrapolating these results leads to a potential refinement of Conjecture \ref{conj:l2bound}. 
\begin{question} \label{Q:lowerbound}
	Let $d \in \N$. Is it true that for every isotropic, log-concave and $H_n$-invariant measure $\mu$ on $\R^n$ and every degree-$d$ homogeneous polynomial $f:\R^n\to \R$,
	$$\int\limits_{\R^n} f^2d\mu \geq \inf\limits_{\mathrm{coeff}_d(\tilde{f}) = 1}\int_{\R^d}\tilde{f}^2d\mu_{\mathrm{cube}} - O\left(\frac{1}{n}\right)?$$
\end{question}
It seems unlikely that the special role that the cube plays when $d=2,3$ is coincidental. This is in line with \cite[Theorem 1.3]{paouris2012small}, according to which the cube has the worst small-ball estimates for certain multi-linear polynomials within the class of product measures. Let us remark that bounds of this form, and therefore also Theorems \ref{thm:lpball} and \ref{thm:symmetric}, go beyond anti-concentration and, for example, are also useful for comparing statistical distances between the associated measures, as in \cite{kosov2018fractional,nourdin2013convergence}.

Another question involves the exponent $\frac{1}{d}$ in Proposition \ref{thm:equivs}. As can be seen by considering the polynomial $f(x) = x^d$ in the one-dimensional setting, the equivalences in Proposition \ref{thm:equivs} are sharp and cannot be improved in general.
However, this example is degenerate, as a general polynomial on $\R^n$ will have much tamer singularities, and hence will not concentrate as tightly (see detailed discussion in $\mathsection$\ref{subsec:Anti-concentration-and-singulari} and Remark \ref{rem:singularities and anto-concentration}). In that sense, heuristically, the set of equivalences should be considerably improved for certain families of polynomials.

There are a few works that verified the above heuristic in several restricted cases. In \cite{hu2024small} the authors show that the factor
$\varepsilon^{\frac{1}{d}}$ appearing in the small-ball estimates of Proposition \ref{thm:equivs}, can be dramatically improved to $\mathrm{polylog}(1/\varepsilon)\varepsilon$
when $f$ is multi-affine and $\mu$ is log-concave with sufficient independence
between different coordinates. When $\mu$ is a product measure, this result was extended in \cite{kosov2025oscillatory} to a bound of the form $\mathrm{polylog}(1/\varepsilon)\varepsilon^{\frac{1}{m}}$,
where $m$ is the maximal individual degree of any variable, rather than
the total degree $d$. It is an interesting question whether such bounds continue to hold without requiring any independence, and whether similar improvements hold for other structured families of polynomials. Moreover, \cite{kosov2025oscillatory} proves the corresponding estimates on the Fourier decay, so we raise the question whether improved small-ball estimates also imply an improved decay of the Fourier transform.
\begin{conjecture} \label{conj:expo}
	Let $X\sim\mu$ be a log-concave measure on $\mathbb{R}^n$ and let $f:\R^n\to \R$ be a polynomial of degree $d$. Suppose that for some constant $C > 0$ and some $\alpha \in (\frac{1}{d},1)$,
	$$\sup\limits_{a\in\R}\P\left(|f(X)-a|\leq\varepsilon\right)\leq C\varepsilon^{\alpha},\text{ for every }\varepsilon>0.$$
	Then, there exists a constant $C'$, depending only on $C$ and $d$, such that
	$$|\mathcal{F}(f_*\mu)(t)|\leq\frac{C'}{|t|^\alpha},\text{ for every }t\neq 0.$$
\end{conjecture}

\subsection{Strategy of the proof of Theorem \ref{thm:lpball}} \label{sec:strat}

Denote by $\mathcal{P}_{d}(\R^{n})$ the space of degree-$d$ homogeneous
polynomials on $\R^{n}$, and let $\mathcal{P}_{d,\mathrm{unit}}(\R^{n})\subseteq\mathcal{P}_{d}(\R^{n})$
be the subset of $f$ with $\mathrm{coeff}_{d}(f)=1$. To prove Theorem
\ref{thm:lpball}, our goal is to find a lower bound on 
\[
\underset{f\in\mathcal{P}_{d,\mathrm{unit}}(\R^{n})}{\min}\mathrm{Var}_{\mu_{n,p}}(f).\tag{\ensuremath{\star}}
\]
Equivalently, $(\star)$ is the minimal eigenvalue $\lambda_{\min}$
of the covariance matrix $\mathcal{V}_{n,p,d}:=(\mathcal{V}_{IJ})_{\left|I\right|=\left|J\right|=d}$,
where $I$ and $J$ are multi-indices, and $\mathcal{V}_{I,J}:=\mathrm{Cov}_{\mu_{n,p}}(x^{I},x^{J})$.

Following Schechtman--Zinn \cite{SZ90}, a random vector $X$ in the unit $L^{p}$-sphere $S_{n,p}$ is distributed as $\frac{Z}{\left\Vert Z\right\Vert _{p}}$ where $Z=(Z_{1},...,Z_{n})\sim\gamma_{p}^{n}$, and $\gamma_{p}=\frac{1}{\frac{2}{p}\Gamma(\frac{1}{p})}e^{-|x|^{p}}dx$
is the\emph{ $p$-Gaussian measure}. Scaling and normalizing gives
the uniform distribution $\mu_{n,p}$ on the isotropic $L^{p}$-ball $B_{n,p}$. 
In \cite{GM22}, we showed that $\mathrm{Var}_{\gamma_{p}^{n}}(f)=\Omega_{p,d}(1)$
and $\E_{\mu_{n,p}}[f^{2}]=\Omega_{p,d}(1)$ for every $f\in\mathcal{P}_{d,\mathrm{unit}}(\R^{n})$.
Our goal is to utilize the connection between $\mu_{n,p}$ and $\gamma_{p}^{n}$
to further deduce that $\mathrm{Var}_{\mu_{n,p}}(f)=\Omega_{p,d}(1)$
whenever $p\neq d$. Concretely, we apply the following steps:

\begin{itemize}[leftmargin=0.5cm]

\item\textbf{Step 1}: In Proposition \ref{prop:Marginals of Euclidean ball},
using a polar integration formula, we find an explicit constant $C_{n,p,d}$
determining the proportion between $d$-marginals of $\mu_{p,n}$
and $\gamma_{p}^{n}$:
\begin{equation} \label{eq:polarintro}
\int_{\R^{n}}x_{1}^{a_{1}}...x_{k}^{a_{k}}\mu_{p,n}=C_{n,p,d}\int_{\R^{n}}x_{1}^{a_{1}}...x_{k}^{a_{k}}\gamma_{p}^{n}\text{ \,\,\,\,for every }\sum_{i=1}^{k}a_{i}=d.
\end{equation}

\item\textbf{Step 2}: Using \eqref{eq:polarintro}, we connect
$\mathrm{Var}_{\mu_{n,p}}(f)$ and $\mathrm{Var}_{\gamma_{p}^{n}}(f)$,
and deduce the following structural condition: if $f\in\mathcal{P}_{d,\mathrm{unit}}(\R^{n})$ satisfies
$\mathrm{Var}_{\mu_{n,p}}(f)\ll_{p,d}1$ then in fact $\mathrm{Var}_{\mu_{n,p}}(f)=\Omega_{p,d}(n^{-1})$, and moreover
$\E_{\mu_{n,p}}\left[f\right]^{2}=\Omega_{p,d}(n)$.
(Lemma \ref{lem:low variance means expectation large}
and Corollary \ref{cor:lower bound on variance}).\vspace{0.2cm}


\item\textbf{Step 3}: We show that the space of \emph{pathological eigenvectors} of $\mathcal{V}_{n,p,d}$,
of eigenvalue $o_{p,d}(1)$, is at most one-dimensional and consists of $H_{n}$-symmetric polynomials (Proposition \ref{prop:reduction to B_n symmetric}).

\item\textbf{Step 4}: By considering the restriction
$\mathcal{V}_{n,p,d}^{\mathrm{sym}}$ of $\mathcal{V}_{n,p,d}$ to
$H_{n}$-symmetric polynomials,
we improve upon Step 2 and show that any pathological
unit eigenvector $f$ of $\mathcal{V}_{n,p,d}$ must satisfy $\mathrm{Var}_{\mu_{n,p}}(f)=\Theta_{p,d}(n^{-1})$
and $\E_{\mu_{n,p}}\left[f\right]^2=\Theta_{p,d}(n)$ (Theorem \ref{thm:Possible potential eigenvalues}).\vspace{0.2cm}

\item\textbf{Step 5}: To further characterize the pathological eigenvectors we establish that $\|x\|_p^d$ is an approximate eigenvector in the following sense: Since $\|x\|_p^d$ is not always a polynomial, we consider $f_{n,p,d}$, the orthogonal projection of $\left\Vert x\right\Vert _{p}^{d}$
to the space of degree-$d$ $H_{n}$-symmetric
polynomials, with respect to the covariance inner product $\mathrm{Cov}_{\mu_{n,p}}(\cdot,\cdot)$. 



 Then, if there exists a pathological eigenvector, using the above steps, we show:
\begin{enumerate}
\item $\overline{f}_{n,p,d}:=\frac{f_{n,p,d}}{\mathrm{coeff}_{d}(f_{n,p,d})}$
has a pathological variance $\Theta_{p,d}(n^{-1})$ (Lemma \ref{lem:approximate eigenvector}).
\item $\E_{\gamma_{p}^{n}}\left[\left(\overline{f}_{n,p,d}-b\left\Vert x\right\Vert _{p}^{d}\right)^{2}\right]=O_{p,d}(n^{-1})$,
for $b=\frac{\E_{\mu_{n,p}}\left[\overline{f}_{n,p,d}\right]}{\E_{\mu_{n,p}}\left[\left\Vert x\right\Vert _{p}^{d}\right]}$
(Lemma \ref{lem:reduction to p_Gaussian}). 
\end{enumerate}
\item\textbf{Step 6}: We decompose $\left\Vert x\right\Vert _{p}^{d}$
in $L^{2}(\gamma_{p}^{n})$ as a sum of orthogonal polynomials:
\begin{itemize}
\item If $p>d$, or $1\leq p\leq d$ is not an even integer, then $\left\Vert x\right\Vert _{p}^{d}$
is not a polynomial function, so $\left\Vert x\right\Vert _{p}^{d}$
has a large projection in the space of homogeneous polynomials of degree $m>d$ (Lemma
\ref{lem:Fourier coefficients of p-norm}(1)). Combining with
\eqref{eq:l2norm} shows that $\E_{\gamma_{p}^{n}}\left[\left(\overline{f}_{n,p,d}-b\left\Vert x\right\Vert _{p}^{d}\right)^{2}\right]$
is large, in contradiction to Item (2) of Step 5. 
\item If $1\leq p\leq d$ is an even integer, then the projection of $\overline{f}_{n,p,d}-b\left\Vert x\right\Vert _{p}^{d}$
to degree $\leq d$ polynomials has large $\mathrm{coeff}_{d}$ (Lemma
\ref{lem:spectrum for large degrees}), again contradicting Item (2)
of Step 5. 
\end{itemize}
\end{itemize}
\subsection{Further related works}

\subsubsection{Small-ball estimates in high-dimensional convex geometry}\label{sec:disc}
There is a rich line of research dealing with small-ball probabilities in the high-dimensional probability literature. Estimates of these probabilities have proven to be fundamental in various contexts, such as asymptotic convex geometry \cite{paouris2012small, dafnis2010small, bizeul2025slicing}, random matrix theory \cite{rudelson2008littlewood,nguyen2012inverse}, boolean
analysis \cite{mossel2010noise,meka2016anti}, and high-dimensional statistics \cite{EMS21,cornacchia2025low,emschwiller2020neural, hsieh2026rigorous}.




Perhaps one of the best known results concerning polynomials is the inequality of Carbery and Wright from \cite{carbery2001distributional} (See also the works by Nazarov, Sodin,
and Volberg \cite{NSV02} and Bourgain \cite{bourgain1991distribution}):
\begin{equation} \label{eq:CW}
\sup\limits_{a\in\R}\P\left(|f(X)-a|\le\varepsilon\right)\lesssim\left(\frac{\varepsilon}{\sqrt{\mathrm{Var}_{\mu}(f)}}\right)^{\frac{1}{d}},
\end{equation}
which applies to any log-concave $X\sim \mu$ and a degree-$d$ polynomial $f$. The case of $n = 1$ goes back to P\'olya and Remez \cite{Rem36}, and \eqref{eq:CW} can be viewed as a dimension-free multivariate extension.
We briefly mention that 
the prefactor in the right-hand
side of the inequality involves $\mathrm{Var}_{\mu}(f)$, and is thus
unsuitable to deduce a variance lower bound, as in \eqref{eq:cheby}.
In turn this bars the possibility to establish Fourier decay
estimates. For product measures, that specific gap was closed in \cite{GM22} where the
authors showed that $\mathrm{Var}_{\mu}(f)$ can be replaced by the measure-independent quantity $\mathrm{coeff}_{d}(f)$, leading to Theorem \ref{thm:indepence}. 

\subsubsection{Decay of Fourier coefficients}
The connection between small-ball estimates and decay of Fourier coefficients goes back to the classical van der Corput Lemma. Indeed, the proof in \cite[Proposition 2]{ste93} essentially establishes an equivalence between these two forms of anti-concentration for the uniform measure on a one-dimensional interval. In high-dimensions, this connection was later explored in \cite{carbery1999multidimensional,carbery2002what}, where an analog statement of Theorem \ref{thm:indepence} was proven with a dimension-dependent constant.

Other than being an intrinsic quantity of interest, the decay of the Fourier transform is useful in many contexts. Bounds on the decay of the characteristic function go back to Statuljavi$\check{\text{c}}$us \cite{statu1965limit}, as well as the Edgeworth expansion \cite{bhattacharya2010normal}, and form an important component for establishing local limit theorems. We refer the reader to \cite{bobkov2025berry} for a more modern treatment that also explains why dimension-free bounds are useful for multivariate approximations. 


\subsubsection{Anti-concentration and singularities of polynomial maps}\label{subsec:Anti-concentration-and-singulari}

For polynomials, the anti-concentration estimates as in Proposition \ref{thm:equivs} are
closely related to the singularities of the map $f:\R^{n}\rightarrow\R$.
Set $f_{x_{0}}(x):=f(x)-f(x_{0})$, and consider the following local
invariants:
\begin{align*}
\alpha_{f,x_{0},\R} & :=\sup\limits_{\alpha>0}\left\{ \exists\text{ ball }B\ni x_{0}\text{ such that }(f_{x_{0}})_{*}\mu_{B}\left([-\varepsilon,\varepsilon]\right)\leq\varepsilon^{\alpha}\,\,\forall0<\varepsilon\ll1\right\} ,\\
\beta_{f,x_{0},\R} & :=\sup\limits_{\beta>0}\left\{ \exists\text{ ball }B\ni x_{0}\text{ such that }\left|\mathcal{F}\left((f_{x_{0}})_{*}\mu_{B}\right)(t)\right|\leq |t|^{-\beta}\,\,\,\forall t\gg1\right\} ,
\end{align*}
where $\mu_{B}$ is the uniform measure on a ball $B\subseteq\R^{n}$.
We further define global invariants: 
\[
\alpha_{f,\R}:=\underset{x_{0}\in\R^{n}}{\inf}\alpha_{f,x_{0},\R}\text{ \,\,\,and \,\,\,}\beta_{f,\R}:=\underset{x_{0}\in\R^{n}}{\inf}\beta_{f,x_{0},\R}
\]
The numbers $\alpha_{f,x_{0},\R},\beta_{f,x_{0},\R},\alpha_{f,\R}$
and $\beta_{f,\R}$ are known to be rational, and $\alpha_{x_{0},f,\R}$
coincides with the \emph{real log-canonical threshold} (see \cite{Saia}),
defined as 
\begin{equation}
\operatorname{lct}_{\R}(f_{x_{0}};x_{0})=\sup\left\{ s\in\R:\int_{\R^{n}}\left|f_{x_{0}}(x)\right|^{-s}d\mu_{B}<\infty\text{ for some ball }B\ni x_{0}\right\} .\label{eq:rlct}
\end{equation}
It can be effectively computed by applying a suitable change of coordinates,
called (\emph{embedded}) \emph{resolution of singularities}, to the
integral in (\ref{eq:rlct}) after which both $f_{x_{0}}$ and the
Jacobian are locally monomials (see e.g.~\cite[Lemma 2.3]{GHS})\footnote{The existence of this map follows from Hironaka's resolution of singularities
\cite{Hir64}.}. The number $\beta_{f,x_{0},\R}$ is called the \emph{oscillation
index} (\cite[Chapters 6-8]{AGV88}), and it satisfies $\alpha_{x_{0},f,\R}=\min\left\{ 1,\beta_{f,x_{0},\R}\right\} $. 

When $\R$ is replaced with $\C$, $\alpha_{f,x_{0},\C}$ is equal
to the \emph{log-canonical threshold} $\operatorname{lct}(f_{x_{0}};x_{0})$,
a singularity invariant which plays an important role in birational
geometry and the minimal model program (see \cite{Sho92,Kol97,Mus12}).
The complex oscillation index $\beta_{f,x_{0},\C}$ is conjecturally
equal to the \emph{minimal exponent}
(see \cite{Sai94} and \cite[p.460]{MP20} for a precise definition).

We now explain the connection between anti-concentration and singularities. In the discussion below, suppose for simplicity that $\mu$ is smooth
and compactly supported in $\R^{n}$. 

We say that $f:\R^{n}\rightarrow\R$ is \emph{smooth} at $x_{0}\in\R^{n}$
if it is a submersion at $x_{0}$. Otherwise, $f$ is \emph{singular}
at $x_{0}$. Note that if $f$ is smooth at the support of $\mu$, then $f_{*}\mu$
is a smooth measure. In particular, $\alpha_{f,\R}=1$ and $\mathcal{F}(f_{*}\mu)(t)$
decay faster than any polynomial, and hence $\beta_{f,\R}=\infty$.
Thus, \textbf{smooth maps have optimal small ball and Fourier exponents}.

If $f$ is singular at $x_{0}$, one can measure how far it is from being
a submersion by considering the dimension of $\ker d_{x_{0}}f\subseteq T_{x_{0}}\R^{n}$.
This quantitative measurement of singularities can be improved if
one takes into account higher orders of vanishing of curves, rather
than just linear ones captured by the tangent space $T_{x_{0}}\R^{n}$.
By a result of Musta\c{t}\u{a} \cite[Corollary 0.2]{Mus02}, this
is \textbf{precisely} controlled by $\operatorname{lct}(f_{x_{0}};x_{0})$.
The following remark quantitatively illustrates how \textbf{polynomials
tend to concentrate around their singularities}.

\begin{rem}
\label{rem:singularities and anto-concentration} Suppose that $\mu$
is smooth and compactly supported in $\R^{n}$. Then $f_{*}\mu$ and $\mathcal{F}(f_{*}\mu)$ admit asymptotic expansions,
as $\varepsilon\rightarrow0$ and $t\rightarrow\infty$ respectively
(see e.g.~\cite{Igu78}, \cite[Theorem 3.6]{GHS} and \cite[Theorems 7.1 and 7.3]{AGV88}). Explicitly, there
exists $C>1$ such that for every $0<\varepsilon\ll1$, every $x_{0}\in\R^{n}$,
and every $t\gg1$,
\begin{equation}
\P\left(\left|f_{x_{0}}(X)\right|\leq\varepsilon\right)\leq C\left(\ln\left|\varepsilon\right|\right)^{n-1}\varepsilon^{\alpha_{f,\R}}\text{ and }\left|\mathcal{F}(f_{*}\mu)(t)\right|<C\left(\ln\left|t\right|\right)^{n-1}\left|t\right|^{-\beta_{f,\R}}.\label{eq:tight small ball}
\end{equation}
This shows that $\alpha_{f,\R}$ and $\beta_{f,\R}$ are the
\textbf{tightest small ball and Fourier exponents}, asymptotically
for very small balls and very high frequencies.
In addition, one always has $\alpha_{f,\R},\beta_{f,\R}\geq\frac{1}{d}$,
and if $\alpha_{f,\R}=\frac{1}{d}$ there is no logarithmic part,
so that (\ref{eq:tight small ball}) agrees with the Carbery\textendash Wright
inequality (\ref{eq:CW}), although the constant $C$ could depend the different parameters of the problem, such as the dimension.
\end{rem}

\subsubsection{Anti-concentration over arbitrary local fields}\label{subsec:Anti-concentration-over-local}

Let $F$ be a local field. It is \emph{Archimedean} if $F\in\left\{ \R,\C\right\}$, and it is \emph{non-Archimedean} if $F$ is either a finite extension of the field of $p$-adic numbers $\Q_{p}$,
or the formal Laurent series over a finite field $\mathbb{F}_{q}((t))$.
The non-Archimedean version $\int_{\Q_{p}^{n}}\left|f_{x_{0}}(x)\right|^{-s}d\mu$
of the integral in (\ref{eq:rlct}) is called \emph{Igusa's local zeta
function}, and was investigated by Igusa, Denef, Loeser and others
\cite{Igu78,Den87,Den91a,DL98,Igu00}. The \emph{$F$-analytic log-canonical
threshold }is defined similarly to (\ref{eq:rlct}) and is equal to
$\alpha_{f,x_{0},F}$. When $F=\mathbb{F}_{q}((t))$, $\alpha_{f,x_{0},F}$
is related to a singularity invariant called the \emph{$F$-pure threshold}
(see e.g.~\cite{TW04} and \cite[Section 2]{Mus12}), and it
was recently shown in \cite[Theorem 1.4]{GH} that $\alpha_{f,x_{0},F}\geq\frac{1}{d}>0$.
When $F$ is a $p$-adic field, one can establish tight Small-ball and Fourier estimates, similar to the estimates in (\ref{eq:tight small ball}) (see e.g.~\cite{Igu78}
and \cite[Corollary 2.9]{VZG08}).

In the $p$-adic world, small-ball estimates translate into estimates on the number of solutions of congruences of $f$ modulo $p^{k}$. Indeed, if $f\in\Z[x_{1},...,x_{n}]$, and if $\mu_{\mathbb{Z}_{p}^{n}}$
is the Haar probability measure on the ring of $p$-adic integers
$\mathbb{Z}_{p}^{n}$ in $\Q_{p}^{n}$, then for each $k\in\N$: 
\begin{equation}
\mu_{\Zp^{n}}\left(\left|f(x)\right|_{p}\leq p^{-k}\right)=\frac{\#\left\{ a\in\left(\Z/p^{k}\Z\right)^{n}:f(a)=0\mod p^{k}\right\} }{p^{kn}}.\label{eq:small ball and solutions of congruences}
\end{equation}
Similarly, Fourier estimates translate to the study of \textbf{exponential
sums}, as the Fourier coefficients of $f_{*}\mu_{\Z_{p}^{n}}$ are essentially
of the form: 
\begin{equation}
\frac{1}{p^{kn}}\sum_{x\in(\Z/p^{k}\Z)^{n}}\exp\left(\frac{2\pi if(x)}{p^{k}}\right).\label{eq:exponential sums}
\end{equation}
Field independent small-ball and Fourier estimates were given in \cite[Theorem 4.12]{CGH23}.
and \cite[Theorem 1.5]{CMN19}, respectively. Moreover, $p$-adic
analogues of the Van der Corput lemma were given in \cite{Rog05,Clu11}.

\subsubsection{Applications to algebraic combinatorics}\label{subsec:Algebraic combinatorics}

Anti-concentration of polynomials $f:R^{n}\rightarrow R$ can be studied over any ring $R\in\left\{ \R,\C,\Qp,\mathbb{F}_{p}((t)),\mathbb{F}_{p},\Z_{p},\Z/p^{k}\Z\right\} $.
With the philosophy of $\mathsection$\ref{subsec:Anti-concentration-and-singulari}
in mind, if $f$ is sufficiently generic, it should have mild singularities,
and thus good small-ball estimates. The following notion plays
a role in \textbf{algebraic combinatorics.}
\begin{defn}[\cite{Sch85}]
\label{def:strength}Let $f$ be a polynomial of degree $d$. The \emph{strength} (or \emph{Schmidt
rank}) of $f$, denoted $\mathrm{str}(f)$, is the minimal $k\in\N$
such that $f$ can be written as 
\[
f=f_{1}g_{1}+...+f_{k}g_{k},
\]
with $f_{1},...,f_{k},g_{1},...,g_{k}$ are of degree $\leq d-1$.
If $d=1$ we define the strength to be $\infty$.
\end{defn}

Green\textendash Tao \cite{GT09} and later Bhowmick\textendash Lovett
\cite{BL}, showed that if $\mu_{\mathbb{F}_{p}^{n}}$ is the uniform
measure on $\mathbb{F}_{p}^{n}$, then $\left\Vert \frac{f_{*}\mu_{\mathbb{F}_{p}^{n}}}{\mu_{\mathbb{F}_{p}}}-1\right\Vert _{\infty}<\delta$,
whenever $f$ is sufficiently strong, i.e.~$\mathrm{str}(f)>C(\delta,d)$. The key point is
that the required strength is \textbf{dimension-independent}.
More recently, Kazhdan\textendash Ziegler \cite[Theorem 8.3]{KZ21} generalized the above results to $p$-adic fields. 

\subsubsection{Small-ball estimates and asymptotic group theory}\label{subsec:Asymptotic group theory}

The setting in $\mathsection$\ref{subsec:Anti-concentration-over-local}
can be further generalized to polynomial mappings $f=(f_{1},...,f_{m}):F^{n}\rightarrow F^{m}$.
Even more generally, if $X\subseteq F^{n}$ and $Y\subseteq F^{n}$ are
$F$-analytic manifolds, $f(X)\subseteq Y$, and $\mu$ is a smooth,
compactly supported measure on $X$, one can study anti-concentration
of $f_{*}\mu$ around small balls in $Y$.

A model case naturally arises in group theory. If $w(x_{1},x_{2})=x_{i_{1}}^{\varepsilon_{1}}...x_{i_{\ell}}^{\varepsilon_{\ell}}$,
with $\varepsilon_{\ell}\in\left\{ \pm1\right\} $, and $i_\ell \in \{1,2\}$ is a formal word in
$2$ letters, and $G$ is a group, one can associate a \emph{word map} $w_{G}:G^{2}\rightarrow G$,
by $(g_{1},g_{2})\mapsto w(g_{1},g_{2})$. For example, the commutator
word $w=xyx^{-1}y^{-1}$ induces the commutator map on $G$. If $G$
is a compact Lie group such as $\mathrm{SO}_{n}$ or $\mathrm{SU}_{n}$,
or a $p$-adic analytic group such as $\mathrm{SL}_{n}(\Zp)$, then $w_{G}:G^{2}\rightarrow G$
is a polynomial map. Taking $\mu=\mu_{G^{2}}$ to be the Haar probability
measure fits into the above setting. 

In asymptotic group theory, one often studies families groups of growing
dimension. In \cite{AGL}, every word
$w(x,y)$ is shown to have a dimension-independent small ball exponent $\varepsilon(w)>0$
on $\left\{ \mathrm{SU}_{n}\right\} _{n\in\N}$. That is, for every
$n\in\N$, every $A\in\SU_{n}$ and every $0<\delta\ll_{w}1$:
\[
\mathbb{P}\Big(w_{\SU_{n}}(x,y)\in B(A,\delta)\Big)\leq\left(\mu_{\SU_{n}}(B(A,\delta))\right)^{\varepsilon(w)},
\]
where $B(A,\delta)\subseteq\SU_{n}$ is a ball of radius $\delta$
in the Hilbert\textendash Schmidt metric, normalized so that the diameter
of $\SU_{n}$ is $1$. Similar results were obtained for $\left\{ \mathrm{SL}_{n}(\Zp)\right\} _{n,p}$
in \cite[Theorem I]{GHb}. In addition, in the spirit of $\mathsection$\ref{subsec:Algebraic combinatorics},
bounds of the form $\left|w_{\mathrm{SL}_{n}(\mathbb{F}_{p})}^{-1}(g)\right|<\left|\mathrm{SL}_{n}(\mathbb{F}_{p})\right|^{2-\varepsilon(w)}$
were obtained in \cite{LaS12,LST19}, uniformly
in $n$ and $p$. Thinking of $\{g\}$ as a ``small ball''
of volume $\left|\mathrm{SL}_{n}(\mathbb{F}_{p})\right|^{-1}$, this
corresponds to a small ball exponent of $\varepsilon(w)>0$.

\begin{org*}
    The rest of the paper is organized as follows: In Section \ref{sec:lpballs} we prove Theorem \ref{thm:lpball}. Section \ref{sec:Preliminaries-in-representation} is devoted to the necessary background on the representation theory of $H_n$, which we then use in Section \ref{sec:Spectrum-of-unconditional,} to study $H_n$-invariant measures. These results lead to the proof of Theorem \ref{thm:symmetric} whose main technical part we defer to Appendix \ref{sec:Proof-of-Propositions}. Finally, Section \ref{sec:just} contains a discussion on the different conditions and assumptions we impose in the paper.
\end{org*}
\begin{acknowledgement*}
We are grateful to Egor Kosov for several enlightening discussions. In particular, Lemma \ref{lem:Inequality on mements} is due to him. We also thank Nir Avni, Max Gurevich, and Emanuel Milman for
useful discussions. I.G.~was supported by ISF grant 3422/24. D.M.~was partially supported by the Brian and Tiffinie Pang Faculty Fellowship.
\end{acknowledgement*}

\section{Fourier coefficients of $L^{p}$-balls} \label{sec:lpballs}
Our aim in this section is to follow the strategy outlined in $\mathsection$\ref{sec:strat}, leading to the proof of Theorem \ref{thm:lpball}. We begin by introducing some relevant notation.
We denote $\widetilde{B}_{n,p}:=\left\{ x\in\R^{n}:\left\Vert x\right\Vert _{p}=1\right\} $
as the unit $L^{p}$-ball. Let $Z=(Z_{1},...,Z_{n})\sim\gamma_{p}^{n}$
where $\gamma_{p}=\frac{1}{\frac{2}{p}\Gamma(\frac{1}{p})}e^{-|x|^{p}}dx$
is the \emph{$p$-Gaussian measure}, so that the density of $\gamma_{p}^{n}$
is given by
\begin{equation}
\rho_{n}(x_{1},...,x_{n}):=\frac{1}{\left(\frac{2}{p}\Gamma(\frac{1}{p})\right)^{n}}e^{-\left\Vert x\right\Vert _{p}^{p}}.\label{eq:p_Gaussian}
\end{equation}
By \cite{SZ90}, the random vector 
\begin{equation}
\widetilde{X}:=\mathrm{U}^{1/n}\frac{Z}{\left\Vert Z\right\Vert _{p}}\label{eq:L^p ball from Gaussian}
\end{equation}
is uniformly distributed on $\widetilde{B}_{n,p}$, where $\mathrm{U}$
is indepdent from $Z$ and uniformly distributed on $[0,1]$. Setting
\begin{equation}
R_{n,p}:=\E[\widetilde{X}_{1}^{2}]^{-\frac{1}{2}}=\left(\frac{n+2}{n}\cdot\frac{\Gamma(\frac{n+2}{p})}{\Gamma(\frac{n}{p})}\cdot\frac{\Gamma(\frac{1}{p})}{\Gamma(\frac{3}{p})}\right)^{\frac{1}{2}} = \Theta_p\left(n^{\frac{1}{p}}\right),\label{eq:isotropic radius}
\end{equation}
the random vector $X:=R_{n,p}\widetilde{X}$ is
uniformly distributed on the isotropic $L^{p}$-ball $B_{n,p}$ (see e.g.~\cite[end of p.490]{BGMN05}). Write
$\mu_{n,p}:=\frac{1}{\mathrm{Vol}(B_{n,p})}1_{B_{n,p}}dx$, for the law of $X$.

Let $\mathcal{P}_{d}(\R^{n})$ be the space of all real-valued homogeneous
polynomials of degree $d$, and let $\mathcal{P}_{d,\mathrm{unit}}(\R^{n})\subseteq\mathcal{P}_{d}(\R^{n})$
be the subset of $f$ with $\mathrm{coeff}_{d}(f)=1$. In this section
we prove the following Theorem \ref{thm:Main theorem L^p}, and then derive Theorem \ref{thm:lpball} as an easy consequence.
\begin{thm}
\label{thm:Main theorem L^p}Let $d\in\N$ and let $p\in\R_{\geq1}$.
Then for every $n\in\N$: 
\[
\underset{f\in\mathcal{P}_{d,\mathrm{unit}}(\R^{n})}{\min}\mathrm{Var}_{\mu_{n,p}}(f)=\begin{cases}
\Theta_{p,d}(n^{-1}) & \text{if }p=d \text{ and } p \in 2\mathbb{Z}\\
\Theta_{p,d}(1) & \text{otherwise.}
\end{cases}
\]
\end{thm}

\subsection{Some integrals on $L^{p}$-balls and $p$-Gaussian measures}
Our first step towards Theorem \ref{thm:Main theorem L^p} is to express the moments of $\mu_{n,p}$ through those of $\gamma_p^{n}$.

For each $a_{1},...,a_{k}\in\N$, let 
\[
\alpha_{(a_{1},...,a_{k}),n,p}:=\E_{\mu_{n,p}}\left[x_{1}^{a_{1}}...x_{k}^{a_{k}}\right] = \frac{1}{\mathrm{Vol}(B_{n,p})}\int\limits_{B_{n,p}} x_{1}^{a_{1}}...x_{k}^{a_{k}}dx.
\]
Further denote 
\begin{equation} \label{eq:gammamoments}
\beta_{k,p}:=\E_{\gamma_p}\left[x^{k}\right] = \int_{\R}x^{k}\frac{1}{\frac{2}{p}\Gamma(\frac{1}{p})}e^{-\left|x\right|^{p}}dx.
\end{equation}
\begin{lem}
\label{lem:moments of p-Gaussian}For each $k\in\N$ we have 
\[
\beta_{k,p}:=\begin{cases}
\frac{\Gamma(\frac{k+1}{p})}{\Gamma(\frac{1}{p})} & \text{if }k\text{ is even}\\
0 & \text{if }k\text{ is odd}
\end{cases}
\]
\end{lem}

\begin{proof}
The case $k\in2\N$ follows from \cite[Lemma 6]{GM22}. If $k$ is
odd, the integral vanishes since $e^{-\left|x\right|^{p}}$ is even. 
\end{proof}
For $\mu_{n,p}$, a key property that follows from its homogeneity is that one can integrate out powers of $\|\cdot\|_p$ in a straightforward manner.
\begin{lem}
\label{lem:properties of integral on L_p ball}For every $f\in\mathcal{P}_{d}(\R^{n})$
and every $\beta\geq-d$, we have, 
\[
\E_{\mu_{n,p}}\left[\left\Vert x\right\Vert _{p}^{\beta}f(x)\right]=\frac{n+d}{n+d+\beta}R_{n,p}^{\beta}\E_{\mu_{n,p}}\left[f\right].
\]
In particular, by taking $f\equiv 1$,
$$\E_{\mu_{n,p}}\left[\left\Vert x\right\Vert _{p}^{\beta}\right] = \frac{n}{n+\beta}R_{n,p}^{\beta}.$$
A similar relation holds for $\gamma_{p}^n$,
\[
\E_{\gamma_p^n}\left[\left\Vert x\right\Vert _{p}^{\beta}f(x)\right]=\frac{\Gamma\left(\frac{n+d+\beta}{p}\right)}{\Gamma\left(\frac{n+d}{p}\right)}\E_{\gamma_p^n}\left[f\right].
\]
\end{lem}

\begin{proof}
Denote by $\mu_{S_{n,p}}$ the cone probability measure on the unit
$L^{p}$-sphere $S_{n,p}:=\partial\widetilde{B}_{n,p}$. Using the polar
integration formula (see e.g.~\cite[Page 485]{BGMN05}), we have:
\begin{align}
\E_{\mu_{n,p}}\left[\left\Vert x\right\Vert _{p}^{\beta}f(x)\right] & =\frac{1}{\mathrm{Vol}(B_{n,p})}\int_{B_{n,p}}\left\Vert x\right\Vert _{p}^{\beta}f(x)dx=n\frac{\mathrm{Vol}(\widetilde{B}_{n,p})}{\mathrm{Vol}(B_{n,p})}\int_{0}^{R_{n,p}}r^{n-1}\int_{S_{n,p}}\left\Vert rz\right\Vert _{p}^{\beta}f(rz)\mu_{S_{n,p}}(z)\nonumber \\
 & =nR_{n,p}^{-n}\int_{0}^{R_{n,p}}r^{n+d+\beta-1}dr\int_{S_{n,p}}f(z)\mu_{S_{n,p}}(z)=\frac{n}{n+d+\beta}R_{n,p}^{d+\beta}\int_{S_{n,p}}f(z)\mu_{S_{n,p}}(z).\label{eq:polar integration 1}
\end{align}
Similarly, 
\begin{equation}
\E_{\mu_{n,p}}\left[f\right]=nR_{n,p}^{-n}\int_{0}^{R_{n,p}}r^{n+d-1}\int_{S_{n,p}}f(z)\mu_{S_{n,p}}(z)=\frac{n}{n+d}R_{n,p}^{d}\int_{S_{n,p}}f(z)\mu_{S_{n,p}}(z).\label{eq:poar integration 2}
\end{equation}
The lemma follows by combining \eqref{eq:polar integration 1} with
\eqref{eq:poar integration 2}. The proof for $\gamma_{p}^n$ is virtually identical except that the radial part now contains the term $e^{-r^p}$, leading to a Gamma integral.
\end{proof}
Denote by 
\begin{equation}
C_{n,p,d}:=R_{n,p}^{d}\frac{n}{n+d}\frac{\Gamma(\frac{n}{p})}{\Gamma(\frac{n+d}{p})}.\label{eq:formula for C}
\end{equation}
We next show that $C_{n,p,d}$ is the proportion between the $d$-moments of $\mu_{n,p}$ and $\gamma_p^n$.
\begin{prop}
\label{prop:Marginals of Euclidean ball}Let $d\in\N$ and let $a_{1},...,a_{k}\in\N$,
with $\sum_{i=1}^{k}a_{i}=d$. Then: 
\[
\alpha_{(a_{1},...,a_{k}),n,p}=C_{n,p,d}\prod_{i=1}^{k}\beta_{a_{i},p}.
\]
Consequently, by linearity of the expectation, if $f \in \mathcal{P}_d(\R^n)$,
$$\E_{\mu_{n,p}}\left[f\right] = C_{n,p,d}\E_{\gamma_p^n}\left[f\right].$$
\end{prop}

\begin{proof}
If one of the $a_{i}$'s is odd, then $\alpha_{(a_{1},...,a_{k}),n,p}=C_{n,p,d}\prod_{i=1}^{k}\beta_{a_{i},p}=0$,
since $\mu_{n,p}$ and $\gamma_{p}^{n}$ are invariant to reflections.
Hence, we may assume $a_{1},...,a_{k}\in2\N$. By (\ref{eq:poar integration 2}),
(\ref{eq:L^p ball from Gaussian}) and by Lemma \ref{lem:properties of integral on L_p ball},
\begin{align*}
\alpha_{(a_{1},...,a_{k}),n,p} & =\E_{\mu_{n,p}}\left[x_{1}^{a_{1}}...x_{k}^{a_{k}}\right]=R_{n,p}^{d}\frac{n}{n+d}\int_{S_{n,p}}z_{1}^{a_{1}}...z_{k}^{a_{k}}\mu_{S_{n,p}}(z)\\
& =R_{n,p}^{d}\frac{n}{n+d}\int_{\R^{n}}\frac{x_{1}^{a_{1}}...x_{k}^{a_{k}}}{\left\Vert x\right\Vert _{p}^{d}}\gamma_{p}^{n}(x)=C_{n,p,d}\int_{\R^{n}}x_{1}^{a_{1}}...x_{k}^{a_{k}}\gamma_{p}^{n}(x)=C_{n,p,d}\prod_{i=1}^{k}\beta_{a_{i},p},
\end{align*}
as required. 
\end{proof}

When computing the variance $\mathrm{Var}_{\mu_{n,p}}(f) = \E_{\mu_{n,p}}\left[ f^2 \right]- \E_{\mu_{n,p}}\left[ f \right]^2$, we need to deal with both $C_{n,p,2d}$ and $C^2_{n,p,d}$. We now show that the ratio of these constants has a tractable expansion.
\begin{lem}
\label{lem:asymptotic expansion}The function $n\mapsto\frac{C_{n,p,d}^{2}}{C_{n,p,2d}}$
admits an asymptotic expansion in $n^{-1}$. Namely, there exist real
numbers $\left\{ a_{p,d,k}\right\} _{k=0}^{\infty}$ such that for
every $D\in\N$: 
\[
\frac{C_{n,p,d}^{2}}{C_{n,p,2d}}=\frac{n(n+2d)}{(n+d)^{2}}\frac{\Gamma(\frac{n}{p})\Gamma(\frac{n+2d}{p})}{\Gamma(\frac{n+d}{p})^{2}}=\sum_{k=0}^{D}\frac{a_{p,d,k}}{n^{k}}+O_{p,d}(n^{-(D+1)}).
\]
Moreover, 
\[
\frac{C_{n,p,d}^{2}}{C_{n,p,2d}}=1+\frac{d^{2}}{pn}+O_{p,d}(n^{-2}).
\]
\end{lem}

\begin{proof}
By \cite{ET51} (see also \cite[Eq. 4.3]{ELe15}), we have an asymptotic expansion: 
\begin{equation}
\frac{\Gamma(x+t)}{\Gamma(x+s)}\sim x^{t-s}\sum_{k=0}^{\infty}\frac{(-1)^{k}B_{k}^{(t-s+1)}(t)\cdot(t-s)_{k}}{k!}x^{-k},\text{ as }x\rightarrow\infty,\label{eq:ratios of gamma}
\end{equation}
where $(t-s)_{k}:=(t-s)(t-s+1)...(t-s+k-1)$ is the rising factorial and where $B_{k}^{(\alpha)}(y)$ is the generalized Bernoulli polynomials,
defined by 
\[
\frac{x^{\alpha}e^{yx}}{(e^{x}-1)^{\alpha}}=\sum_{k=0}^{\infty}B_{k}^{(\alpha)}(y)\frac{x^{k}}{k!}.
\]
Set $x=\frac{n}{p}$, $t=0$ and $s=\frac{u}{p}$ for $u\in\N$, and denote $b_{p,u,k}:=p^{k}\frac{(-1)^{k}B_{k}^{(1-\frac{u}{p})}(0)\cdot(-\frac{u}{p})_{k}}{k!}$.
An explicit computation shows that
\[
B_{0}^{(1-\frac{u}{p})}(0)=1,\text{\,\,\,\,\,and \,\,\,\,}B_{1}^{(1-\frac{u}{p})}(0)=\frac{p-u}{2p}.
\]
Hence, for every $D\in\N$ and every $u\in\N$, we have: 
\begin{equation}\label{eq:ratios of Gamma}
\frac{\Gamma(\frac{n}{p})}{\Gamma(\frac{n+u}{p})}=\left(\frac{n}{p}\right)^{-\frac{u}{p}}\left(\sum_{k=0}^{D}b_{p,u,k}n^{-k}+O_{p,u}(n^{-(D+1)})\right)=\left(\frac{n}{p}\right)^{-\frac{u}{p}}\left(1+\frac{u(p-u)}{2pn}+O_{p,u}(n^{-2})\right).
\end{equation}
Thus, taking the square of \eqref{eq:ratios of Gamma} with $u=d$, and dividing by \eqref{eq:ratios of Gamma} with $u=2d$ yields:
\begin{align*}
\frac{n(n+2d)}{(n+d)^{2}}\frac{\Gamma(\frac{n}{p})\Gamma(\frac{n+2d}{p})}{\Gamma(\frac{n+d}{p})^{2}} & =\frac{n(n+2d)}{(n+d)^{2}}\frac{\left(\sum_{k=0}^{D}b_{p,d,k}n^{-k}+O_{p,d}(n^{-(D+1)})\right)^{2}}{\sum_{k=0}^{D}b_{p,2d,k}n^{-k}+O_{p,d}(n^{-(D+1)})},\\
 & =\sum_{k=0}^{D}\frac{a_{p,d,k}}{n^{k}}+O_{p,d}(n^{-(D+1)})
\end{align*}
for suitable $\left\{ a_{p,d,k}\right\} _{k=0}^{\infty}$. Moreover,
since $\frac{n(n+2d)}{(n+d)^{2}}=1-\frac{d^{2}}{(n+d)^{2}}=1+O_{d}(n^{-2})$,
we have: 
\[
\frac{C_{n,p,d}^{2}}{C_{n,p,2d}}=1+\frac{d(p-d)}{pn}-\frac{d(p-2d)}{pn}+O_{p,d}(n^{-2})=1+\frac{d^{2}}{pn}+O_{p,d}(n^{-2}).
\]
This concludes the lemma. 
\end{proof}
We can now prove an a priori estimate for potential values of $\mathrm{Var}_{\mu_{n,p}}(f)$ when $f\in\mathcal{P}_{d,\mathrm{unit}}(\R^{n})$; If $\mathrm{Var}_{\mu_{n,p}}(f)$ is smaller than some constant, then $\left|\E_{\mu_{n,p}}\left[f\right]\right|$ must grow as a function of $n$.
\begin{lem}
\label{lem:low variance means expectation large}Let $d\in\N$
and $p\geq1$. There exists a constant $c(d,p)>0$, such that if $f\in\mathcal{P}_{d,\mathrm{unit}}(\R^{n})$
and $\mathrm{Var}_{\mu_{n,p}}(f)<c(d,p)$, then 
\[
\E_{\mu_{n,p}}\left[f\right]^{2}\geq c(d,p)n.
\]
\end{lem}

\begin{proof}
Since $\E_{\mu_{n,p}}[f]^{2}$ is uniformly bounded from below for
every fixed $n$, we may assume $n\gg_{d,p}1$. By Proposition \ref{prop:Marginals of Euclidean ball}
and Lemma \ref{lem:asymptotic expansion}, the following holds for
$n\gg_{p,d}1$: 
\begin{align*}
\mathrm{Var}_{\mu_{n,p}}\left(f\right) & =\E_{\mu_{n,p}}\left[f^{2}\right]-\E_{\mu_{n,p}}\left[f\right]^{2}=C_{n,p,2d}\E_{\gamma_{p}^{n}}\left[f^{2}\right]-C_{n,p,d}^{2}\E_{\gamma_{p}^{n}}\left[f\right]^{2}\\
 & =C_{n,p,2d}\left(\E_{\gamma_{p}^{n}}\left[f^{2}\right]-\frac{C_{n,p,d}^{2}}{C_{n,p,2d}}\E_{\gamma_{p}^{n}}\left[f\right]^{2}\right)\\
 & =C_{n,p,2d}\left(\mathrm{Var}_{\gamma_{p}^{n}}(f)-\E_{\gamma_{p}^{n}}\left[f\right]^{2}\left(\frac{C_{n,p,d}^{2}}{C_{n,p,2d}}-1\right)\right)\\
 & =C_{n,p,2d}\left(\mathrm{Var}_{\gamma_{p}^{n}}(f)-\E_{\gamma_{p}^{n}}\left[f\right]^{2}\left(\frac{d^{2}}{pn}+O_{p,d}(n^{-2})\right)\right).
\end{align*}
By (\ref{eq:isotropic radius}) and (\ref{eq:ratios of Gamma}), $C_{n,p,d}>c'(d,p)$
for $n\gg_{d,p}1$. Hence, by \cite[Theorem 1]{GM22}, $C_{n,p,2d}\mathrm{Var}_{\gamma_{p}^{n}}(f)\geq c''(d,p)$
for some $c''(d,p)>0$. Taking $c(d,p)<\frac{1}{2}c''(d,p)$, for
$n\gg_{d,p}1$, we get:
\[
\frac{1}{2}c''(d,p)\geq\mathrm{Var}_{\mu_{n,p}}(f)\geq c''(d,p)-\frac{2d^{2}}{pn}\E_{\gamma_{p}^{n}}\left[f\right]^{2}.
\]
 In particular, $\E_{\gamma_{p}^{n}}\left[f\right]^{2}\geq\frac{pn}{4d^{2}}c''(d,p)$.
If further $c(d,p)<\frac{p\left(c'(d,p)\right)^{2}c''(d,p)}{4d^2}$ then for $n\gg_{d,p}1$,
\[
\E_{\mu_{n,p}}\left[f\right]^{2}\geq\left(c'(d,p)\right)^{2}\E_{\gamma_{p}^{n}}\left[f\right]^{2}\geq\frac{pn}{4d^{2}}\left(c'(d,p)\right)^{2}c''(d,p)\geq c(d,p)n.\qedhere
\]
\end{proof}
\begin{cor}
\label{cor:lower bound on variance}For every $d\in\N$ and every
$p\geq1$, and every $f\in\mathcal{P}_{d,\mathrm{unit}}(\R^{n})$,
one has:
\begin{enumerate}
\item $\mathrm{Var}_{\mu_{n,p}}(f)\geq\frac{d^{2}}{(n+2d)n}\E_{\mu_{n,p}}\left[f\right]^{2}$. 
\item $\mathrm{Var}_{\mu_{n,p}}(f)=\Omega_{p,d}(n^{-1})$. 
\end{enumerate}
\end{cor}

\begin{proof}
Recall from \eqref{eq:L^p ball from Gaussian} that $X\sim\mu_{n,p}$
distributes as $R_{n,p}\mathrm{U}^{1/n}\frac{Z}{\left\Vert Z\right\Vert _{p}},$ where
$\mathrm{U}$ has the uniform measure on $[0,1]$, and $Z\sim\gamma_{p}^{n}$
are independent. Since $\mathrm{Var}\left(\mathrm{U}^{\frac{d}{n}}\right)=\frac{nd^{2}}{(n+2d)(n+d)^{2}}$,
we obtain: 
\begin{align*}
\mathrm{Var}_{\mu_{n,p}}(f(X)) & =\mathrm{Var}\left(f(R_{n,p}\mathrm{U}^{1/n}\frac{Z}{\left\Vert Z\right\Vert _{p}})\right)=\mathrm{Var}\left(R_{n,p}^{d}\mathrm{U}^{\frac{d}{n}}f(\frac{Z}{\left\Vert Z\right\Vert _{p}})\right)\\
 & \geq\mathrm{Var}(R_{n,p}^{d}\mathrm{U}^{\frac{d}{n}})\cdot\left(\E_{\gamma_{p}^{n}}\left[f\left(\frac{Z}{\|Z\|_{p}}\right)\right]\right)^{2} \geq\frac{nd^{2}R_{n,p}^{2d}}{(n+2d)(n+d)^{2}}\cdot\left(\E_{\gamma_{p}^{n}}\left[\|Z\|_{p}^{-d}f(Z)\right]\right)^{2}.
\end{align*}
By Lemma \ref{lem:properties of integral on L_p ball}, and since
$\frac{n}{n+d}\frac{R_{n,p}^{d}\Gamma\left(\frac{n}{p}\right)}{\Gamma\left(\frac{n+d}{p}\right)}=C_{n,p,d}$,
we conclude Item (1):
\begin{align*}
\mathrm{Var}_{\mu_{n,p}}(f) & \geq\frac{nd^{2}R_{n,p}^{2d}}{(n+2d)(n+d)^{2}}\left(\frac{\Gamma\left(\frac{n}{p}\right)}{\Gamma\left(\frac{n+d}{p}\right)}\right)^{2}\cdot\E_{\gamma_{p}^{n}}\left[f\right]^{2}\\
 & =\frac{d^{2}}{(n+2d)n}\left(C_{n,p,d}\E_{\gamma_{p}^{n}}\left[f\right]\right)^{2}=\frac{d^{2}}{(n+2d)n}\E_{\mu_{n,p}}\left[f\right]^{2}.
\end{align*}
If $\mathrm{Var}_{\mu_{n,p}}(f)<\frac{c_{2}(p,d)}{n}$ for $c_{2}(p,d)>0$
small enough, then, if we combine with the above display,
\[
\E_{\mu_{n,p}}\left[f\right]^{2}\leq\frac{c_{2}(p,d)(n+2d)}{d^{2}},
\]
which contradicts Lemma \ref{lem:low variance means expectation large}.
This concludes Item (2).
\end{proof}

\subsection{On the variance spectrum and a reduction to $H_{n}$-symmetric polynomials}

\label{subsec:On-the-variance spectrum}

Each polynomial $f\in\mathcal{P}_{d}(\R^{n})$ can be written as $\sum_{I:|I|=d}a_{I}x^{I}$,
where $I\in\N^{n}$ is a multi-index, and $a_{I}\in\R$. We define
an the inner product on $\mathcal{P}_{d}(\R^{n})$ by $\langle\sum_{I}a_{I}x^{I},\sum b_{J}x^{J}\rangle:=\sum_{I}a_{I}b_{I}$,
so that $\langle f,f\rangle=\mathrm{coeff}^{2}_{d}(f)$. In particular,
the collection $\{x^{I}\}_{|I|=d}$ is an orthonormal basis for $\mathcal{P}_{d}(\R^{n})$. 
\begin{defn}
\label{def:definition of matrices}Let $\eta$ be a measure on $\R^{n}$,
let $d\in\N$ and set $M:={n+d-1 \choose d}$. Consider the $M\times M$
matrices $\mathcal{E}:=\left\{ \mathcal{E}_{I,J}\right\} _{\left|I\right|,\left|J\right|=d}$
and $\mathcal{V}=\left\{ \mathcal{V}_{I,J}\right\} _{\left|I\right|,\left|J\right|=d}$
where 
\[
\mathcal{E}_{I,J}=\mathbb{E}_{\eta}\left[x^{I+J}\right]\text{ and }\mathcal{V}_{I,J}=\mathbb{E}_{\eta}\left[x^{I+J}\right]-\mathbb{E}_{\eta}\left[x^{I}\right]\mathbb{E}_{\eta}\left[x^{J}\right].
\]
\end{defn}

Note that if $f=\sum_{I}a_{I}x^{I}$ then: 
\[
\mathbb{E}_{\eta}\left[f^{2}\right]=\sum_{I,J}a_{I}a_{J}\mathbb{E}_{\eta}\left[x^{I+J}\right]=\sum_{I,J}a_{I}a_{J}\mathcal{E}_{I,J}=\langle\mathcal{E}f,f\rangle=\frac{\langle\mathcal{E}f,f\rangle}{\langle f,f\rangle}\cdot\mathrm{coeff}_d^{2}(f),
\]
and similarly, 
\[
\mathrm{Var}_{\eta}(f)=\sum_{I,J}a_{I}a_{J}\left(\mathbb{E}_{\eta}\left[x^{I+J}\right]-\mathbb{E}_{\eta}\left[x^{I}\right]\mathbb{E}_{\eta}\left[x^{J}\right]\right)=\frac{\langle\mathcal{V}f,f\rangle}{\langle f,f\rangle}\cdot\mathrm{coeff}_d^{2}(f).
\]
In particular, we get: 
\begin{lem}
\label{lem:The-minimal-eigenvalue}The minimal eigenvalue of $\mathcal{E}$
(resp.~$\mathcal{V}$) is given by $\min_{f\in\mathcal{P}_{d,\mathrm{unit}}(\R^{n})}\mathbb{E}_{\eta}\left[f^{2}\right]$
(resp.~ $\min_{f\in\mathcal{P}_{d,\mathrm{unit}}(\R^{n})}\mathrm{Var}_{\eta}(f)$).
\end{lem}
We specialize to the case of $\eta=\mu_{n,p}$ and set $\mathcal{V}_{n,p,d}=(\mathcal{V}_{IJ})_{\left|I\right|=\left|J\right|=d}$
to be the variance matrix as in Definition \ref{def:definition of matrices}, and similarly for $\mathcal{E}_{n,p,d}$.
In \cite[Theorem 2]{GM22}, we have shown that
\begin{equation} \label{eq:l2aprioribound}
    \lambda_{\min} (\mathcal{E}_{n,p,d}) := \min_{f\in\mathcal{P}_{d,\mathrm{unit}}(\R^{n})}\mathbb{E}_{\mu_{n,p}}\left[f^{2}\right]=\Theta_{p,d}(1).
\end{equation}
So, we are left with bounding the spectrum of $\mathcal{V}_{n,p,d}$. 

In principle, $\mathcal{V}_{n,p,d}$ is a matrix of dimensions roughly $\Theta(n^d) \times \Theta(n^d)$. However, as we will show in this subsection, the symmetries of $\mu_{n,p}$ allow us to reduce the complexity of $\mathcal{V}$ to a much smaller matrix of symmetric polynomials, allowing to obtain a tractable expansion. We formalize this notion in the next definition. Later, in $\mathsection$\ref{sec:Spectrum-of-unconditional,}, we will further capitalize on these ideas and extend some of the results to other symmetric measures, beyond $\mu_{n,p}$.
\begin{defn} \label{def:Hn}
The \emph{group of signed permutations,} or the \emph{Hyperoctahedral
group}, is the semidirect product 
\[
H_{n}=(\Z/2\Z)^{n}\rtimes S_{n},
\]
where $S_{n}$ acts on $(\Z/2\Z)^{n}$ by permuting the coordinates.
Concretely, each element of $H_{n}$ is of the form $(\varepsilon,\sigma)$
for $\varepsilon=(\varepsilon_{1},...,\varepsilon_{n})\in\left\{ \pm 1\right\} ^{n}$
and $\sigma\in S_{n}$, and the multiplication is given by $(\varepsilon,\sigma)\cdot(\varepsilon',\sigma')=(\varepsilon\cdot\sigma(\varepsilon'),\sigma\sigma')$, where $\sigma(\varepsilon'):=(\varepsilon'_{\sigma^{-1}(1)},...,\varepsilon'_{\sigma^{-1}(n)})$. 
\end{defn}

\begin{lem}
\label{lem:symmetrization is an eigenvector}Let $f$ be an eigenvector
of minimal eigenvalue $\lambda_{\mathrm{min}}$ for $\mathcal{V}_{n,p,d}$.
Then the symmetrization $\widetilde{f}:=\frac{1}{n!2^{n}}\sum_{(\varepsilon,\sigma)\in H_{n}}f\circ(\varepsilon,\sigma)$
of $f$ is also an eigenvector of eigenvalue
$\lambda_{\mathrm{min}}$. 
\end{lem}

\begin{proof}
Since $\mu_{n,p}$ is invariant to permutations and reflections, $\mathrm{Var}_{\mu_{n,p}}(f\circ(\varepsilon,\sigma))=\mathrm{Var}_{\mu_{n,p}}(f)$
for every $(\varepsilon,\sigma)\in H_{n}$. Since $\mathcal{V}_{n,p,d}$
has a basis of orthonormal eigenvectors, this forces $f\circ(\varepsilon,\sigma)$
to be a $\lambda_{\mathrm{min}}$-eigenvector, and hence $\widetilde{f}$
is a (possibly $0$) eigenvector of eigenvalue $\lambda_{\mathrm{min}}$. 
\end{proof}
\begin{rem}
In $\mathsection$\ref{sec:Spectrum-of-unconditional,} we thoroughly investigate
$H_{n}$-invariant measures using the representation theory of $H_{n}$.
Lemma \ref{lem:symmetrization is an eigenvector} is simply a special
case of Lemma \ref{lem:isotypic components}. 
\end{rem}

The next proposition shows it is enough to consider $H_{n}$-symmetric
polynomials. 
\begin{prop}
\label{prop:reduction to B_n symmetric}Fix $d\in\N$ and $p\geq1$.
For each $n\in\N$, let $\lambda_{\mathrm{min},n}$ be the minimal
eigenvalue of $\mathcal{V}_{n,p,d}$. Suppose that $\underset{n\in\N}{\liminf}\lambda_{\mathrm{min},n}=0$,
and let $\left\{ n_{k}\right\} _{k=1}^{\infty}$ be a subsequence
such that $\underset{k\rightarrow\infty}{\lim}\lambda_{\mathrm{min},n_{k}}=0$.
If $\left\{ f_{n_{k}}\right\} _{n_{k}}$ is a sequence of non-zero
eigenvectors of $\mathcal{V}_{n,p,d}$ with eigenvalue $\lambda_{\mathrm{min},n_{k}}$,
then $f_{n_{k}}$ must be $H_{n_{k}}$-symmetric for $k\gg_{d,p}1$. 
\end{prop}

\begin{proof}
Suppose that $f_{n_k}$ is not $H_n$-invariant, and by assumption $\mathrm{Var}_{\mu_{n_k,p}}(f_{n_k}) \to 0$.
It then follows from Lemma \ref{lem:low variance means expectation large} that $\mathbb{E}_{\mu_{n,p}}\left[f_{n_{k}}\right]\neq0$ for $k$ large enough.
By Lemma \ref{lem:symmetrization is an eigenvector}, the symmetrization
$\widetilde{f}_{n_{k}}$ of $f_{n_{k}}$ is also an eigenvector, and furthermore $\mathbb{E}_{\mu_{n,p}}\left[\widetilde f_{n_{k}}\right] = \mathbb{E}_{\mu_{n,p}}\left[ f_{n_{k}}\right]\neq0$. However, $\widetilde{f}_{n_{k}}-f_{n_{k}}$
is an eigenvector with eigenvalue $\lambda_{\mathrm{min},n_{k}}$ 
and $\mathbb{E}_{\mu_{n,p}}\left[\widetilde{f}_{n_{k}}-f_{n_{k}}\right]=0$,
which again by Lemma \ref{lem:low variance means expectation large} is impossible when $k$ is sufficiently large. Thus $\widetilde{f}_{n_{k}}=f_{n_{k}}$
and $f_{n_{k}}$ is $H_{n_{k}}$-symmetric. 
\end{proof}
By Proposition \ref{prop:reduction to B_n symmetric}, we are motivated
to restrict to the space $\mathcal{P}_{d}^{H_{n}}(\R^{n})$ of $H_{n}$-symmetric polynomials. We now introduce a basis for this space. 
\begin{defn}
~\label{def:symmetric monomials} 
\begin{enumerate}
\item A \emph{partition} $\lambda$ of $d$, denoted $\lambda\vdash d$,
is a sequence $\lambda=(\lambda_{1},..,\lambda_{k})$ of positive integers with $\lambda_{1}\geq\ldots\geq\lambda_{k}>0$ and $\lambda_{1}+...+\lambda_{k}=d$. Denote by $\ell(\lambda)$ the number of parts
in $\lambda$. It is convenient to write $\lambda=(1^{a_{1}}\cdots d^{a_{d}})$
for the partition 
\[
(\underset{a_{d}\text{ times}}{\underbrace{d,\ldots,d}},\ldots,\underset{a_{1}\text{ times}}{\underbrace{1,\ldots,1}})\vdash d.
\]
\item For each $n\in\N$ and each partition $\lambda=(1^{a_{1}}\cdots d^{a_{d}})\vdash d$, let $m_{\lambda,n}$
be its \emph{monomial symmetric polynomial}, i.e.~the sum of all
monomials $x^{\mu}$ in $\mathcal{P}_{d}(\R^{n})$ of the same type
as $x^{\lambda}$. For example, $m_{4,n}=\sum_{i=1}^{n}x_{i}^{4}$
and $m_{(2,2),n}=\sum_{i<j}x_{i}^{2}x_{j}^{2}$. Denote $\widetilde{m}_{\lambda,n}:={n \choose \lambda}^{-\frac{1}{2}}m_{\lambda,n}$,
where 
\[
{n \choose \lambda}:=\frac{n!}{a_{1}!...a_{d}!(n-\ell(\lambda))!}\sim n^{\ell(\lambda)},
\]
so that $\mathrm{coeff}_d(\tilde{m}_{\lambda,n})=1$.
For each $\tau,\lambda\vdash d$, we denote by $\left\{ C_{\lambda,\tau}^{\nu}\right\} _{\nu}$
the integer coefficients satisfying: 
\[
m_{\lambda,n}\cdot m_{\tau,n}:=\sum_{\nu\vdash2d}C_{\lambda,\tau}^{\nu}m_{\nu,n}.
\]

\item The space $\mathcal{P}_{d}^{H_{n}}(\R^{n})$ is spanned by $\left\{ m_{\lambda,n}\right\} _{\lambda}$,
where each part $\lambda_{i}$ in $\lambda$ is even. We will denote $D := \mathrm{dim}\left(\mathcal{P}_d^{H_n}(\mathbb{R}^n)\right)$. Note that $D$ is \textbf{independent of $n$} for $n>d$.
\item Finally, for $\lambda=(\lambda_{1},..,\lambda_{k})\vdash d$, denote
by $\beta_{\lambda,p}:=\prod_{i=1}^{k}\beta_{\lambda_{i},p}$, for $\beta_{\lambda_i,p}$ as in \eqref{eq:gammamoments}.
\end{enumerate}
\end{defn}

Consider the $D\times D$ matrix $\mathcal{V}_{n,p,d}^{\mathrm{sym}}:=(\mathcal{V}_{\lambda\tau})_{\lambda,\tau\vdash d:\lambda_{i},\tau_{i}\text{ even}}$,
where 
\[
\mathcal{V}_{\lambda\tau}:=\mathrm{Cov}_{\mu_{n,p}}(\widetilde{m}_{\lambda,n},\widetilde{m}_{\tau,n}).
\]

\begin{thm}
\label{thm:Possible potential eigenvalues}Let $d\in\N$ and let $p\geq1$. Then: 
\begin{enumerate}
\item There exists a sequence of real numbers $\left\{ e_{p,d,k}\right\} _{k=-1}^{\infty}$,
such that: 
\[
C_{n,p,2d}\cdot\tr\left(\left(\mathcal{V}_{n,p,d}^{\mathrm{sym}}\right)^{-1}\right)\sim\sum_{k=-1}^{\infty}e_{p,d,k}n^{-k}.
\]
\item If $\min_{f\in\mathcal{P}_{d,\mathrm{unit}}(\R^{n})}\mathrm{Var}_{\mu_{n,p}}(f)=o_{p,d}(1),$
then for any collection of unit eigenvectors $\left\{ f_{n}\right\} _{n}$
of $\mathcal{V}_{n,p,d}$ of minimal eigenvalues, one has $\mathrm{Var}_{\mu_{n,p}}(f_{n})=\Theta_{p,d}(n^{-1})$
and $\E_{\mu_{n,p}}\left[f_{n}\right]=\Theta_{p,d}(\sqrt{n})$. 
\end{enumerate}
\end{thm}

\begin{proof}
Note that, by Proposition \ref{prop:Marginals of Euclidean ball}, for $n>2d$,
\begin{align*}
\frac{\mathcal{V}_{\lambda\tau}{n \choose \lambda}^{\frac{1}{2}}{n \choose \tau}^{\frac{1}{2}}}{C_{n,p,2d}} & =\frac{1}{C_{n,p,2d}}\left(\E_{\mu_{n,p}}\left[m_{\lambda,n}m_{\tau,n}\right]-\E_{\mu_{n,p}}\left[m_{\lambda,n}\right]\E_{\mu_{n,p}}\left[m_{\tau,n}\right]\right)\\
 & =\E_{\gamma^n_{p}}\left[m_{\lambda,n}m_{\tau,n}\right]-\frac{C_{n,p,d}^{2}}{C_{n,p,2d}}\E_{\gamma^n_{p}}\left[m_{\lambda,n}\right]\E_{\gamma^n_{p}}\left[m_{\tau,n}\right]\\
 & =\sum_{\nu\vdash2d}C_{\lambda,\tau}^{\nu}\E_{\gamma^n_{p}}\left[m_{\nu,n}\right]-\frac{C_{n,p,d}^{2}}{C_{n,p,2d}}{n \choose \lambda}{n \choose \tau}\E_{\gamma^n_{p}}\big[x^\lambda\big]\E_{\gamma^n_{p}}\big[x^\tau\big]\\
 & =\sum_{\nu\vdash2d}C_{\lambda,\tau}^{\nu}{n \choose \nu}\beta_{\nu,p}-\frac{C_{n,p,d}^{2}}{C_{n,p,2d}}{n \choose \lambda}{n \choose \tau}\beta_{\lambda,p}\beta_{\tau,p}.
\end{align*}
By Lemma \ref{lem:asymptotic expansion}, there exists $\left\{ a_{p,d,k}\right\} _{k=0}^{\infty}$
so that 
\[
\frac{C_{n,p,d}^{2}}{C_{n,p,2d}}\sim\sum_{k=0}^{\infty}a_{p,d,k}n^{-k}.
\]
In particular, there exist $\left\{ a_{p,d,\tau,\lambda,k}\right\} _{k=-r'}^{\infty}$,
and $r'=O_{d}(1)$, such that: 
\begin{equation}
\frac{\mathcal{V}_{\lambda\tau}{n \choose \lambda}^{\frac{1}{2}}{n \choose \tau}^{\frac{1}{2}}}{C_{n,p,2d}}\sim\sum_{k=-r'}^{\infty}a_{p,d,\tau,\lambda,k}n^{-k}.\label{eq:nice coefficients}
\end{equation}
Similarly, as $\mathcal{V}_{n,p,d}^{\mathrm{sym}}$ is a $D\times D$-matrix,
we can find $\left\{ c_{p,d,k}\right\} _{k=-r''}^{\infty},$ so that:
\[
\frac{\det(\mathcal{V}_{n,p,d}^{\mathrm{sym}})\prod_{\tau}{n \choose \tau}}{(C_{n,p,2d})^{D}}\sim\sum_{k=-r''}^{\infty}c_{p,d,k}n^{-k},
\]
as in (\ref{eq:nice coefficients}). Similarly, for each of the diagonal
minors $\mathcal{M}_{\lambda\lambda}$ of $\mathcal{V}_{n,p,d}^{\mathrm{sym}}$,
the term $\frac{\mathcal{M}_{\lambda\lambda}\prod_{\tau}{n \choose \tau}}{(C_{n,p,2d})^{D-1}{n \choose \lambda}}$
admits a similar asymptotic expansion. By Cramer's rule, and again by the asymptotic expansions of the terms above, we may find $\left\{ e_{p,d,k}\right\} _{k=-r}^{\infty}$, for $r=O_{d}(1)$ so that 
\begin{equation}
C_{n,p,2d}\cdot\tr\left(\left(\mathcal{V}_{n,p,d}^{\mathrm{sym}}\right)^{-1}\right)=\sum_{\lambda\vdash d:\lambda\text{ has even parts}}\frac{C_{n,p,2d}\mathcal{M}_{\lambda\lambda}}{\det(\mathcal{V}_{n,p,d}^{\mathrm{sym}})}\sim\sum_{k=-r}^{\infty}e_{p,d,k}n^{-k}.\label{eq:asymptotic expansion of minimal eigenvalue}
\end{equation}
By Corollary \ref{cor:lower bound on variance}
$\lambda_{\min}\left(\mathcal{V}_{n,p,d}^{\mathrm{sym}}\right) \geq \frac{c}{n}$, and hence $r\leq1$.

Furthermore, by Proposition \ref{prop:reduction to B_n symmetric}, $\lambda_{\min}\left(\mathcal{V}_{n,p,d}\right) =o_{p,d}(1)$ if and only if
$\lambda_{\min}\left(\mathcal{V}_{n,p,d}^{\mathrm{sym}}\right) = o_{p,d}(1)$. Looking at the asymptotic expansion in \eqref{eq:asymptotic expansion of minimal eigenvalue},
this is equivalent to $e_{p,d,-1}\neq0$. The lower bound on the $L^2$-norm in \eqref{eq:l2aprioribound} implies that there can only be a one-dimensional
space of eigenvectors of low variance, so the minimal eigenvalue is
$\Theta_{p,d}(n^{-1})$. By Lemma \ref{lem:low variance means expectation large} and by Corollary \ref{cor:lower bound on variance}(1) we have $\E_{\mu_{n,p}}\left[f\right]=\Theta_{p,d}(\sqrt{n})$, concluding Item
(2). 
\end{proof}

\subsection{Bounding the variance}

\label{subsec:Bounding-the-variance}

So far we have shown that if $\left\{ \mu_{n,p}\right\} _{n}$ do
not have a uniform lower bound on the variance spectrum, then there
must be an $H_{n}$-invariant polynomial $f$, with
$\mathrm{coeff}_{d}(f)=1$ such that $\mathrm{Var}_{\mu_{n,p}}(f)=\Theta_{p,d}(n^{-1})$
and $\E_{\mu_{n,p}}\left[f\right]=\Theta_{p,d}(\sqrt{n})$. 
\begin{defn}
For each $d\in\N,p\geq1$ and $n\in\N$, define $f_{n,p,d}\in\mathcal{P}_{d}^{H_{n}}(\R^{n})$
to be the unique $H_{n}$-symmetric polynomial satisfying:
\[
\mathrm{Cov}_{\mu_{n,p}}\left(f,f_{n,p,d}\right)=\mathrm{Cov}_{\mu_{n,p}}\left(f,\left\Vert x\right\Vert _{p}^{d}\right)\text{ for every }f\in\mathcal{P}_{d}^{H_{n}}(\R^{n}).
\]
Note that $f_{n,p,d}$ uniquely exists, as it is the orthogonal projection
of $\left\Vert x\right\Vert _{p}^{d}$ to the space $\mathcal{P}_{d}^{H_{n}}(\R^{n})$,
with respect to the covariance inner product $\mathrm{Cov}_{\mu_{n,p}}(\cdot,\cdot)$. 
\end{defn}
We first state a useful result for representing the covariance of a polynomial against $\|x\|_p^d$.
\begin{lem}
\label{lem:covariance with p_norm}Let $d\geq1$ and let $f\in\mathcal{P}_{d}(\R^{n})$.
Then for every $p\geq1$, 
\[
\mathrm{Cov}_{\mu_{n,p}}\left(f,\frac{\left\Vert x\right\Vert _{p}^{d}}{R_{n,p}^{d/2}}\right)=R_{n,p}^{d/2}\frac{d^{2}}{(n+2d)(n+d)}\E_{\mu_{n,p}}[f].
\]
\end{lem}
\begin{proof}
By Lemma \ref{lem:properties of integral on L_p ball},
\begin{align*}
\mathrm{Cov}_{\mu_{n,p}}\left(f,\frac{\left\Vert x\right\Vert _{p}^{d}}{R_{n,p}^{d/2}}\right) & =\frac{1}{R_{n,p}^{d/2}}\left(\E_{\mu_{n,p}}\left[f(x)\left\Vert x\right\Vert _{p}^{d}\right]-\E_{\mu_{n,p}}\left[f\right]\E_{\mu_{n,p}}\left[\left\Vert x\right\Vert _{p}^{d}\right]\right)\\
 & =R_{n,p}^{d/2}\left(\frac{n+d}{n+2d}-\frac{n}{n+d}\right)\E_{\mu_{n,p}}\left[f\right]=R_{n,p}^{d/2}\frac{d^{2}}{(n+2d)(n+d)}\E_{\mu_{n,p}}\left[f\right].\qedhere
\end{align*}
\end{proof}

The next lemma is the key structural result $f_{n,p,d}$, the orthogonal projection of $\|x\|_p^d$. In particular, it shows that if there is a polynomial with low variance in
$\mathcal{P}_{d}^{H_{n}}(\R^{n})$, then $f_{n,p,d}$ must also have low variance.
\begin{lem}
\label{lem:approximate eigenvector}For each $d,n\in\N$ and $p\geq1$,
let $f_{1,n},...,f_{D,n}$ be an orthogonal basis of eigenvectors of the covariance
matrix $\mathcal{V}_{n,p,d}^{\mathrm{sym}}$, with $\mathrm{coeff}_{d}(f_{i,n})=1$,
and of eigenvalues $\lambda_{i}=\mathrm{Var}_{\mu_{n,p}}(f_{i,n})$.
Then: 
\begin{enumerate}
\item We have 
\[
\mathrm{Var}_{\mu_{n,p}}(f_{n,p,d})=\frac{d^{2}}{(n+2d)(n+d)}R_{n,p}^{d}\E_{\mu_{n,p}}\left[f_{n,p,d}\right].
\]
\item We have 
\[
f_{n,p,d}=\sum_{i=1}^{D}\frac{d^{2}}{(n+2d)(n+d)}R_{n,p}^{d}\frac{\E_{\mu_{n,p}}\left[f_{i,n}\right]}{\mathrm{Var}_{\mu_{n,p}}(f_{i,n})}f_{i,n}.
\]
\item If $\lambda_{\min}=\Theta_{p,d}(n^{-1})$, then: 
\begin{enumerate}
\item $\mathrm{coeff}_{d}(f_{n,p,d})=\Theta_{p,d}(n^{\frac{d}{p}-\frac{1}{2}})$. 
\item $\E_{\mu_{n,p}}\left[f_{n,p,d}\right]=\Theta_{p,d}(n^{\frac{d}{p}})$
and $\mathrm{Var}_{\mu_{n,p}}(f_{n,p,d})=\Theta_{p,d}(n^{\frac{2d}{p}-2})$. 
\item In particular, $\mathrm{Var}_{\mu_{n,p}}\left(\frac{f_{n,p,d}}{\mathrm{coeff}_{d}(f_{n,p,d})}\right)=\Theta_{p,d}(n^{-1})$.
\end{enumerate}
\end{enumerate}
\end{lem}

\begin{proof}
Item (1) is a direct consequence of Lemma \ref{lem:covariance with p_norm},
since, by definition, 
\[
\mathrm{Var}_{\mu_{n,p}}(f_{n,p,d})=\mathrm{Cov}_{\mu_{n,p}}\left(f_{n,p,d},\|x\|_{p}^{d}\right).
\]
Note that $\frac{f_{1,n}}{\sqrt{\lambda_{1}}},...,\frac{f_{D,n}}{\sqrt{\lambda_{D}}}$
is an orthonormal basis with respect to $\mathrm{Cov}_{\mu_{n,p}}\left(\cdot,\cdot\right)$.
Item (2) now follows by a second application of Lemma \ref{lem:covariance with p_norm}:
\begin{align}
f_{n,p,d} & =\sum_{i=1}^{D}\mathrm{Cov}_{\mu_{n,p}}\left(\frac{f_{i,n}}{\sqrt{\lambda_{i}}},f_{n,p,d}\right)\frac{f_{i,n}}{\sqrt{\lambda_{i}}}=\sum_{i=1}^{D}\mathrm{Cov}_{\mu_{n,p}}\left(\frac{f_{i,n}}{\sqrt{\lambda_{i}}},\left\Vert x\right\Vert _{p}^{d}\right)\frac{f_{i,n}}{\sqrt{\lambda_{i}}}\nonumber \\
 & =\sum_{i=1}^{D}\frac{d^{2}}{(n+2d)(n+d)}R_{n,p}^{d}\frac{\E_{\mu_{n,p}}\left[f_{i,n}\right]}{\mathrm{Var}_{\mu_{n,p}}(f_{i,n})}f_{i,n}.\label{eq:expansion of f_n,p,d}
\end{align}
For Item (3), note that, by taking an expectation on both sides of
\eqref{eq:expansion of f_n,p,d},
\[
\E_{\mu_{n,p}}\left[f_{n,p,d}\right]=\sum_{i=1}^{D}\frac{d^{2}}{(n+2d)(n+d)}R_{n,p}^{d}\frac{\E_{\mu_{n,p}}\left[f_{i,n}\right]^{2}}{\mathrm{Var}_{\mu_{n,p}}(f_{i,n})}
\]
Now suppose, without loss of generality, that $\mathrm{Var}_{\mu_{n,p}}(f_{1,n})=\Theta_{p,d}(n^{-1})$.
By Theorem \ref{thm:Possible potential eigenvalues},
\[
\E_{\mu_{n,p}}\left[f_{1,n}\right]=\Theta_{p,d}(\sqrt{n}),\quad\text{ and }\quad\frac{\E_{\mu_{n,p}}\left[f_{1,n}\right]^{2}}{\mathrm{Var}_{\mu_{n,p}}(f_{1,n})}=\Theta_{p,d}(n^{2}).
\]
For $i\neq1$, by Corollary \eqref{cor:lower bound on variance}(1)
we have $\frac{\E_{\mu_{n,p}}\left[f_{i,n}\right]^{2}}{\mathrm{Var}_{\mu_{n,p}}(f_{i,n})}=O_{p,d}(n^{2})$,
and thus 
\[
\E_{\mu_{n,p}}\left[f_{n,p,d}\right]=\frac{d^{2}}{(n+2d)(n+d)}R_{n,p}^{d}\sum_{i=1}^{D}\frac{\E_{\mu_{n,p}}\left[f_{i,n}\right]^{2}}{\mathrm{Var}_{\mu_{n,p}}(f_{i,n})}=\Theta_{p,d}(n^{\frac{d}{p}}),
\]
where we have used \eqref{eq:isotropic radius}, the bound on the
isotropic radius $R_{n,p}$. In particular, by Item (1), 
\[
\mathrm{Var}_{\mu_{n,p}}(f_{n,p,d})=\frac{d^{2}}{(n+2d)(n+d)}R_{n,p}^{d}\E_{\mu_{n,p}}\left[f_{n,p,d}\right]=\Theta_{p,d}(n^{\frac{2d}{p}-2}).
\]
We are left with estimating $\mathrm{coeff}_{d}(f_{n,p,d})$. Note
that the square root of the matrix $\mathcal{V}_{n,p,d}^{\mathrm{sym}}$
is an isometry mapping the inner product induced by $\mathrm{Cov}_{\mu_{n,p}}\left(\cdot,\cdot\right)$
to the one induced by $\mathrm{coeff}_{d}$.
In particular, since $\frac{f_{1,n}}{\sqrt{\lambda_{1}}},...,\frac{f_{D,n}}{\sqrt{\lambda_{D}}}$
is an orthonormal basis with respect to $\mathrm{Cov}_{\mu_{n,p}}\left(\cdot,\cdot\right)$,
it follows that $f_{1,n},...,f_{D,n}$ is an orthonormal basis with
respect to $\mathrm{coeff}_{d}$. In particular, 
\begin{align*}
\mathrm{coeff}_{d}^{2}(f_{n,p,d}) & =\sum_{i=1}^{D}\left(\frac{d^{2}}{(n+2d)(n+d)}R_{n,p}^{d}\frac{\E_{\mu_{n,p}}\left[f_{i,n}\right]}{\mathrm{Var}_{\mu_{n,p}}(f_{i,n})}\right)^{2}\\
 & =\left(\frac{d^{2}}{(n+2d)(n+d)}R_{n,p}^{d}\right)^{2}\sum_{i=1}^{D}\frac{\E_{\mu_{n,p}}\left[f_{i,n}\right]^{2}}{\mathrm{Var}_{\mu_{n,p}}(f_{i,n})^{2}}\\
 & =\left(\frac{d^{2}}{(n+2d)(n+d)}R_{n,p}^{d}\right)^{2}\Theta_{p,d}(n^{3})=\Theta_{p,d}(n^{\frac{2d}{p}-1}),
\end{align*}
where we again used that, by Theorem \ref{thm:Possible potential eigenvalues},
$\mathrm{Var}_{\mu_{n,p}}(f_{i,n})=\Omega_{p,d}\left(n^{-1}\right)$.
This concludes the proof. 
\end{proof}

\subsection{Comparison to the $p$-Gaussian measure}
Lemma \ref{lem:approximate eigenvector} can be interpreted as follows. Suppose that the minimal eigenvalue of $\mathcal{V}_{n,p,d}^{\mathrm{sym}}$ is small, then $f_{n,p,d}$ is close to being an eigenvector with minimal eigenvalue. Since $f_{n,p,d}$ is the orthogonal projection of the norm $\|x\|_{p}^d$, there should be a corresponding statement for $\|x\|_p^d$. This is the content of the next lemma, with respect to the product measure $\gamma_p^n$
\begin{lem}
\label{lem:reduction to p_Gaussian} Suppose that $\left\{ f_{n}\right\} _{n\in\N}$
is a sequence of eigenvectors of $\mathcal{V}_{n,p,d}^{\mathrm{sym}}$
such that $f_{n}\in\mathcal{P}_{d,\mathrm{unit}}(\R^{n})$ and
$\mathrm{Var}_{\mu_{n,p}}(f_{n})=o_{p,d}(1)$. Then 
\[
\E_{\gamma_{p}^{n}}\left[\left(f_{n}-\frac{\E_{\mu_{n,p}}\left[f_{n}\right]}{\E_{\mu_{n,p}}\left[\left\Vert x\right\Vert _{p}^{d}\right]}\left\Vert x\right\Vert _{p}^{d}\right)^{2}\right]=O_{p,d}(n^{-1}).
\]
\end{lem}

\begin{proof}
For every $f\in\mathcal{P}_{d,\mathrm{unit}}(\R^{n})$,
set $g(x):=f(x)-\frac{\E_{\mu_{n,p}}\left[f\right]}{\E_{\mu_{n,p}}\left[\left\Vert x\right\Vert _{p}^{d}\right]}\left\Vert x\right\Vert _{p}^{d}$. We first show that:
\begin{align}
\E_{\gamma_{p}^{n}}\left[g^{2}\right]\leq O_{p,d}\left(\mathrm{Var}_{\mu_{n,p}}(f)\right).\label{eq:projection}
\end{align}
First, since $g(x)^{2}$ is $2d$-homogeneous, by Proposition \ref{prop:Marginals of Euclidean ball},
\[
\E_{\gamma_{p}^{n}}\left[g^{2}\right]=\frac{1}{C_{n,p,2d}}\E_{\mu_{n,p}}\left[g^{2}\right].
\]
By Lemma \ref{lem:covariance with p_norm}, $\frac{\E_{\mu_{n,p}}[f]}{\E_{\mu_{n,p}}[\left\Vert x\right\Vert _{p}^{d}]}=\frac{\mathrm{Cov}_{\mu_{n,p}}(f,\left\Vert x\right\Vert _{p}^{d})}{\mathrm{Var}_{\mu_{n,p}}(\left\Vert x\right\Vert _{p}^{d})}$.
In other words, $\frac{\E_{\mu_{n,p}}\left[f\right]}{\E_{\mu_{n,p}}\left[\left\Vert x\right\Vert _{p}^{d}\right]}\left\Vert x\right\Vert _{p}^{d}$
is the orthogonal projection of $f$ onto $\mathrm{span}(\left\Vert x\right\Vert _{p}^{d})$
with respect to the inner product $\mathrm{Cov}_{\mu_{n,p}}(\cdot,\cdot)$.
So, 
\[
\E_{\mu_{n,p}}\left[g^{2}\right]=\mathrm{Var}_{\mu_{n,p}}\left(g\right)\leq\mathrm{Var}_{\mu_{n,p}}(f).
\]
We now establish \eqref{eq:projection}, since by combining \eqref{eq:formula for C}
and \eqref{eq:isotropic radius} we can see that $\frac{1}{C_{n,p,2d}}=O_{p,d}(1).$
To finish the proof we specialize \eqref{eq:projection} to $f_{n}$.
Since, by assumption, $\mathrm{Var}_{\mu_{n,p}}(f_{n})=o_{p,d}(1)$,
Theorem \ref{thm:Possible potential eigenvalues}(2) implies $\mathrm{Var}_{\mu_{n,p}}(f_{n})=O_{p,d}(n^{-1})$,
and so we are done. 
\end{proof}
With the last lemma in mind, we would like bound the second moment
of $d$-homogeneous functions of the form $f-b\left\Vert x\right\Vert _{p}^{d}$,
with $f$ a $d$-homogeneous polynomial, with respect to the $p$-Gaussian
distribution $\gamma_{p}^{n}$. The upshot is that we have reduced the problem to a computation in a product space which has a tractable orthogonal decomposition.

Let $\left\{ q_{m}\right\} _{m=0}^{\infty}$ be the collection of
polynomials in $\R$, with $\deg(q_{i})=i$, after applying the Gram-Schmidt
algorithm to $1,x,...,x^{m},...$ with respect to the inner product
$\E_{\gamma_{p}}\left[f_{1}f_{2}\right]$.
For a multi-index $I \in \N^n$, let $q_{I}:=q_{1}^{i_{1}}...q_{n}^{i_{n}}$. Since $\gamma_p^n$ is a product measure we have that $\left\{ q_{I}\right\} _{I\in\N^{n}}$
is an orthonormal basis of $L^{2}(\gamma_{p}^{n})$. Further define the
symmetric polynomials $Q_{m}(x_{1},...,x_{n})=\frac{1}{\sqrt{n}}\sum_{i=1}^{n}q_{m}(x_{i})$,
normalized to ensure $\mathrm{Var}_{\gamma_{p}^{n}}(Q_{m})=1$. 
\begin{lem}
\label{lem:Fourier coefficients of p-norm}Let $d\in\N$ be an even
number, and let $p\geq1$ be a real number. 
\begin{enumerate}
\item Suppose that either $p>d$, or $1\leq p<d$ is not an even integer.
Then there exists $m>d$ and a constant $c_{p,d}$ such that 
\[
\E_{\gamma_{p}^{n}}\left[\left\Vert x\right\Vert _{p}^{d}Q_{m}(x)\right]^{2}>c_{p,d}\cdot n^{\frac{2d}{p}-1}
\]
\item On the other hand, if $1<p<d$ is an even integer then for every $m>p$:
\[
\E_{\gamma_{p}^{n}}\left[\left\Vert x\right\Vert _{p}^{d}Q_{m}(x)\right]^{2}=O_{p,d}(n^{\frac{2d}{p}-3}).
\]
\end{enumerate}
\end{lem}

\begin{proof}
Write $q_{m}=\sum_{j\leq m}a_{m,j}x^{j}$, and $p_{j,n}(x):=\sum_{i=1}^{n}x_{i}^{j}$,
so that 
\[
Q_{m}(x)=\frac{1}{\sqrt{n}}\sum_{j\leq m}a_{m,j}p_{j,n}(x).
\]
As a consequence of Lemma \ref{lem:properties of integral on L_p ball}, for every $j$-homogeneous polynomial $f$, we have: 
\[
\E_{\gamma_p^n}\left[\left\Vert x\right\Vert _{p}^{d}f(x)\right]=\frac{\Gamma(\frac{n+j+d}{p})}{\Gamma(\frac{n+j}{p})}\E_{\gamma_p^n}\left[f\right].
\]
Hence, by linearity of expectation, 
\begin{align*}
\E_{\gamma_{p}^{n}}\left[\left\Vert x\right\Vert _{p}^{d}Q_{m}(x)\right] & =\frac{1}{\sqrt{n}}\sum_{j\leq m}a_{m,j}\E_{\gamma_{p}^{n}}\left[\left\Vert x\right\Vert _{p}^{d}p_{j,n}(x)\right]=\frac{1}{\sqrt{n}}\sum_{j\leq m}a_{m,j}\frac{\Gamma(\frac{n+j+d}{p})}{\Gamma(\frac{n+j}{p})}\E_{\gamma_{p}^{n}}\left[p_{j,n}(x)\right]\\
 & =\sqrt{n}\sum_{j\leq m}a_{m,j}\frac{\Gamma(\frac{n+j+d}{p})}{\Gamma(\frac{n+j}{p})}\E_{\gamma^1_{p}}\left[x_1^{j}\right]=\sqrt{n}\E_{\gamma^1_{p}}\left[\sum_{j\leq m}a_{m,j}\frac{\Gamma(\frac{n+j+d}{p})}{\Gamma(\frac{n+j}{p})}x_1^{j}\right].
\end{align*}
Taking \eqref{eq:ratios of Gamma} with $u=j$ and dividing by \eqref{eq:ratios of Gamma} with $u=j+d$ yields:
\[
\frac{\Gamma(\frac{n+j+d}{p})}{\Gamma(\frac{n+j}{p})}=\left(\frac{n}{p}\right)^{\frac{d}{p}}\left(\frac{1+\frac{j(p-j)}{2pn}+O_{p,m}(n^{-2})}{1+\frac{(j+d)(p-j-d)}{2pn}+O_{p,m}(n^{-2})}\right)=\left(\frac{n}{p}\right)^{\frac{d}{p}}\left(1+\frac{d(2j+d-p)}{2pn}+O_{p,m}(n^{-2})\right).
\]
Hence, for $n$ large enough, we obtain
\begin{align*}
\E_{\gamma_{p}^{n}}\left[\left\Vert x\right\Vert _{p}^{d}Q_{m}(x)\right] & =\frac{n^{\frac{d}{p}+\frac{1}{2}}}{p^{\frac{d}{p}}}\E_{\gamma_{p}^{1}}\left[\sum_{j\leq m}a_{m,j}x_{1}^{j}\right]+\frac{n^{\frac{d}{p}-\frac{1}{2}}}{p^{\frac{d}{p}}}\E_{\gamma_{p}^{1}}\left[\sum_{j\leq m}a_{m,j}\frac{d(2j+d-p)}{2p}x_{1}^{j}\right]+O_{p,m}(n^{\frac{d}{p}-\frac{3}{2}})\\
 & =\frac{n^{\frac{d}{p}-\frac{1}{2}}}{p^{\frac{d}{p}}}\frac{d}{p}\E_{\gamma_{p}^{1}}\left[\sum_{j\leq m}a_{m,j}jx_{1}^{j}\right]+O_{p,m}(n^{\frac{d}{p}-\frac{3}{2}}),
\end{align*}
where we have used that $\E_{\gamma_{p}^{1}}\left[\sum_{j\leq m}a_{m,j}x_{1}^{j}\right]=\E_{\gamma_{p}^{1}}\left[q_{m}(x_{1})\right]=0$,
since $q_{m}$ is orthogonal to $q_{0}\equiv1.$ Further, note that
$(q_{m})'(x_{1})=\sum\limits _{j\leq m}ja_{m,j}x_{1}^{j-1}$. Hence,
we integrate by parts,
\begin{align}
 & \E_{\gamma_{p}^{1}}\left[\sum_{j\leq m}a_{m,j}jx_{1}^{j}\right]=\E_{\gamma_{p}^{1}}\left[x_{1}(q_{m})'(x_{1})\right]=\frac{1}{\frac{2}{p}\Gamma(\frac{1}{p})}\int_{\R}x_{1}(q_{m})'(x_{1})e^{-\left|x_{1}\right|^{p}}dx_{1}\nonumber \\
= & -\frac{1}{\frac{2}{p}\Gamma(\frac{1}{p})}\int_{-\infty}^{\infty}q_{m}(x_{1})\cdot\left(e^{-\left|x_{1}\right|^{p}}-p\left|x_{1}\right|^{p}e^{-\left|x_{1}\right|^{p}}\right)dx_{1}=p\E_{\gamma_{p}^{1}}\left[q_{m}(x_{1})\left|x_{1}\right|^{p}\right],\label{eq:integration by parts}
\end{align}
where we again used the fact that $q_{m}$ has vanishing expectation.
So, 
\begin{align*}
\E_{\gamma_{p}^{n}}\left[\left\Vert x\right\Vert _{p}^{d}Q_{m}(x)\right]^{2} & =\left(\frac{n^{\frac{d}{p}-\frac{1}{2}}}{p^{\frac{d}{p}}}d\E_{\gamma_{p}^{1}}\left[q_{m}(x_{1})\left|x_{1}\right|^{p}\right]+O_{p,m}(n^{\frac{d}{p}-\frac{3}{2}})\right)^{2}\\
 & =n^{\frac{2d}{p}-1}\frac{d^{2}}{p^{\frac{2d}{p}}}\left(\E_{\gamma_{p}^{1}}\left[q_{m}(x_{1})\left|x_{1}\right|^{p}\right]\right)^{2}+O_{p,m}(n^{\frac{2d}{p}-2}).
\end{align*}
For the first item, if $p$ is not an even integer, then $\left|x_{1}\right|^{p}$
is not a polynomial of degree $\leq d$. Since $\{q_{m}\}_{m=0}^{d}$
span the space of degree-$d$ polynomials, there exists $m>d$ such
that $\left|\E_{\gamma_{p}^{1}}\left[q_{m}(x_{1})\left|x_{1}\right|^{p}\right]\right|=\alpha_{p,d}>0$.
If $p$ is an even integer larger than $d$, we can take $m=p$, and
get $\left|\E_{\gamma_{p}^{1}}\left[q_{p}(x_{1})\left|x_{1}\right|^{p}\right]\right|=\beta_{p}>0$.
Combining with the above estimate we now get 
\[
\E_{\gamma_{p}^{n}}\left[\left\Vert x\right\Vert _{p}^{d}Q_{m}(x)\right]^{2}>n^{\frac{2d}{p}-1}\frac{d^{2}}{2p^{\frac{2d}{p}}}\min(\alpha_{p,d},\beta_{p}),
\]
as required. For the second item, suppose that $1<p<d$ is an even
integer. In this case when $m>p$, $q_{m}$ is orthogonal to the degree
$p$ polynomial $x_{1}^{p}=\left|x_{1}\right|^{p}$. I.e. $\E_{\gamma_{p}^{1}}\left[q_{m}(x_{1})\left|x_{1}\right|^{p}\right]=0$.
Hence, for every $m>p$,
\[
\E_{\gamma_{p}^{n}}\left[\left\Vert x\right\Vert _{p}^{d}Q_{m}(x)\right]^{2}=O_{p,d}\left(n^{\frac{d}{p}-\frac{3}{2}}\right)^{2}=O_{p,d}\left(n^{\frac{2d}{p}-3}\right).\qedhere
\]
\end{proof}

As a direct consequence of Lemma \ref{lem:Fourier coefficients of p-norm} we prove the following result about possible second moments of functions of the form $f-b\left\Vert x\right\Vert _{p}^{d}$, with $f$ a $d$-homogeneous polynomial.
\begin{prop}
\label{prop:generalization of spectrum of p_Gaussian}Let $d\in\N$ and suppose that either $p>d$, or $1\leq p<d$ is not an even integer. Let $b\in\R$, let $f=\sum_{I}a_{I}x^{I}\in\mathcal{P}_{d}(\R^{n})$, and let $g=f+b\left\Vert x\right\Vert _{p}^{d}$.
Then: 
\[
\E_{\gamma_{p}^{n}}\left[g^{2}\right]\geq c_{p,d}b^{2}\cdot n^{\frac{2d}{p}-1}.
\]
\end{prop}

\begin{proof}
Since $f\in\mathcal{P}_{d}(\R^{n})$, $\E_{\gamma_{p}^{n}}\left[f\cdot q_{I}\right]=0$
for every $I$ with $\left|I\right|>d$. Hence, by Lemma \ref{lem:Fourier coefficients of p-norm}(1), there exists some $m > d$ such that
\begin{align*}
\E_{\gamma_{p}^{n}}\left[g^{2}\right] & \geq\sum_{I:\left|I\right|>d}\E_{\gamma_{p}^{n}}\left[g\cdot q_{I}\right]^{2}=b^{2}\sum_{I:\left|I\right|>d}\E_{\gamma_{p}^{n}}\left[\left\Vert x\right\Vert _{p}^{d}\cdot q_{I}\right]^{2}\\
 & \geq b^{2}\E_{\gamma_{p}^{n}}\left[\left\Vert x\right\Vert _{p}^{d}\cdot Q_{m}(x)\right]^{2}>c_{p,d}b^{2}\cdot n^{\frac{2d}{p}-1}.\qedhere
\end{align*}
\end{proof}
Proposition \ref{prop:generalization of spectrum of p_Gaussian} together with Lemma \ref{lem:reduction to p_Gaussian} are already enough to show that in some cases there cannot exist polynomials of low variance. Lemma \ref{lem:reduction to p_Gaussian} suggests the existence of a function of the form $f + b\|x\|_p^d$ with a small second moment, while Proposition \ref{prop:generalization of spectrum of p_Gaussian} asserts that the second moment of \emph{every} such function must be large, allowing to reach a contradiction.

However, in order to prove Theorem \ref{thm:Main theorem L^p} in full generality, we will need to deal separately
with the case when $p<d$ which is an even integer, where Lemma \ref{lem:Fourier coefficients of p-norm}(2) predicts a different behavior. In a similar fashion to the definition above, we denote by $\widetilde{f}_{n,p,d}$
the projection of $\left\Vert x\right\Vert _{p}^{d}$ to $\mathcal{P}_{\leq d}(\R^{n})$, inside $L^2(\gamma_p^n).$
The product structure of $\gamma_p^n$ allows a more precise representation of $\widetilde{f}_{n,p,d}$: If $\left\Vert x\right\Vert _{p}^{d}=\sum_{I}b_{I}q_{I}(x)$,
then 
\begin{equation}
\widetilde{f}_{n,p,d}(x):=\sum_{I:\left|I\right|\leq d}b_{I}q_{I}(x)=\sum_{I:\left|I\right|\leq d}\E_{\gamma_{p}^{n}}\left[\left\Vert x\right\Vert _{p}^{d}q_{I}(x)\right]\cdot q_{I}(x).\label{eq:def of projection to polynomials}
\end{equation}
We have the following analog of Lemma \ref{lem:approximate eigenvector}. 
\begin{lem}
\label{lem:spectrum for large degrees}Let $p\in2\N$ and $p<d$.
Then: 
\[
\mathrm{coeff}_{d}(\widetilde{f}_{n,p,d})^{2}=O_{p,d}\left(n^{\frac{d}{p}}\right).
\]
\end{lem}

\begin{proof}
By (\ref{eq:def of projection to polynomials}), we have 
\[
\mathrm{coeff}_{d}(\widetilde{f}_{n,p,d})=\mathrm{coeff}_{d}\left(\sum_{I:\left|I\right|=d}\E_{\gamma_{p}^{n}}\left[\left\Vert x\right\Vert _{p}^{d}q_{I}(x)\right]\cdot q_{I}(x)\right).
\]
Note there exist constants $\{c_{I,J}\}_{|I|=d,|J|<d}$ such
that $q_{I}=c_{I,I}x^{I}+\sum_{J:\left|J<d\right|}c_{J,I}x^{J}$,
and for every $|I|=d$, $c_{I,I}<\alpha_{p,d}$, for some constant
$\alpha_{p,d}$. Thus, 
\begin{equation}
\mathrm{coeff}_{d}(\widetilde{f}_{n,p,d})^{2}\leq\alpha_{p,d}^{2}\sum_{I:\left|I\right|=d}\E_{\gamma_{p}^{n}}\left[\left\Vert x\right\Vert _{p}^{d}q_{I}(x)\right]^{2}.\label{eq:coefftildebound}
\end{equation}
We estimate $\E_{\gamma_{p}^{n}}\left[\left\Vert x\right\Vert _{p}^{d}q_{I}(x)\right]$
using an argument similar to that in Lemma \ref{lem:Fourier coefficients of p-norm}.
Write $x^{I}=x_{1}^{i_{1}}...x_{n}^{i_{n}}$, and $q_{m}(x_{u})=\sum_{j\leq m}a_{m,j}x_{u}^{j}$,
so that 
\[
q_{I}(x_{1},...,x_{n})=\sum_{0\leq j_{1}\leq i_{1}}...\sum_{0\leq j_{n}\leq i_{n}}a_{i_{1},j_{1}}...a_{i_{n},j_{n}}x_{1}^{j_{1}}\cdot...\cdot x_{n}^{j_{n}}.
\]
Now, by Lemma \ref{lem:properties of integral on L_p ball},
\begin{align*}
 & \E_{\gamma_{p}^{n}}\left[\left\Vert x\right\Vert _{p}^{d}q_{I}(x_{1},...,x_{n})\right]=\sum_{0\leq j_{1}\leq i_{1}}...\sum_{0\leq j_{n}\leq i_{n}}a_{i_{1},j_{1}}...a_{i_{n},j_{n}}\E_{\gamma_{p}^{n}}\left[\left\Vert x\right\Vert _{p}^{d}x_{1}^{j_{1}}\cdot...\cdot x_{n}^{j_{n}}\right]\\
= & \sum_{0\leq j_{1}\leq i_{1}}...\sum_{0\leq j_{n}\leq i_{n}}a_{i_{1},j_{1}}...a_{i_{n},j_{n}}\frac{\Gamma(\frac{n+j_{1}+...+j_{n}+d}{p})}{\Gamma(\frac{n+j_{1}+...+j_{n}}{p})}\E_{\gamma_{p}^{n}}\left[x_{1}^{j_{1}}\cdot...\cdot x_{n}^{j_{n}}\right].
\end{align*}
Using (\ref{eq:ratios of gamma}) with $t=(\frac{j_{1}+...+j_{n}+d}{p})$,
$s=(\frac{j_{1}+...+j_{n}}{p})$, and $x=\frac{n}{p}$, yields: 
\[
\frac{\Gamma(\frac{n+j_{1}+...+j_{n}+d}{p})}{\Gamma(\frac{n+j_{1}+...+j_{n}}{p})}\sim\sum_{k=0}^{\infty}\frac{(-1)^{k}B_{k}^{(\frac{d}{p}+1)}(\frac{j_{1}+...+j_{n}+d}{p})\cdot(\frac{d}{p})_{k}}{k!}\left(\frac{n}{p}\right)^{\frac{d}{p}-k},
\]
where $B_{k}^{(\alpha)}(y)$ is the generalized Bernoulli polynomial,
which in particular is a polynomial of degree $k$ in $(\frac{j_{1}+...+j_{n}+d}{p})$.
Hence 
\begin{align*}
 & \E_{\gamma_{p}^{n}}\left[\left\Vert x\right\Vert _{p}^{d}q_{I}(x_{1},...,x_{n})\right]\\
\sim & \sum_{k=0}^{\infty}\left(\frac{n}{p}\right)^{\frac{d}{p}-k}\left((-1)^{k}\frac{(\frac{d}{p})_{k}}{k!}\right)\cdot\sum_{0\leq j_{1}\leq i_{1}}...\sum_{0\leq j_{n}\leq i_{n}}a_{i_{1},j_{1}}...a_{i_{n},j_{n}}B_{k}^{(\frac{d}{p}+1)}\left(\frac{j_{1}+...+j_{n}+d}{p}\right)\E_{\gamma_{p}^{n}}\left[x_{1}^{j_{1}}\dots x_{n}^{j_{n}}\right].
\end{align*}
We claim that for every $k<\sum_{e=1}^{n}\left\lceil \frac{i_{e}}{p}\right\rceil $,
we have 
\begin{equation}
\sum_{0\leq j_{1}\leq i_{1}}...\sum_{0\leq j_{n}\leq i_{n}}a_{i_{1},j_{1}}...a_{i_{n},j_{n}}B_{k}^{(\frac{d}{p}+1)}\left(\frac{j_{1}+...+j_{n}+d}{p}\right)\E_{\gamma_{p}^{n}}\left[x_{1}^{j_{1}}\dots x_{n}^{j_{n}}\right]=0.\label{eq:lowdegvanish}
\end{equation}
Indeed, as $B_{k}$ is a degree $k$ polynomial, it is enough to show
\[
\sum_{0\leq j_{1}\leq i_{1}}...\sum_{0\leq j_{n}\leq i_{n}}a_{i_{1},j_{1}}...a_{i_{n},j_{n}}j_{1}^{u_{1}}...j_{n}^{u_{n}}\E_{\gamma_{p}^{n}}\left[x_{1}^{j_{1}}\dots x_{n}^{j_{n}}\right]=0,\text{ whenever }\sum_{e=1}^{n}u_{e}<\sum_{e=1}^{n}\left\lceil \frac{i_{e}}{p}\right\rceil .
\]
But, if $\sum_{e=1}^{n}u_{e}<\sum_{e=1}^{n}\left\lceil \frac{i_{e}}{p}\right\rceil $
then necessarily $u_{e}<\left\lceil \frac{i_{e}}{p}\right\rceil $
for some $e\in[n]$, so it will suffice to show, for any $m\in\N$,
\[
\E_{\gamma_{p}^{1}}\left[\sum_{j\leq m}a_{m,j}j^{u}x_{1}^{j}\right]=0\text{ whenever }u<\left\lceil \frac{m}{p}\right\rceil .\tag{\ensuremath{\star}}
\]
Integrating by parts, in a similar way to \eqref{eq:integration by parts},
we get 
\begin{align*}
\E_{\gamma_{p}^{1}}\left[\sum_{j\leq m}a_{m,j}j^{u}x_{1}^{j}\right] & =\E_{\gamma_{p}^{1}}\left[\left(x_{1}\frac{d}{dx_{1}}\right)^{u}q_{m}(x_{1})\right]=\frac{1}{\frac{2}{p}\Gamma(\frac{1}{p})}\int_{\R}\frac{d}{dx_{1}}\left(\left(x_{1}\frac{d}{dx_{1}}\right)^{u-1}q_{m}(x)\right)\cdot x_{1}e^{-\left|x_{1}\right|^{p}}dx_{1}\\
 & =-\frac{1}{\frac{2}{p}\Gamma(\frac{1}{p})}\int_{\R}\left(\left(x_{1}\frac{d}{dx_{1}}\right)^{u-1}q_{m}(x_{1})\right)\cdot\left(e^{-\left|x_{1}\right|^{p}}-p\left|x_{1}\right|^{p}e^{-\left|x_{1}\right|^{p}}\right)dx_{1}\\
 & =p\E_{\gamma_{p}^{1}}\left[\left(\left(x_{1}\frac{d}{dx_{1}}\right)^{u-1}q_{m}(x_{1})\right)x_{1}^{p}\right]-\E_{\gamma_{p}^{1}}\left[\left(x_{1}\frac{d}{dx_{1}}\right)^{u-1}q_{m}(x_{1})\right],
\end{align*}
where we used the fact that $\left|x_{1}\right|^{p}=x_{1}^{p}$ since $p$
is an even integer. Continuing by induction, $\E_{\gamma_{p}^{1}}\left[\sum_{j\leq m}a_{i,j}j^{u}x_{1}^{j}\right]$
reduces to a linear combination of terms of the form
\[
\left\{ \E_{\gamma_{p}^{1}}\left[q_{m}(x_{1})x_{1}^{pu'}\right]\right\} _{u'\leq u}.
\]
If $u<\left\lceil \frac{m}{p}\right\rceil $ then $up<m$, and so,
because $q_{m}$ is orthogonal to polynomials of degree smaller than
$m$, $\E_{\gamma_{p}^{1}}\left[q_{m}(x_{1})x_{1}^{pu'}\right]=0$ for every
$u'\leq u$. As a consequence, $\E_{\gamma_{p}^{1}}\left[\sum_{j\leq m}a_{m,j}j^{u}x_{1}^{j}\right]=0$,
and $(\star)$ is established. In particular, \eqref{eq:lowdegvanish}
holds as well.

Returning to the asymptotic expansion of $\E_{\gamma_{p}^{n}}\left[\left\Vert x\right\Vert _{p}^{d}q_{I}(x_{1},...,x_{n})\right]$,
we deduce 
\begin{equation}
\E_{\gamma_{p}^{n}}\left[\left\Vert x\right\Vert _{p}^{d}q_{I}(x_{1},...,x_{n})\right]=O_{p,d}\left(n^{\frac{d}{p}-\sum_{e=1}^{n}\left\lceil \frac{i_{e}}{p}\right\rceil }\right).\label{eq:contribution of q_I}
\end{equation}
Therefore, summing over all $I$ with $\left|I\right|=d$, and plugging
\eqref{eq:contribution of q_I} into \eqref{eq:coefftildebound} we
arrive at, 
\begin{equation}
\mathrm{coeff}_{d}(\widetilde{f}_{n,p,d})^{2}=O_{p,d}\left(\sum\limits _{|I|=d}n^{\frac{2d}{p}-2\left(\left\lceil \frac{i_{1}}{p}\right\rceil +...+\left\lceil \frac{i_{n}}{p}\right\rceil \right)}\right).\label{eq:coefftildefinalbound}
\end{equation}
For a multi-index $I$, set $|\mathrm{support}(I)|=\#\{i_{e}>0\}$,
corresponding to the number of different variables appearing in $q_{I}$.
For any $I$ with $|\mathrm{support}(I)|=\ell$,
\[
\sum_{e=1}^{n}\left\lceil \frac{i_{e}}{p}\right\rceil =\sum_{e:i_{e}>0}\left\lceil \frac{i_{e}}{p}\right\rceil \geq\max\left\{ \sum_{e:i_{e}>0}\frac{i_{e}}{p},\ell\right\} =\max\left\{ \frac{d}{p},\ell\right\} .
\]
We therefore get: 
\[
\sum\limits _{|I|=d,\ |\mathrm{support}(I)|=\ell}n^{\frac{2d}{p}-2\left(\left\lceil \frac{i_{1}}{p}\right\rceil +...+\left\lceil \frac{i_{n}}{p}\right\rceil \right)}=O_{d}(n^{\ell}\cdot n^{\frac{2d}{p}-\max\left\{ \frac{2d}{p},2\ell\right\} })=O_{d}(n^{\min\left\{ \ell,\frac{2d}{p}-\ell\right\} })=O_{d}(n^{\frac{d}{p}}).
\]
Applying this into \eqref{eq:coefftildefinalbound} we can finally
conclude $\mathrm{coeff}_{d}(\widetilde{f}_{n,p,d})^{2}=O_{p,d}\left(n^{\frac{d}{p}}\right)$.
\end{proof}
\subsection{Proof of the  main results for $\mu_{n,p}$}
We are now ready to prove Theorem \ref{thm:Main theorem L^p}. 

\begin{proof}[Proof of Theorem \ref{thm:Main theorem L^p}]
We begin with the case of $p=d$ and $p\in2\N$, which requires
us to show that the minimal variance is $\Theta_{p,d}(n^{-1}).$ This
case follows from a direct computation for the polynomial $f(x)=\frac{1}{\sqrt{n}}\|x\|_{p}^{p}$,
as in \cite[Example 3]{GM22}, showing that $\mathrm{Var}_{\mu_{n,p}}(f)=\Theta_{p}(n^{-1})$,
combining with the general characterization of minimal eigenvalues
in Theorem \ref{thm:Possible potential eigenvalues}(2) finishes this
case.

Thus, from now on we assume that $p\neq d$, and assume towards a
contradiction that 
\[
\underset{f\in\mathcal{P}_{d,\mathrm{unit}}(\R^{n})}{\min}\mathrm{Var}_{\mu_{n,p}}(f)=o_{p,d}(1),
\]
Again invoking Theorem \ref{thm:Possible potential eigenvalues}(2)
this assumption implies 
\[
\underset{f\in\mathcal{P}_{d,\mathrm{unit}}(\R^{n})}{\min}\mathrm{Var}_{\mu_{n,p}}(f)=\Theta_{p,d}(n^{-1}).
\]
Furthermore, by Lemma \ref{lem:approximate eigenvector}, if $f_{n,p,d}$
is as in the lemma, we also know that 
\[
\mathrm{Var}_{\mu_{n,p}}(f_{n,p,d})=\Theta_{p,d}(n^{\frac{2d}{p}-2}),\ \E_{\mu_{n,p}}\left[f_{n,p,d}\right]=\Theta_{p,d}(n^{\frac{d}{p}}),\ \mathrm{coeff}_{d}(f_{n,p,d})=\Theta_{p,d}(n^{\frac{d}{p}-\frac{1}{2}}),
\]
as well as $\mathrm{Var}_{\mu_{n,p}}(\frac{f_{n,p,d}}{\mathrm{coeff}_{d}(f_{n,p,d})})=\Theta_{p,d}(n^{-1}).$

Passing to $p$-Gaussian spaces, by Lemma \ref{lem:reduction to p_Gaussian},
and normalizing by $\mathrm{coeff}_{d}(f_{n,p,d})$, 
\begin{equation}
\E_{\gamma_{p}^{n}}\left[\left(f_{n,p,d}-\frac{\E_{\mu_{n,p}}\left[f_{n,p,d}\right]}{\E_{\mu_{n,p}}\left[\left\Vert x\right\Vert _{p}^{d}\right]}\left\Vert x\right\Vert _{p}^{d}\right)^{2}\right]=O_{p,d}(n^{\frac{2d}{p}-2}).\label{eq:1.8}
\end{equation}
Suppose for now that either $p>d$, or $1\leq p<d$ is not an even
integer. Applying Proposition \ref{prop:generalization of spectrum of p_Gaussian}
for $g=f_{n,p,d}-\frac{\E_{\mu_{n,p}}[f_{n,p,d}]}{\E_{\mu_{n,p}}\left[\left\Vert x\right\Vert _{p}^{d}\right]}\left\Vert x\right\Vert _{p}^{d}$,
and using Lemma \ref{lem:properties of integral on L_p ball} to estimate
$\E_{\mu_{n,p}}\left[\left\Vert x\right\Vert _{p}^{d}\right]$, we have 
\[
\E_{\gamma_{p}^{n}}\left[\left(f_{n,p,d}-\frac{\E_{\mu_{n,p}}\left[f_{n,p,d}\right]}{\E_{\mu_{n,p}}\left[\left\Vert x\right\Vert _{p}^{d}\right]}\left\Vert x\right\Vert _{p}^{d}\right)^{2}\right]\geq c_{p,d}\left(\frac{\E_{\mu_{n,p}}\left[f_{n,p,d}\right]}{\E_{\mu_{n,p}}\left[\left\Vert x\right\Vert _{p}^{d}\right]}\right)^{2}\cdot n^{\frac{2d}{p}-1}=\Theta_{p,d}(n^{\frac{2d}{p}-1}),
\]
which is a contradiction to \eqref{eq:1.8}. This proves the cases
where $p>d$ and where $1\leq p\leq d$ is not an even integer.

Finally, assume that $p<d$ is an even integer. Then by \eqref{eq:def of projection to polynomials},
we can replace $\|x\|_{p}^{d}$ by its orthogonal projection $\widetilde{f}_{n,p,d}$,
\[
\E_{\gamma_{p}^{n}}\left[\left(f_{n,p,d}-\frac{\E_{\mu_{n,p}}\left[f_{n,p,d}\right]}{\E_{\mu_{n,p}}\left[\left\Vert x\right\Vert _{p}^{d}\right]}\left\Vert x\right\Vert _{p}^{d}\right)^{2}\right]\geq\E_{\gamma_{p}^{n}}\left[\left(f_{n,p,d}-\frac{\E_{\mu_{n,p}}\left[f_{n,p,d}\right]}{\E_{\mu_{n,p}}\left[\left\Vert x\right\Vert _{p}^{d}\right]}\widetilde{f}_{n,p,d}\right)^{2}\right].
\]
On the other hand, by Lemma \ref{lem:spectrum for large degrees},
\[
\mathrm{coeff}_{d}\left(\frac{\E_{\mu_{n,p}}\left[f_{n,p,d}\right]}{\E_{\mu_{n,p}}\left[\left\Vert x\right\Vert _{p}^{d}\right]}\widetilde{f}_{n,p,d}\right)^{2}=\Theta_{p,d}(1)\cdot\mathrm{coeff}_{d}\left(\widetilde{f}_{n,p,d}\right)^{2}=O_{p,d}(n^{\frac{d}{p}}).
\]
However, recall that by our assumption (towards a contradiction),
$\mathrm{coeff}_{d}\left(f_{n,p,d}\right)^{2}=\Theta_{p,d}(n^{\frac{2d}{p}-1}).$
Now, if $d>p$, then $n^{\frac{2d}{p}-1}\gg n^{\frac{d}{p}}$ and
so, 
\[
\mathrm{coeff}_{d}\left(f_{n,p,d}-\frac{\E_{\mu_{n,p}}\left[f_{n,p,d}\right]}{\E_{\mu_{n,p}}\left[\left\Vert x\right\Vert _{p}^{d}\right]}\widetilde{f}_{n,p,d}\right)^{2}=\Theta_{p,d}(n^{\frac{2d}{p}-1}).
\]
Hence, by \cite[Theorem 1]{GM22}, which gives a general lower bound
for second moments of polynomials, we deduce: 
\[
\E_{\gamma_{p}^{n}}\left[\left(f_{n,p,d}-\frac{\E_{\mu_{n,p}}\left[f_{n,p,d}\right]}{\E_{\mu_{n,p}}\left[\left\Vert x\right\Vert _{p}^{d}\right]}\left\Vert x\right\Vert _{p}^{d}\right)^{2}\right]\geq\Theta_{p,d}(1)\cdot\mathrm{coeff}_{d}\left(f_{n,p,d}-\frac{\E_{\mu_{n,p}}\left[f_{n,p,d}\right]}{\E_{\mu_{n,p}}\left[\left\Vert x\right\Vert _{p}^{d}\right]}\widetilde{f}_{n,p,d}\right)^{2}\geq\Theta(n^{\frac{2d}{p}-1}),
\]
in contradiction to (\ref{eq:1.8}). This concludes the proof.
\end{proof}
Theorem \ref{thm:lpball} is now a consequence of Theorem \ref{thm:Main theorem L^p}.
\begin{proof}[Proof of Theorem \ref{thm:lpball}]
For each fixed $n_{0}\in\N$, $\inf\limits_{f\in\mathcal{P}_{d,\mathrm{unit}}\left(\R^{n_{0}}\right)}\mathrm{Var}_{\mu_{n_{0},p}}(f)>v_{p,d,n_{0}}>0$.
Hence, we may assume that $n\gg_{p,d}1$. When $p\notin2\N$
or when $p\neq d$, by Theorem \ref{thm:Main theorem L^p}, $\inf\limits_{f\in\mathcal{P}_{d,\mathrm{unit}}\left(\R^{n}\right)}\mathrm{Var}_{\mu_{n,p}}(f)=\Omega_{p,d}(1)$,
which is the required result.

Suppose, on the other hand, that $p=d\in2\N$. Let $\left\{ g_{i,n}\right\} _{i=1}^{M}$
be a basis of unit eigenvectors of $\mathcal{V}_{n,p,d}$ with eigenvalues
$\lambda_{M}\geq...\geq\lambda_{1}$. Theorem \ref{thm:Main theorem L^p}
implies that $\lambda_{1}=\Theta_{p}(n^{-1})$. Note that the second moment matrix $\mathcal{E}_{n,p,d}$ differs from the variance matrix $\mathcal{V}_{n,p,d}$ by a matrix of rank one, namely the expectation. Thus, as a consequence of the lower bound on second moment \eqref{eq:l2aprioribound}, there can be at
most a single one-dimensional eigenspace of eigenvalue $o_{p}(1)$, so $\lambda_{M}\geq...\geq\lambda_{2}=\Omega_{p}(1)$.
Since $\mathrm{Var}_{\mu_{n,p}}(\frac{1}{\sqrt{n}}\left\Vert x\right\Vert _{p}^{p})=\Theta_{p}(n^{-1})$,
it follows that
\[
\left\langle \frac{1}{\sqrt{n}}\left\Vert x\right\Vert _{p}^{p},g_{1,n}\right\rangle =1-o_{p}(1),
\]
Therefore if $\left|\langle f,\frac{1}{\sqrt{n}}\left\Vert x\right\Vert _{p}^{p}\rangle\right|\leq c_{p}<1$
then for $n\gg_{p}1$,
\[
\left|\langle f,g_{1,n}\rangle\right|=\left|\langle f,\frac{1}{\sqrt{n}}\left\Vert x\right\Vert _{p}^{p}\rangle+\langle f,\frac{1}{\sqrt{n}}\left\Vert x\right\Vert _{p}^{p}-g_{1,n}\rangle\right|\leq\left|\langle f,\frac{1}{\sqrt{n}}\left\Vert x\right\Vert _{p}^{p}\rangle\right|+o_{p}(1)\leq c'_{p}<1,
\]
and in particular, $\mathrm{coeff}_{d}(f-\langle f,g_{1,n}\rangle g_{1,n})\geq\sqrt{1-c'{}_{p}^{2}}$.
We conclude the proof as follows,
\[
\mathrm{Var}_{\mu_{n,p}}(f)\geq\mathrm{Var}_{\mu_{n,p}}(f-\langle f,g_{1,n}\rangle g_{1,n})\geq\lambda_{2}(1-c'{}_{p}^{2})>c''_{p}.\qedhere
\]
\end{proof}
\section{Preliminaries in representation theory}\label{sec:Preliminaries-in-representation}

Our goal in this section is to describe the irreducible representations
of the symmetric group $S_{n}$ and the group $H_{n}$, from Definition \ref{def:Hn}. We start by recollecting basic facts about
the representation theory of finite groups. For a detailed exposition,
see \cite{Ste12} and \cite{FH91}.

\subsection{Basic facts in representation theory }

Let $G$ be a finite group, and let $F$ be a field. A \emph{representation}
$(V,\pi)$ of $G$ is a homomorphism $\pi:G\rightarrow\mathrm{GL}(V)$,
where $V$ is an $F$-vector space. A \emph{$G$-morphism} $T:(V,\pi)\rightarrow(W,\tau)$
is a linear map $T:V\rightarrow W$ such that $T\circ\pi(g)=\tau(g)\circ T$
for every $g\in G$. If $T$ is a linear isomorphism, then we say
that $(V,\pi)$ and $(W,\tau)$ are \emph{isomorphic representations}.

A $G$-invariant subspace $W$ in $V$ is called a \emph{subrepresentation}.
A representation $(V,\pi)$ of $G$ is called \emph{irreducible},
if any subrepresentation $W\leq V$ is either $\{0\}$ or $V$. We
write $\mathrm{Rep}_{F}(G)$ (resp.~$\mathrm{Irr}_{F}(G)$) for the
set of equivalence classes of (resp.~irreducible) representations of
$G$. For every $(V_{1},\pi_{1}),(V_{2},\pi_{2})\in\Rep_{F}(G)$,
we denote by $\mathrm{Hom}_{G}(\pi_{1},\pi_{2})$ the space of $G$-morphisms
$T:(V_{1},\pi_{1})\rightarrow(V_{2},\pi_{2})$ and denote by $\langle\pi_{1},\pi_{2}\rangle_{G}:=\dim\mathrm{Hom}_{G}(\pi_{1},\pi_{2})$.

Given $(V,\pi)\in\mathrm{Rep}_{F}(G)$, one can define its \emph{character}
$\chi_{\pi}:G\rightarrow F$ by $\chi_{\pi}(g):=\tr(\pi(g))$. 
\begin{example}[{see e.g.~\cite[Proposition 3.12]{FH91}}]
\label{exa:standard rep}Let $G=S_{n}$. Then $S_{n}$ acts on the
set $X:=\{1,...,n\}$ by permutations. If $F[X]=F^{|X|}$ is the vector
space of functions on $X$, this induces a representation $\pi:S_{n}\rightarrow\mathrm{GL}(F[X])$
by 
\begin{equation}
\left(\pi(\sigma).f\right)(i):=f(\sigma^{-1}(i)).\label{eq:standard representation}
\end{equation}
The representation $(F[X],\pi)$ is not irreducible. Indeed, $F[X]=F_{\mathrm{const}}[X]\oplus F_{0}[X]$,
where $F_{\mathrm{const}}[X]$ is the subspace of constant functions
on $X$, and $F_{0}[X]$ is the subspace of $f\in F[X]$ with $\sum_{i=1}^{n}f(i)=0$.
Denoting $\pi_{\mathrm{std}}:=\pi|_{F_{0}[X]}$ for $\pi|_{F_{0}[X]}:S_{n}\rightarrow\mathrm{GL}(F_{0}[X])$,
then $(F_{0}[X],\pi_{\mathrm{std}})$ is called the \emph{standard
(irreducible) representation} of $S_{n}$. Its character is $\chi_{\mathrm{std}}(\sigma):=\mathrm{fix}(\sigma)-1$,
where $\mathrm{fix}(\sigma)$ is the number of elements in $\{1,...,n\}$
fixed by $\sigma$. 
\end{example}

\begin{defn}
~\label{def:isotypic}$(V_{\pi},\pi)\in\Rep(G)$ is called\emph{
isotypic} if it is a direct sum of isomorphic irreducible representations,
that is, if $V_{\pi}=V_{\tau}^{\oplus a_{\tau}}$ for $(V_{\tau},\tau)\in\Irr(G)$
and $a_{\tau}\in\N$. In this case we say that $V_{\pi}$ is \emph{$\tau$-isotypic}. 
\end{defn}

We now specialize to the case that $F\in\left\{ \R,\C\right\} $. 
\begin{thm}
\label{thm:Maschkes}Let $G$ be a finite group, let $F\in\left\{ \R,\C\right\} $
and let $(V,\pi)\in\Rep_{F}(G)$ be a finite-dimensional representation.
Then: 
\begin{enumerate}
\item There exists a \textbf{unique} set of subrepresentations $W_{\tau}$
of $V$, indexed by $\tau\in\Irr_{F}(G)$, such that $W_{\tau}$ is
$\tau$-isotypic and: 
\[
V=\bigoplus_{\tau\in\Irr_{F}(G)}W_{\tau}.
\]
The subspace $W_{\tau}$ is called the \textbf{$\tau$-isotypic component}
of $V$. 
\item For every subrepresentation $V'\subseteq V$, we have $V'=\bigoplus_{\tau\in\Irr_{F}(G)}W_{\tau}\cap V'$. 
\item If $T\in\mathrm{Hom}_{G}(\pi,\pi)$, then $T(W_{\tau})\subseteq W_{\tau}$,
with equality if $T$ is an isomorphism. 
\end{enumerate}
\end{thm}

The following construction will be useful to describe the representations
of $S_{n}$ and $B_{n}$. 
\begin{defn}
\label{def:induced representations}Let $G$ be a finite group, let
$H\leq G$ be a subgroup, and let $(V_{\tau},\tau)\in\Rep_{F}(H)$
with character $\chi_{\tau}$. We define the \emph{induction of $\pi$
from $H$ to $G$}, as $\left(\Ind_{H}^{G}(V_{\tau}),\Ind_{H}^{G}(\tau)\right)\in\Rep_{F}(G)$,
where 
\begin{equation}
\Ind_{H}^{G}(V_{\tau}):=\left\{ f:G\rightarrow V_{\tau}:f(gh)=\tau(h^{-1}).f(g)\right\} ,\label{eq:definition of induction}
\end{equation}
and where $G$ acts by left translations, $\left(\Ind_{H}^{G}(\tau)(x)\right).f(g)=f(x^{-1}g)$.
The character of $\Ind_{H}^{G}(\tau)$ is denoted by $\Ind_{H}^{G}(\chi_{\tau})$.
Moreover, $\dim\Ind_{H}^{G}(V_{\tau})=\dim V_{\tau}\cdot\left|G:H\right|$. 
\end{defn}

\begin{thm}[{{Frobenius reciprocity, \cite[Corollary 3.20]{FH91}}}]
\label{thm:Frobenius reciprocity}Let $G$ be a finite group and
$H<G$ a subgroup. Then for every $\pi\in\Rep_{F}(G)$ and $\tau\in\Rep_{F}(H)$,
\begin{equation}
\langle\Ind_{H}^{G}\tau,\pi\rangle_{G}=\langle\pi,\Ind_{H}^{G}\tau\rangle_{G}=\langle\pi|_{H},\tau\rangle_{H}=\langle\tau,\pi|_{H}\rangle_{H}.\label{eq:Frobenous reciprocity}
\end{equation}
\end{thm}

\subsection{Representation theory of $S_{n}$ and $H_{n}$}

In general, finite groups might have different representation theory
over $\R$ and $\C$. However, for a finite Coxeter group like $S_{n}$
and $H_{n}$, every representation can be realized over $\R$ (such
groups are called \emph{totally orthogonal}, see e.g.~\cite{Ben71}).
In particular, if $G\in\left\{ S_{n},H_{n}\right\} $, then $\Irr_{\R}(G)$
is naturally identified with $\Irr_{\C}(G)$, and every $(V,\pi)\in\Rep_{\R}(G)$
is uniquely determined by its character $\chi_{\pi}$. Thus, we may
drop the subscripts, and keep the notation $\mathrm{Rep}(G)$ and
$\Irr(G)$ both for representations and for characters. We now turn
to describe $\Irr(G)$. 

A partition $\lambda=(\lambda_{1},..,\lambda_{\ell})\vdash n$, with
$\lambda_{\ell}>0$ is graphically encoded by a \emph{Young diagram},
which is a finite collection of boxes arranged in $\ell$ left-justified
rows, where the $j$-th row has $\lambda_{j}$ boxes. Each $\lambda\vdash n$
defines a unique \emph{$\lambda$-tableau} $T_{\lambda}$, by placing
the numbers $1,...,\lambda_{1}$ in the first row of the Young diagram
for $\lambda$, ordered from left to right, and similarly, the numbers
$\lambda_{1}+1,...,\lambda_1 + \lambda_{2}$ are placed in the second row, etc. 
\begin{defn}
\label{def:Young subgroups}Let $\lambda=(\lambda_{1},...,\lambda_{\ell})\vdash n$
with $\lambda_{\ell}>0$. The \emph{row stabilizer} $R_{\lambda}$
of $\lambda$ is the subgroup of $S_{n}$ preserving the rows of $T_{\lambda}$,
i.e.~
\[
R_{\lambda}:=S_{\left\{ 1,...,\lambda_{1}\right\} }\times S_{\left\{ \lambda_{1}+1,...,\lambda_1+\lambda_{2}\right\} }...\times S_{\left\{ n-\lambda_{\ell}+1,...,n\right\} }\leq S_{n}.
\]
The \emph{column stabilizer} $C_{\lambda}$ of $\lambda$ is the subgroup
of $S_{n}$ preserving the columns of $T_{\lambda}$. 
\end{defn}

\begin{example}
If $\lambda=(5,4,2,1)\vdash12$, then $T_{\lambda}$ is given by\[
\begin{ytableau}
1 & 2 & 3 & 4 & 5 \\
6 & 7 & 8 & 9 \\
10 & 11 \\
12 \\
\end{ytableau}
\]Moreover, the row and column stabilizers of $\lambda$ are given by:
\begin{align*}
R_{\lambda} & \simeq S_{\left\{ 1,2,3,4,5\right\} }\times S_{\left\{ 6,7,8,9\right\} }\times S_{\left\{ 10,11\right\} }\times S_{\left\{ 12\right\} }\leq S_{12}.\\
C_{\lambda} & \simeq S_{\left\{ 1,6,10,12\right\} }\times S_{\left\{ 2,7,11\right\} }\times S_{\left\{ 3,8\right\} }\times S_{\left\{ 4,9\right\} }\times S_{\{5\}}\leq S_{12}.
\end{align*}
\end{example}

We denote by $\mathrm{triv}:S_{n}\rightarrow\R$ the trivial representation
of $S_{n}$, and by $\sgn:S_{n}\rightarrow\R$ the sign representation,
taking $\sigma$ to its sign as a permutation.
\begin{fact}[{{cf.~\cite[Section 10.2]{Ste12}}}]
To each $\lambda\vdash n$, one can associate the \textbf{Specht
representation} $(V_{\lambda},\pi_{\lambda})\in\Irr(S_{n})$, with
character $\chi_{\lambda}$. It is the unique subrepresentation of
both $\Ind_{C_{\lambda}}^{S_{n}}\sgn$ and $\Ind_{R_{\lambda}}^{S_{n}}\mathrm{triv}$. This construction gives rise to all irreducible characters of $S_{n}$. I.e.:
\[
\Irr(S_{n})=\left\{ \chi_{\lambda}\right\} _{\lambda\vdash n}.
\]
\end{fact}

\begin{example}
~\label{exa:few examples} 
\begin{enumerate}
\item The partition $(n)\vdash n$ gives $\chi_{(n)}=\mathrm{triv}$. Indeed,
$R_{(n)}=S_{n}$ so $\Ind_{R_{(n)}}^{S_{n}}\mathrm{triv}=\mathrm{triv}$. 
\item The partition $(1,1,...,1)\vdash n$ gives $\chi_{(1,...,1)}=\sgn$.
Indeed, $\Ind_{C_{(1,...,1)}}^{S_{n}}\sgn=\sgn$. 
\item The partition $(n-1,1)\vdash n$ induces the character $\chi_{(n-1,1)}$
of the standard representation $(\R_{0}[X],\pi_{\mathrm{std}})$ defined
in Example \ref{exa:standard rep}. Note that $R_{(n-1,1)}=S_{n-1}<S_{n}$
and $\R[X]\simeq\Ind_{S_{n-1}}^{S_{n}}\mathrm{triv}$. 
\end{enumerate}
\end{example}

The representation theory of $H_{n}$ is discussed in \cite{JK81},
\cite[Chapter 1, Appendix B]{Mac95} and \cite{GK78}. It turns out that
\[
\Irr(H_{n})=\left\{ \chi_{(\mu,\nu)}:\mu,\nu\text{ are partitions, }\left|\mu\right|+\left|\nu\right|=n\right\} .
\]
Explicitly, writing $H_{j,n-j}:=\left(S_{j}\times S_{n-j}\right)\rtimes(\Z/2\Z)^{n}$,
$\left|\mu\right|=j$ and $\left|\nu\right|=n-j$, 
\begin{equation}
\chi_{(\mu,\nu)}:=\mathrm{Ind}_{H_{j,n-j}}^{H_{n}}((\chi_{\mu}\otimes\chi_{\nu})\otimes\varepsilon_{\{1,...,j\}}),\label{eq:(reps of B_n)}
\end{equation}
where $\varepsilon_{j}:(\Z/2\Z)^{n}\rightarrow\R$ is defined by $\varepsilon_{j}(a_{1},...,a_{n})=\begin{cases}
1 & \text{if }a_{j}=0\\
-1 & \text{if }a_{j}=1
\end{cases}$, and $\varepsilon_{J}:=\prod_{j\in J}\varepsilon_{j}$. Any $\chi_{\lambda}\in\Irr(S_{n})$
can be pulled back to $\chi_{(0,\lambda)}\in\Irr(H_{n})$ via the
projection map $H_{n}\rightarrow S_{n}$. We denote by $\left(V_{(\mu,\nu)},\pi_{(\mu,\nu)}\right)$
the representation with character $\chi_{(\mu,\nu)}$. 
\begin{example}[{{\cite[p.8]{GK78}}}]
\label{exa:standard action}Let $\mu=(1)$ and $\nu=(n-1)$. Then
\[
\chi_{((1),(n-1))}=\mathrm{Ind}_{H_{1,n-1}}^{B_{n}}(\left(\mathrm{triv}_{S_{1}}\otimes\mathrm{triv}_{S_{n-1}}\right)\otimes\varepsilon_{\left\{ 1\right\} })\in\Irr(H_{n}),
\]
is the character of the standard representation of $H_{n}$ on $\text{\ensuremath{\R}}^{n}$,
defined by $(\varepsilon,\sigma).(a_{1},...,a_{n}):=(\varepsilon_{1}a_{\sigma(1)},...,\varepsilon_{n}a_{\sigma(n)})$. 
\end{example}

\section{Spectrum of $H_{n}$-invariant log-concave isotropic measures }\label{sec:Spectrum-of-unconditional,}

A measure $\eta$ on $\R^{n}$ is called \emph{unconditional} if $\eta$
is invariant under $x_{i}\rightarrow-x_{i}$ for all $i$. $\eta$
is called \emph{$S_{n}$-invariant} if $\sigma_{*}\eta=\eta$ for
all coordinate permutations $\sigma:\R^{n}\rightarrow\R^{n}$. If
$\eta$ is both unconditional and $S_{n}$-invariant, then it is \emph{$H_{n}$-invariant.}
Given $\eta$ as above, we recall the matrices $\mathcal{E}:=\left\{ \mathcal{E}_{I,J}\right\} _{\left|I\right|,\left|J\right|=d}$
and $\mathcal{V}=\left\{ \mathcal{V}_{I,J}\right\} _{\left|I\right|,\left|J\right|=d}$
from $\mathsection$\ref{subsec:On-the-variance spectrum}, where
\[
\mathcal{E}_{I,J}=\mathbb{E}_{\eta}\left[x^{I+J}\right]\text{ and }\mathcal{V}_{I,J}=\mathbb{E}_{\eta}\left[x^{I+J}\right]-\mathbb{E}_{\eta}\left[x^{I}\right]\mathbb{E}_{\eta}\left[x^{J}\right].
\]
In this section we compute the spectrum of $\mathcal{E}$ and $\mathcal{V}$
for $H_{n}$-invariant measures. Note that $\mathcal{P}_{d}(\R^{n})\simeq\mathrm{Sym}^{d}(\R^{n})$
is a representation of $H_{n}$ of dimension $\binom{n+d-1}{d}$.
By Theorem \ref{thm:Maschkes} we can write: 
\[
\mathcal{P}_{d}(\R^{n})=\bigoplus_{(\mu,\nu):\left|\mu\right|+\left|\nu\right|=n}W_{(\mu,\nu)},
\]
where $W_{(\mu,\nu)}$ is the $\chi_{(\mu,\nu)}$-isotypic component
of $\mathcal{P}_{d}(\R^{n})$, for $\chi_{(\mu,\nu)}\in\Irr(H_{n})$. 
\begin{lem}
\label{lem:isotypic components}Let $\mathcal{C}\in\left\{ \mathcal{E},\mathcal{V}\right\} $.
Then: 
\begin{enumerate}
\item The linear map $\mathcal{C}:\mathcal{P}_{d}(\R^{n})\rightarrow\mathcal{P}_{d}(\R^{n})$
is a morphism of $H_{n}$-representations. 
\item For every ordered pair $(\mu,\nu)$ of partitions summing to $n$,
$\mathcal{C}(W_{(\mu,\nu)})\subseteq W_{(\mu,\nu)}$. 
\item For any $\lambda\geq0$, the $\mathcal{C}$-eigenspace $V_{\lambda}\subseteq\mathcal{P}_{d}(\R^{n})$
with eigenvalue $\lambda$ is a direct sum of irreducible representations
of $H_{n}$. 
\end{enumerate}
\end{lem}

\begin{proof}
Since both $\eta$ and the inner product $\langle\,,\,\rangle$ in
$\mathcal{P}_{d}(\R^{n})$ are $H_{n}$-invariant, we get: 
\begin{align*}
\langle\mathcal{E}\circ\left(\varepsilon,\sigma\right).x^{I},\left(\varepsilon,\sigma\right).x^{J}\rangle & =\mathbb{E}_{\eta}(\left(\varepsilon,\sigma\right).x^{I}\cdot\left(\varepsilon,\sigma\right).x^{J})=\mathbb{E}_{\eta}(x^{I+J})\\
 & =\langle\mathcal{E}x^{I},x^{J}\rangle=\langle\left(\varepsilon,\sigma\right)\circ\mathcal{E}.x^{I},\left(\varepsilon,\sigma\right).x^{J}\rangle,
\end{align*}
for every $I,J\subseteq n$ with $\left|I\right|=\left|J\right|=d$.
Since $\langle\,,\,\rangle$ is non-degenerate, we get that $\mathcal{E}\circ\left(\varepsilon,\sigma\right)=\left(\varepsilon,\sigma\right)\circ\mathcal{E}$
for every $\left(\varepsilon,\sigma\right)\in H_{n}$. A similar computation
shows $\mathcal{V}\circ\left(\varepsilon,\sigma\right)=\left(\varepsilon,\sigma\right)\circ\mathcal{V}$,
so both $\mathcal{E}$ and $\mathcal{V}$ are $H_{n}$-morphisms, and Item (1) follows.

If $f\in V_{\lambda}$ then $\mathcal{C}$.$\left(\left(\varepsilon,\sigma\right).f\right)=\left(\varepsilon,\sigma\right)\circ\mathcal{C}.f=\lambda\left(\varepsilon,\sigma\right).f$
for every $\left(\varepsilon,\sigma\right)\in H_{n}$ and therefore
$V_{\lambda}$ is an $H_{n}$-subrepresentation. Items (2) and (3)
now follows from Theorem \ref{thm:Maschkes}. 
\end{proof}
The following corollary will be the key statement for the next subsection. 
\begin{cor}
\label{cor:Multiplicity free}Let $\mathcal{C}\in\left\{ \mathcal{E},\mathcal{V}\right\} $,
let $G_{n}\in\{S_{n},H_{n}\}$ and let $\mathcal{P}_{d}(\R^{n})=\bigoplus_{\tau\in\Irr(G_{n})}W_{\tau}$
be its decomposition into $G_{n}$-isotypic components. Then: 
\begin{enumerate}
\item Each $W_{\tau}$ is a direct sum of $\mathcal{C}$-eigenspaces. 
\item In particular, if every $W_{\tau}$ is irreducible (i.e.~$\mathcal{P}_{d}(\R^{n})$
is multiplicity free as a $G_{n}$-representation), then each $W_{\tau}$
is a $\mathcal{C}$-eigenspace. 
\end{enumerate}
\end{cor}

\begin{proof}
Item (1) follows since $\mathcal{C}(W_{\tau})\subseteq W_{\tau}$ and $\mathcal{C}|_{W_{\tau}}$
is diagonalizable. Item (2) follows by Lemma \ref{lem:isotypic components}(3), since for every $\lambda\geq0$, $(W_{\tau})_{\lambda}:=V_{\lambda}\cap W_{\tau}$
is a subrepresentation of $W_{\tau}$ and hence either $(W_{\tau})_{\lambda}=0$
or $W_{\tau}=(W_{\tau})_{\lambda}$. 
\end{proof}
\begin{example}
The isotypic component of the trivial representation of $S_{n}$ (resp.~$H_{n}$)
in $\mathcal{P}_{d}(\R^{n})$ is the space $\mathcal{P}_{d}^{S_{n}}(\R^{n})$
of degree $d$-homogeneous symmetric polynomials (resp.~the space
$\mathcal{P}_{d}^{H_{n}}(\R^{n})$ of degree $d$-homogeneous symmetric
polynomials in the variables $x_{1}^{2},...,x_{n}^{2}$). 
\end{example}

\subsection{Computation of the spectrum for low degree homogeneous polynomials}

We now turn to compute the spectrum of $\mathcal{E}$ and $\mathcal{V}$
on $\mathcal{P}_{d}(\R^{n})$ for $d\leq3$, for an $H_{n}$-invariant
measure $\eta$ on $\R^{n}$. If $\eta$ is log-concave and isotropic,
we further show that the eigenvalues of $\mathcal{E}$ are bounded
from below by an absolute constant for $d\leq3$, and that the eigenvalues
of $\mathcal{V}$ are bounded, expect possibly when $d=2$ if the
polynomial $\frac{1}{\sqrt{n}}\|x\|_2^2$ has a low variance.

\subsubsection{The linear case}

In this case, the spectrum of $\mathcal{E}$ and $\mathcal{V}$ can
be easily computed without representation theory (Corollary \ref{cor:linear polynomials}
below), but the analysis here serves as a simple model case for higher
degrees. 
\begin{lem}
\label{lem:multiplicity free}Let $n\geq2$. The space $\mathcal{P}_{1}(\R^{n})$
of linear polynomials on $\R^{n}$: 
\begin{enumerate}
\item Decomposes into direct sum of $S_{n}$-subrepresentations: 
\[
V_{1}:=\left\langle \frac{1}{\sqrt{n}}\sum_{i=1}^{n}x_{i}\right\rangle \text{ and }V_{2}:=\left\langle\left\{ \frac{1}{\sqrt{2}}(x_{i}-x_{i+1})\right\} _{i=1}^{n-1}\right\rangle.
\]
where $V_{1}$ has the trivial character $\chi_{(n)}$ and $V_{2}$
has character $\chi_{(n-1,1)}$. 
\item Is irreducible as an $H_{n}$-representation, with character $\chi_{((1),(n-1))}$. 
\end{enumerate}
\end{lem}

\begin{proof}
Item (1) is as known fact in the representation theory of $S_{n}$
(see Examples \ref{exa:standard rep} and \ref{exa:few examples}).
Item (2) follows from Example \ref{exa:standard action}. 
\end{proof}
\begin{defn}
\label{def:moments of invariant measures}Let $\eta$ be an $S_{n}$-invariant
measure on $\R^{n}$. For each $m\in\N$ and $k_{1},...,k_{m}\in\N$, denote by 
\[
\alpha_{(k_{1},...k_{m})}:=\E_{\eta}\left[x_{1}^{k_{1}}...x_{m}^{k_{m}}\right].
\]
\end{defn}
When $\eta$ is $S_n$-invariant it is clear that both $\mathcal{E}$ and $\mathcal{V}$ are of the form $\alpha_1I_{n} + \alpha_2J$, where $J$ is the all-ones matrix. Therefore, there are at most $2$ eigenvalues. If $\eta$ is also unconditional, the situation simplifies, and both $\mathcal{E}$ and $\mathcal{V}$ are scalar matrices. This simple fact is captured in the next result. Below, we will use the same strategy to handle higher degrees, for which the situation is not as clear.
\begin{prop}
\label{prop:spectrum for quadratic case }Let $\eta$ be an $S_{n}$-invariant
measure on $\R^{n}$. Let $\mathcal{C}\in\left\{ \mathcal{E},\mathcal{V}\right\} $.
Then: 
\begin{enumerate}
\item The $\mathcal{C}$-eigenspaces in $\mathcal{P}_{1}(\R^{n})$ are $V_{1}$
and $V_{2}$, with eigenvalues: 
\begin{align*}
(\lambda_{1},\lambda_{2}) & =(\alpha_{2}+(n-1)\alpha_{(1,1)},\alpha_{2}-\alpha_{(1,1)})\text{ \,for\,}\ \mathcal{E}.\\
(\delta_{1},\delta_{2}) & =(\alpha_{2}+(n-1)\alpha_{(1,1)}-n\alpha_{1}^{2},\alpha_{2}-\alpha_{(1,1)})\text{ \,for\,} \ \mathcal{V}.
\end{align*}
\item If $\eta$ is $H_{n}$-invariant, then $\mathcal{P}_{1}(\R^{n})$
is a $\mathcal{C}$-eigenspace, with eigenvalue $\lambda=\delta=\alpha_{2}$. 
\end{enumerate}
\end{prop}

\begin{proof}
By Lemma \ref{lem:multiplicity free}, $\mathcal{P}_{1}(\R^{n})$
is a multiplicity free representation for both $S_{n}$ and $H_{n}$,
and hence by Corollary \ref{cor:Multiplicity free}, each of its isotypic
components is a $\mathcal{C}$-eigenspace. We can choose representatives:
$\frac{1}{\sqrt{n}}\sum_{i=1}^{n}x_{i}$ in $V_{1}$ and $\frac{1}{\sqrt{2}}(x_{1}-x_{2})\in V_{2}$. 
Note that: 
\[
\lambda_{1}=\mathbb{E}\left[\frac{1}{\sqrt{n}}\sum_{i=1}^{n}x_{i}\right]^{2}=\frac{1}{n}\mathbb{E}\left[\sum_{i=1}x_{i}^{2}+\sum_{i\neq j}x_{i}x_{j}\right]=\alpha_{2}+(n-1)\alpha_{(1,1)},
\]
and 
\[
\lambda_{2}=\delta_{2}=\mathbb{E}\left[\frac{1}{\sqrt{2}}(x_{1}-x_{2})\right]^{2}=\frac{1}{2}\mathbb{E}\left[x_{1}^{2}-2x_{1}x_{2}+x_{2}^{2}\right]=\alpha_{2}-\alpha_{(1,1)}.
\]
Moreover,
\[
\delta_{1}=\mathbb{E}\left[\frac{1}{\sqrt{n}}\sum_{i=1}^{n}x_{i}\right]^{2}-\mathbb{E}\left[\frac{1}{\sqrt{n}}\sum_{i=1}^{n}x_{i}\right]^{2}=\alpha_{2}+(n-1)\alpha_{(1,1)}-n\alpha_{1}^{2}.\qedhere
\]
\end{proof}
\begin{cor}
\label{cor:linear polynomials}Let $\eta$ be an $H_{n}$-invariant,
isotropic measure on $\R^{n}$, and let $f\in\mathcal{P}_{1}(\R^{n})$
be a linear polynomial with $\mathrm{coeff}(f)=1$. Then 
\[
\E_{\eta}\left[f^{2}\right]=\mathrm{Var}_{\eta}(f^{2})=1.
\]
\end{cor}

\begin{proof}
For isotropic measures, $\alpha_{2}=1$. 
\end{proof}

\subsubsection{The quadratic case}

We now assume that $d=2$. As before, we write $\lambda_{1},...,\lambda_{k}$
for the eigenvalues of $\mathcal{E}$ and $\delta_{1},...,\delta_{k}$
for the eigenvalues of $\mathcal{V}$. 
\begin{prop} \label{prop: quadratic case}
Let $n\geq2$, let $\eta$ be an $H_{n}$-invariant measure, and let $\mathcal{C}\in\left\{ \mathcal{E},\mathcal{V}\right\}$.
Then the $\mathcal{C}$-eigenspaces in $\mathcal{P}_{2}(\R^{n})$
are as follows: 
\begin{align*}
V_{1} & :=\left\langle\frac{1}{\sqrt{n}}\sum_{i=1}^{n}x_{i}^{2}\right\rangle\text{ with eigenvalues }\lambda_{1}=\alpha_{4}+(n-1)\alpha_{(2,2)}\text{ \,and\, }\delta_{1}=\alpha_{4}+(n-1)\alpha_{(2,2)}-n\alpha_{2}^{2}.\\
V_{2} & :=\left\langle\left\{ \frac{1}{\sqrt{2}}(x_{i}^{2}-x_{i+1}^{2})\right\} _{i=1}^{n-1}\right\rangle\text{ with eigenvalue }\lambda_{2}=\delta_{2}=\alpha_{4}-\alpha_{(2,2)}.\\
V_{3} & :=\left\langle\{x_{i}x_{j}\}_{i<j}\right\rangle\text{ with eigenvalues }\lambda_{3}=\delta_{3}=\alpha_{(2,2)}.
\end{align*}
\end{prop}

\begin{proof}
By Corollary \ref{cor:Multiplicity free}, it will be enough to show that
$\mathcal{P}_{2}(\R^{n})$ is a multiplicity free $H_{n}$-representation,
and deduce that each isotypic component is a $\mathcal{C}$-eigenspace.
Firstly, the space of $H_{n}$-symmetric polynomials $V_{1}:=\left\langle \sum_{i=1}^{n}x_{i}^{2}\right\rangle \subseteq\mathcal{P}_{2}(\R^{n})$
is isomorphic to the trivial representation of $H_{n}$ whose character
is $\chi_{((0),(n))}$. Secondly, the action of $H_{n}$ on the subspace
$V_{2}$ factors through the action of $S_{n}$. Since $V_{2}$ is isomorphic to the standard representation
of $S_{n}$ (see Example \ref{exa:standard rep}), it is irreducible
as an $H_{n}$-representation, with character $\chi_{(0,(n-1,1))}$.

Finally, note that $V_{3}$ is irreducible as an $H_{n}$-representation.
Indeed, the vector $x_{1}x_{2}\in V_{3}$ is $\left(H_{2,n-2},\mathrm{triv}\otimes\varepsilon_{\{1,2\}}\right)$-equivariant.
By Frobenius reciprocity (Theorem \ref{thm:Frobenius reciprocity}),
there is a non-zero $H_{n}$-morphism $\mathrm{Ind}_{H_{2,n-2}}^{H_{n}}(\mathrm{triv}\otimes\varepsilon_{\{1,2\}})\rightarrow V_{3}$.
By (\ref{eq:(reps of B_n)}), $\mathrm{Ind}_{H_{2,n-2}}^{H_{n}}(\mathrm{triv}\otimes\varepsilon_{\{1,2\}})$
is irreducible, with character $\chi_{((2),(n-2))}$ and of dimension
$\left|H_{n}:H_{2,n-2}\right|={n \choose 2}$. Since $V_{3}$ is of
the same dimension, we deduce that $V_{3}$ is irreducible, with character
$\chi_{((2),(n-2))}$. Since 
\[
\dim V_{1}+\dim V_{2}+\dim V_{3}=1+(n-1)+{n \choose 2}={n+1 \choose 2}=\dim\mathcal{P}_{2}(\R^{n}),
\]
we deduce that $\mathcal{P}_{2}(\R^{n})$ is multiplicity free, and
decomposes into isotypic components by: 
\[
\mathcal{P}_{2}(\R^{n})=V_{1}\oplus V_{2}\oplus V_{3}.
\]
By Corollary \ref{cor:Multiplicity free}, each $V_{i}$ is a $\mathcal{C}$-eigenspace.
The eigenvalues are: 
\begin{align*}
\lambda_{1} & =\E_{\eta}\left[\left(\frac{1}{\sqrt{n}}\sum_{i=1}^{n}x_{i}^{2}\right)^{2}\right]=\frac{1}{n}\E_{\eta}\left[\sum_{i=1}^{n}x_{i}^{4}+\sum_{i\neq j}x_{i}^{2}x_{j}^{2}\right]=\alpha_{4}+(n-1)\alpha_{(2,2)}.\\
\delta_{1} & =\E_{\eta}\left[\left(\frac{1}{\sqrt{n}}\sum_{i=1}^{n}x_{i}^{2}\right)^{2}\right]-\E_{\eta}\left[\frac{1}{\sqrt{n}}\sum_{i=1}^{n}x_{i}^{2}\right]^{2}=\alpha_{4}+(n-1)\alpha_{(2,2)}-n\alpha_{2}^{2}.
\end{align*}
\begin{align*}
\lambda_{2} & =\delta_{2}=\E_{\eta}\left[\left(\frac{1}{\sqrt{2}}\left(x_{1}^{2}-x_{2}^{2}\right)\right)^{2}\right]=\frac{1}{2}\mathbb{E}_\eta\left[x_{1}^{4}+x_{2}^{4}-2x_{1}^{2}x_{2}^{2}\right]=\alpha_{4}-\alpha_{(2,2)}.\\
\lambda_{3} & =\delta_{3}=\mathbb{E}_\eta\left[x_{1}^{2}x_{2}^{2}\right]=\alpha_{(2,2)}.\qedhere
\end{align*}
\end{proof}
\begin{rem}
The quadratic case is interesting from two reasons: 
\begin{enumerate}
\item Unlike the linear case, the isotropic condition is not sufficient
for obtaining a lower bound on the spectrum of $\mathcal{E}$. Indeed,
one can take $\mu$ supported on $\pm1$ with probability $\frac{1}{2}$.
Then $\eta=\mu^{\otimes n}$ is unconditional, permutation-invariant,
isotropic measure. But $\alpha_{4}-\alpha_{(2,2)}=0$, and indeed
the polynomial $\frac{1}{\sqrt{2}}\left(x_{1}^{2}-x_{2}^{2}\right)$
is constant, and thus has zero variance. The proposition below shows
that the log-concavity is a sufficient condition for a lower bound
on the spectrum of $\mathcal{E}$. 
\item Even in the log-concave, isotropic case, it is not always the case
that the spectrum of $\mathcal{V}$ is bounded. The only potentially
small eigenvalue is 
\[
\delta_{1}=\alpha_{4}+(n-1)\alpha_{(2,2)}-n\alpha_{2}^{2}
\]
for the polynomial $\frac{1}{\sqrt{n}}\sum_{i=1}^{n}x_{i}^{2}$, see for example Theorem \ref{thm:lpball}.
\end{enumerate}
\end{rem}

We prove the following proposition in Appendix \ref{sec:Proof-of-Propositions}. 
\begin{prop}
\label{prop:Lower bound on second moment}Let $\eta$ be an $H_{n}$-invariant,
log-concave, isotropic measure on $\R^{n}$, and let $f\in\mathcal{P}_{2}(\R^{n})$
with $\mathrm{coeff}_2(f)=1$. Then 
\begin{equation}
\E_{\eta}[f^{2}]\geq\frac{4}{5}+O(n^{-1}),\label{eq:asymptotic lower bound}
\end{equation}
and 
\begin{equation}
\mathrm{Var}_{\eta}[f^{2}]\geq\min\left(\mathrm{Var}_\eta\left(\frac{1}{\sqrt{n}}\|x\|_2^2\right),\frac{4}{5}+O(n^{-1})\right).\label{eq:asymptotic var lower bound}
\end{equation}
\end{prop}

\begin{rem}
\label{rem:Maybe cube}The lower bound in Proposition \ref{prop:Lower bound on second moment}
is asymptotically tight. Indeed, if $\eta$ is the isotropic cube,
then $\alpha_{4}-\alpha_{(2,2)}=\alpha_{4}-1=\frac{4}{5}$. 
\end{rem}
\subsubsection{The cubic case}

We now turn to the case of $d=3$. 
\begin{prop}
\label{prop:cubic case}Let $\eta$ be an $H_{n}$-invariant measure
on $\R^{n}$. For $\alpha_{(4,2)}\neq0$,
set 
\[
\beta:=\frac{(n-2)\alpha_{(2,2,2)}-\alpha_{6}}{\alpha_{(4,2)}}+1,\text{ and let }a_{\pm}:=\frac{-\beta\pm\sqrt{\beta^{2}+4(n-1)}}{2}
\]
Then for $n\geq3$, and $\alpha_{(4,2)}\neq0$, the $\mathcal{C}$-eigenspaces in $\mathcal{P}_{3}(\R^{n})$
are as follows: 
\begin{align*}
V_{1} & :=\left\langle \{x_{i}x_{j}x_{k}\}_{i<j<k}\right\rangle \text{ with eigenvalue }\lambda_{1}=\delta_{1}=\alpha_{(2,2,2)}.\\
V_{2} & :=\left\langle \left\{ \frac{1}{\sqrt{2}}(x_{i}^{2}-x_{i+1}^{2})x_{j}\right\} _{i=1...n-1,j\neq i,i+1}\right\rangle \text{ with eigenvalue }\lambda_{2}=\delta_{2}=\alpha_{(4,2)}-\alpha_{(2,2,2)}.\\
V_{3} & :=\left\langle \left\{ (\sum_{i\neq j}x_{i}^{2})x_{j}+a_{+}x_{j}^{3}\right\} _{j=1,...,n}\right\rangle \text{ with eigenvalue }\lambda_{3}=\delta_{3}=\frac{(n-1)\alpha_{(4,2)}}{a_{+}}+\alpha_{6}.\\
V_{4} & :=\left\langle \left\{ (\sum_{i\neq j}x_{i}^{2})x_{j}+a_{-}x_{j}^{3}\right\} _{j=1,...,n}\right\rangle \text{ with eigenvalue }\lambda_{4}=\delta_{4}=\frac{(n-1)\alpha_{(4,2)}}{a_{-}}+\alpha_{6}.
\end{align*}
If $\alpha_{(4,2)}=0$, then $V_3=\left\langle \left\{ x_{i}^{3}\right\}_{i=1}^{n}\right\rangle$, $V_4=\left\langle \left\{ (\sum_{i\neq j}x_{i}^{2})x_{j}\right\} _{j=1}^{n}\right\rangle$, $\lambda_i=\delta_i=0$ for $i\in\{1,2,4\}$, and $\lambda_3=\delta_3=\alpha_6$.

Moreover, the spaces $V_{1}$ and $V_{2}$ are irreducible $H_{n}$-representations
with characters $\chi_{((3),(n-3))}$ and $\chi_{((1),(n-2,1))}$,
both appearing in multiplicity one. On the other hand, the direct
sum $V_{3}\oplus V_{4}$ is the isotypic component of the standard
representation $\chi_{((1),(n-1))}$, which appears in multiplicity
$2$ in $\mathcal{P}_{3}(\R^{n})$. 
\end{prop}

\begin{proof}
Because $d=3$ is odd, $\mathcal{E} = \mathcal{V}$ and we need to make no further distinctions in this case.
The subspace $V_{1}$ is $H_{n}$-stable, and the vector $x_{1}x_{2}x_{3}\in V_{1}$
is $(H_{3,n-3},\mathrm{triv}\otimes\varepsilon_{\{1,2,3\}})$-equivariant.
For $V_{((3),(n-3))}$, the representation associated with $ \mathrm{Ind}_{H_{3,n-3}}^{H_{n}}(\mathrm{triv}\otimes\varepsilon_{\{1,2,3\}})$, as in \eqref{eq:(reps of B_n)}, by Frobenius reciprocity, Theorem \ref{thm:Frobenius reciprocity}:
\[
\langle V_{((3),(n-3))},V_{1}\rangle_{H_{n}}=
\langle\mathrm{triv}\otimes\varepsilon_{\{1,2,3\}},V_{1}\rangle_{H_{3,n-3}}\geq1
\]
But $V_{((3),(n-3))}$ is irreducible, and $\dim V_{((3),(n-3))}=\left|H_{n}:H_{3,n-3}\right|=\binom{n}{3}=\dim V_{1}$.
Now, any non-zero $H_{n}$-morphism $T:V_{((3),(n-3))}\rightarrow V_{1}$
must be an embedding (since $\ker T$ is an $H_{n}$-subrepresentation),
so $V_{1} \simeq V_{((3),(n-3))} \in \mathrm{Irr}(H_n)$.

Next, the subspace $\left\langle \left\{ (x_{i}^{2}-x_{i+1}^{2})x_{1}\right\} _{i=2...n-1}\right\rangle $
in $V_{2}$ is isomorphic to $(\mathrm{triv}\otimes V_{(n-2,1)})\otimes\varepsilon_{\{1\}}$ as $H_{1,n-1}$-representations. Again
by Frobenius reciprocity, 
\[
\langle\mathrm{Ind}_{H_{1,n-1}}^{H_{n}}((\mathrm{triv}\otimes V_{(n-2,1)})\otimes\varepsilon_{\{1\}}),V_{2}\rangle_{H_{n}}=\langle (\mathrm{triv}\otimes V_{(n-2,1)})\otimes\varepsilon_{\{1\}},V_{2}|_{H_{1,n-1}}\rangle_{H_{1,n-1}}\geq1.
\]
Since $V_{((1),(n-2,1))}:=\mathrm{Ind}_{H_{1,n-1}}^{H_{n}}((\mathrm{triv}\otimes V_{(n-2,1)})\otimes\varepsilon_{\{1\}})$ is irreducible and since $\dim V_{2}=\dim V_{((1),(n-2,1))}=n(n-2)$,
as above, we deduce $V_{2}\simeq V_{((1),(n-2,1))}\in\Irr(H_{n})$.

Next, the space $W_{3}:=\left\langle \left\{ x_{i}^{3}\right\} _{i=1}^{n}\right\rangle $
is irreducible of dimension $n$, and isomorphic to the standard representation
$V_{((1),(n-1))}\in\Irr(H_{n})$. Similarly, the subspace $W_{4}=\left\langle \left\{ (\sum_{i\neq j}x_{i}^{2})x_{j}\right\} _{j=1}^{n}\right\rangle $
is irreducible, again isomorphic to $V_{((1),(n-1))}$. Observing
that 
\[
\mathcal{P}_{3}(\R^{n})=V_{1}\oplus V_{2}\oplus W_{3}\oplus W_{4},
\]
we deduce that $W_{3}\oplus W_{4}=\left\langle \left\{ x_{i}^{3}\right\} _{i=1}^{n},\left\{ (\sum_{i\neq j}x_{i}^{2})x_{j}\right\} _{j=1}^{n}\right\rangle $
is the isotypic component of $V_{((1),(n-1))}$ of multiplicity $2$,
and that the other isotypic components are irreducible.

By Corollary \ref{cor:Multiplicity free}, a direct computation shows that $V_{1}$ and $V_{2}$ are
$\mathcal{C}$-eigenspaces of eigenvalues 
\[
\lambda_{1}=\delta_{1}=\alpha_{(2,2,2)}\text{ and }\lambda_{2}=\delta_{2}=\alpha_{(4,2)}-\alpha_{(2,2,2)}.
\]
To find the eigenspaces in $W_{3}\oplus W_{4}$, we know by Frobenius
reciprocity, that each of them contains an $(H_{1,n-1},\varepsilon_{\{1\}})$-equivariant
eigenvector. Note that 
\[
(W_{3}\oplus W_{4})^{(H_{1,n-1},\varepsilon_{\{1\}})}=\left\langle x_{1}^{3},(\sum_{i=2}^{n}x_{i}^{2})x_{1}\right\rangle .
\]
If $\alpha_{(4,2)}=0$, then $W_{3}$ and $W_{4}$ are eigenspaces of eigenvalues $\alpha_{6}$ and $0$, respectively. If $\alpha_{(4,2)}\neq0$, we should search for eigenvectors of the form 
\[
f_{a}:=(\sum_{i=2}^{n}x_{i}^{2})x_{1}+ax_{1}^{3},\text{ for some }a\in\R.
\]
Note that 
\begin{equation}
\E_\eta\left[f_{a}x_{k}^{3}\right]=\begin{cases}
(n-1)\alpha_{(4,2)}+a\alpha_{6} & \text{if }k=1\\
0 & \text{otherwise.}
\end{cases}\label{eq:coefficients of C}
\end{equation}
\begin{equation}
\E_\eta\left[f_{a}x_{k}x_{j}^{2}\right]=\begin{cases}
(n-2)\alpha_{(2,2,2)}+(1+a)\alpha_{(4,2)} & \text{if }k=1\\
0 & \text{otherwise.}
\end{cases}.\label{eq:Coeff of C 2}
\end{equation}
  Hence, in order for $f_{a}$ to be a $\mathcal{C}$-eigenvector, we
want 
\begin{equation}
\frac{(n-1)\alpha_{(4,2)}+a\alpha_{6}}{a}=(n-2)\alpha_{(2,2,2)}+(1+a)\alpha_{(4,2)}.\label{eq:a satisfies}
\end{equation}
Multiplying by $a$, dividing by $\alpha_{(4,2)}$, and rearranging
gives: 
\[
a^{2}+\left(\frac{(n-2)\alpha_{(2,2,2)}-\alpha_{6}}{\alpha_{(4,2)}}+1\right)a-(n-1)=0.
\]
Denote $\beta:=\frac{(n-2)\alpha_{(2,2,2)}-\alpha_{6}}{\alpha_{(4,2)}}+1$,
the solutions are 
\[
a_{\pm}=\frac{-\beta\pm\sqrt{\beta^{2}+4(n-1)}}{2}.
\]
We find that the subspaces generated by $(\sum_{i\neq1}x_{i}^{2})x_{1}+a_{+}x_{1}^{3})$
and $(\sum_{i\neq1}x_{i}^{2})x_{1}+a_{-}x_{1}^{3})$ are the two eigenspaces,
with eigenvalues $\frac{(n-1)\alpha_{(4,2)}}{a_{\pm}}+\alpha_{6}$,
as required. 
\end{proof}
We prove the following in Appendix \ref{sec:Proof-of-Propositions}. 
\begin{prop}
\label{prop:lower bound on spectrum cubic}Let $\eta$ be an $H_{n}$-invariant,
log-concave, isotropic measure on $\R^{n}$, and let $f\in\mathcal{P}_{3}(\R^{n})$
with $\mathrm{coeff}_3(f)=1$. Then 
\[
\E_{\eta}\left[f^{2}\right]\geq\frac{108}{175}+O(n^{-1}).
\]
\end{prop}
As in the case $d=2$, Proposition \ref{prop:lower bound on spectrum cubic}
is asymptotically tight, where the minimum is obtained for the isotropic
cube.

\subsubsection{The case $d=4$}

We now consider the case that $d=4$, which is the first case we don't know whether the spectrum of $\mathcal{E}$ is bounded from below
by an absolute constant, for log-concave measures. To find a potentially pathological eigenvalue,
we focus on the isotypic component of the trivial representation of
$H_{n}$ in $\mathcal{P}_{4}(\R^{n})$, namely:
\[
\mathcal{P}_{4}^{H_{n}}(\R^{n})=\left\langle \sum_{i=1}^{n}x_{i}^{4},\sum_{i<j}x_{i}^{2}x_{j}^{2}\right\rangle .
\]

\begin{lem}
Let $n\geq4$ and suppose $\alpha_{(2,2,2,2)}>0$. Then the eigenvalues of $\mathcal{E}|_{\mathcal{P}_{4}^{H_{n}}(\R^{n})}$
are 
\[
\lambda_{1}=\frac{n^{2}\alpha_{(2,2,2,2)}}{2}+O(n),
\]
and 
\[
\lambda_{2}=\left(n\left(\alpha_{(4,4)}-\frac{\alpha_{(4,2,2)}^{2}}{\alpha_{(2,2,2,2)}}\right)+\frac{\alpha_{(4,2,2)}\left(4\alpha_{(4,4)}-4\alpha_{(6,2)}+5\alpha_{(4,2,2)}\right)}{\alpha_{(2,2,2,2)}}+\alpha_{8}-6\alpha_{(4,4)}\right)(1+O(n^{-1})).
\]
\end{lem}

\begin{proof}
Let $\widetilde{m}_{4,n}:=\frac{1}{\sqrt{n}}m_{4,n}$ and $\widetilde{m}_{(2,2),n}:=\frac{1}{\sqrt{n(n-1)/2}}m_{(2,2),n}$
be as in Definition \ref{def:symmetric monomials}. Then: 
\begin{align*}
 & \mathcal{E}=\left(\begin{array}{cc}
\E_{\eta}[(\widetilde{m}_{4,n})^{2}] & \E_{\eta}[\widetilde{m}_{4,n}\cdot\widetilde{m}_{(2,2),n}]\\
\E_{\eta}[\widetilde{m}_{4,n}\cdot\widetilde{m}_{(2,2),n}] & \E_{\eta}[(\widetilde{m}_{(2,2),n})^{2}]
\end{array}\right)\\
= & \left(\begin{array}{cc}
\frac{1}{n}\E_{\eta}\left[m_{8,n}+2m_{(4,4),n}\right] & \frac{1}{n\sqrt{(n-1)/2}}\E_{\eta}\left[m_{(6,2),n}+m_{(4,2,2),n}\right]\\
\frac{1}{n\sqrt{(n-1)/2}}\E_{\eta}\left[m_{(6,2),n}+m_{(4,2,2),n}\right]\,\, & \frac{2}{n(n-1)}\E_{\eta}\left[m_{(4,4),n}+6m_{(2,2,2,2),n}+2m_{(4,2,2),n}\right]
\end{array}\right)\\
= & \left(\begin{array}{cc}
\alpha_{8}+(n-1)\alpha_{(4,4)} & \sqrt{\frac{n-1}{2}}\left(2\alpha_{(6,2)}+\alpha_{(4,2,2)}(n-2)\right)\\
\sqrt{\frac{n-1}{2}}\left(2\alpha_{(6,2)}+\alpha_{(4,2,2)}(n-2)\right) & \,\,\alpha_{(4,4)}+\frac{1}{2}(n-2)(n-3)\alpha_{(2,2,2,2)}+2(n-2)\alpha_{(4,2,2)}
\end{array}\right).
\end{align*}
Hence, 
\[
\tr(\mathcal{E})=\alpha_{8}+n\alpha_{(4,4)}+\frac{1}{2}(n-2)(n-3)\alpha_{(2,2,2,2)}+2(n-2)\alpha_{(4,2,2)}=\frac{n^{2}\alpha_{(2,2,2,2)}}{2}+O(n).
\]
The determinant is 
\begin{align}
\det(\mathcal{E})= & \left(\alpha_{8}+(n-1)\alpha_{(4,4)}\right)\left(\alpha_{(4,4)}+\frac{1}{2}(n-2)(n-3)\alpha_{(2,2,2,2)}+2(n-2)\alpha_{(4,2,2)}\right)\nonumber \\
 & -\frac{(n-1)}{2}\left(2\alpha_{(6,2)}+\alpha_{(4,2,2)}(n-2)\right)^{2}\nonumber \\
= & \alpha_{8}\alpha_{(4,4)}+\frac{1}{2}(n-2)(n-3)\alpha_{8}\alpha_{(2,2,2,2)}+(n-1)\alpha_{(4,4)}^{2}+\frac{1}{2}(n-1)(n-2)(n-3)\alpha_{(4,4)}\alpha_{(2,2,2,2)}\nonumber \\
 & +2(n-2)\alpha_{(4,2,2)}\alpha_{8}+2(n-1)(n-2)\alpha_{(4,2,2)}\alpha_{(4,4)}\nonumber \\
 & -\frac{(n-1)}{2}\left(4\alpha_{(6,2)}^{2}+\alpha_{(4,2,2)}^{2}(n-2)^{2}+4(n-2)\alpha_{(6,2)}\alpha_{(4,2,2)}\right)\nonumber \\
= & \left((n-3)\alpha_{(4,4)}\alpha_{(2,2,2,2)}-(n-2)\alpha_{(4,2,2)}^{2}\right)\frac{(n-1)(n-2)}{2}\nonumber \\
 & +\frac{1}{2}(n-2)(n-3)\alpha_{8}\alpha_{(2,2,2,2)}-2(n-1)(n-2)\alpha_{(6,2)}\alpha_{(4,2,2)}+2(n-1)(n-2)\alpha_{(4,2,2)}\alpha_{(4,4)}\nonumber \\
 & +(n-1)\left(\alpha_{(4,4)}^{2}-\frac{\alpha_{(6,2)}^{2}}{2}+2\alpha_{(4,2,2)}\alpha_{8}\right)+\alpha_{8}\alpha_{(4,4)}-2\alpha_{(4,2,2)}\alpha_{8}.\nonumber \\
= & \frac{n^{3}}{2}\left(\alpha_{(4,4)}\alpha_{(2,2,2,2)}-\alpha_{(4,2,2)}^{2}\right)\nonumber \\
 & +n^{2}\left(2\alpha_{(4,2,2)}\alpha_{(4,4)}-2\alpha_{(6,2)}\alpha_{(4,2,2)}+\frac{1}{2}\alpha_{8}\alpha_{(2,2,2,2)}-3\alpha_{(4,4)}\alpha_{(2,2,2,2)}+\frac{5}{2}\alpha_{(4,2,2)}^{2}\right)+O(n).\label{eq:formula for det}
\end{align}
Note that the eigenvalues are 
\[
\frac{\tr(\mathcal{E})\pm\sqrt{\tr(\mathcal{E})^{2}-4\det(\mathcal{E})}}{2}=\frac{\tr(\mathcal{E})\pm\tr(\mathcal{E})\sqrt{1-4\frac{\det(\mathcal{E})}{\tr(\mathcal{E})^{2}}}}{2}.
\]
We have $\det(\mathcal{E})=\mathrm{O}(n^{3})$ and $\tr(\mathcal{E})^{2}=\Theta(n^{4})$
so for $n\gg1$, the eigenvalues are: 
\[
\lambda_{1}=\tr(\mathcal{E})+O(n)=\frac{n^{2}\alpha_{(2,2,2,2)}}{2}+O(n),
\]
and 
\[
\lambda_{2}=\frac{\det(\mathcal{E})}{\tr(\mathcal{E})}(1+O(n^{-1}))=\frac{2\det(\mathcal{E})}{n^{2}\alpha_{(2,2,2,2)}}(1+O(n^{-1})).
\]
Plugging (\ref{eq:formula for det}) implies the lemma. 
\end{proof}
\begin{rem}
For the standard Gaussian measure on $\R^{n}$,
$\alpha_{(4,4)}-\frac{\alpha_{(4,2,2)}^{2}}{\alpha_{(2,2,2,2)}}=0$
and then 
\[
\lambda_{2}=\left(4\alpha_{4}^{3}-4\alpha_{6}\alpha_{4}+\alpha_{8}-\alpha_{4}^{2}\right)(1+O(n^{-1}))=24+O(n^{-1})
\]
so the eigenvalues are $(\frac{n^{2}}{2}+O(n),24+O(n^{-1}))$. However, in general it is not true that $\alpha_{(4,4)}\alpha_{(2,2,2,2)}=\alpha_{(4,2,2)}^{2}$.
\end{rem}
\section{On the roles of symmetry, homogeneity, and the normalization} \label{sec:just}
In this section, we expand upon the conditions imposed by Theorem \ref{thm:lpball} and \ref{thm:symmetric}. Specifically, we will explain why changing any of the conditions could either render these results null or lead to completely different statements. Below, we focus on the role of symmetry, the focus on homogeneous polynomials, and the normalization by $\mathrm{coeff}_d$.
\subsection{Lack of symmetry can lead to pathological polynomials} \label{sec:product}
We will show that the assumption of $H_n$-invariance helps us rule out potentially pathological behavior. We demonstrate this through one bad example, where a slight lack of symmetry can create polynomials with poor anti-concentration properties. As will become clear, this construction can be generalized to create many other examples. 
Let $\mu_{n}:=\mu_{n,2}^{\otimes2}$ be
a product of two isotropic $L_{2}$-balls, which is an isotropic and log-concave measure. The next proposition shows that the spectrum of $\mathcal{E}$
(and thus also $\mathcal{V}$) is not bounded from below. 
\begin{prop}
Let $f_{n}:\R^{2n}\rightarrow\R$ be the homogeneous polynomial 
\[
f_{n}=\frac{1}{\sqrt{2n}}\left(x_{1}^{2}+...+x_{n}^{2}-x_{n+1}^{2}-...-x_{2n}^{2}\right).
\]
Then for every $n\in\N$ we have $\mathrm{Var}_{\mu_{n}}(f_{n})\leq \frac{4}{n}$, and 
\[
\mathbb{P}\left(\left|f_{n}\right|<\frac{4}{\sqrt{n}}\right)\geq\frac{3}{4}.
\]
In particular, there are no $C(d)>0$ and $\alpha(d)>0$ such that
for every $\varepsilon>0$: 
\[
\mathbb{P}\left(\left|f_{n}\right|<\varepsilon\right)<C(d)\varepsilon^{\alpha(d)}.
\]
\end{prop}

\begin{proof}
Write $f_{n}=f_{1,n}(x_{1},...,x_{n})+f_{2,n}(x_{n+1},...,x_{2n})$.
Then $\mathbb{E}_{\mu_{n}}[f_{1,n}]=-\mathbb{E}_{\mu_{n}}[f_{2,n}]=\sqrt{\frac{n}{2}}$
and $\mathrm{Var}_{\mu_{n}}(f_{1,n})=\mathrm{Var}_{\mu_{n}}(f_{2,n})=\frac{2}{n+4}$.
Hence, $\mathbb{E}_{\mu_{n}}[f_{n}]=0$ and $\mathrm{Var}_{\mu_{n}}(f_{n})=\frac{4}{n+4}<\frac{4}{n}$.
By Chebyshev's inequality: 
\[
\mathbb{P}\left(\left|f_{n}\right|<\frac{4}{\sqrt{n}}\right)\geq\frac{3}{4}.\qedhere
\]
\end{proof}
\subsection{Non-homogeneous polynomials can have much smaller variance}
Both Theorem \ref{thm:lpball} and Theorem \ref{thm:symmetric} focus on homogeneous polynomials. We will now show that allowing for non-homogeneous polynomials could drastically alter the statement of these results, as well as the classification of the pathological polynomials. For concreteness, we focus on $\mu_{n,2}$, the uniform measure on the isotropic Euclidean ball. Under this measure, let $k \in \mathbb{N}$ and consider the non-homogeneous polynomial $f_{n,k}(x) = n^{-\frac{k}{2}}\left(\|x\|_2^2 - \E_{\mu_{n,2}}[\|x\|_2^2]\right)^k$. This is a polynomial of degree $2k$ with $\mathrm{coeff}_{2k}(f_{n,k}) = \Theta_k(1)$, yet by the reverse H\"older inequality for log-concave measures, as in \cite{carbery2001distributional}, there exists $C_k > 0$ depending only on $k$, such that:
$$\E_{\mu_{n,2}}\left[f_{n,k}^2\right]\leq \frac{C_k}{n^k}\E_{\mu_{n,2}}\left[\left(\|x\|_2^2 - \E_{\mu_{n,2}}[\|x\|_2^2]\right)^2\right]^k = O_k(n^{-k}),$$
where we applied Theorem \ref{thm:Main theorem L^p} for the last inequality. Thus, the lower bound of $\Theta(n^{-1})$ for the minimal variance fails to hold in the general case.

To further explain why the situation is different for non-homogeneous polynomials, we also remark that if $g$ is {\bf any} other fixed polynomial, then by the Cauchy-Schwarz inequality and applying the reverse H\"older inequality again,
$$\E_{\mu_{n,2}}\left[(g\cdot f_{n,k})^2\right]\lesssim \E_{\mu_{n,2}}[g^2]\sqrt{\E_{\mu_{n,2}}\left[(f_{n,k})^4\right]} = \E_{\mu_{n,2}}[g^2]\sqrt{\E_{\mu_{n,2}}\left[(f_{n,2k})^2\right]} \lesssim \frac{\E_{\mu_{n,2}}[g^2]}{n^k}.$$
Thus, $g\cdot f_{n,k}$ can also have a very small second moment. In addition, in the non-homogeneous case it is not clear what the correct normalization is. In the next section, we see that the naive guess, summing the square of all coefficients, does not give desired results. 
\subsection{Normalizing by $\mathrm{coeff}_d$ captures the correct scale}
Finally, we explain the role of $\mathrm{coeff}_d$, which only sees the $d$-homogeneous part of a degree-$d$ polynomial. If $f(x) = \sum\limits_{|I|\leq d}\alpha_Ix^I$ is non-homogeneous, it seems natural to define
$$\mathrm{coeff}(f)^2 = \sum\limits_{0\neq I} \alpha_I^2.$$ 
However, for $f_{n,k}(x) = n^{-\frac{k}{2}}\left(\|x\|_2^2 - \E_{\mu_{n,2}}[\|x\|_2^2]\right)^k$ as in the example above, $\mathrm{coeff}_{2k}(f_{n,k}) = \Theta_k(1)$, yet $\mathrm{coeff}(f_{n,k}) = \Omega_k(n^\frac{k-1}{2})$. On the other hand, if $\gamma_n$ is the standard Gaussian in $\R^n$, then a calculation and another application of the reverse H\"older inequality give:
$$\mathrm{Var}_\gamma(f_{n,k})\leq \E_\gamma\left[f^2_{n,k}\right] = O_k(1).$$
Thus, the normalization $\mathrm{coeff}$ is overly pessimistic: it grows like
$n^{(k-1)/2}$, while the second moment under the Gaussian remains bounded. By contrast,
$\mathrm{coeff}_{2k}$ stays of constant order and therefore captures the correct scale,
consistent with \cite[Theorem 1]{GM22} for non-homogeneous polynomials of product measures.

\appendix


\section{Proof of Proposition \ref{prop:Lower bound on second moment} and Proposition \ref{prop:lower bound on spectrum cubic}}
\label{sec:Proof-of-Propositions}

In this appendix, we prove Propositions \ref{prop:Lower bound on second moment}
and \ref{prop:lower bound on spectrum cubic}, thus completing the proof of Theorem \ref{thm:symmetric}. We start with Lemmas
\ref{lem:Inequality on mements} and \ref{lem:mixed moments} below. The first lemma, about possible growth of moments of even log-concave measures, is due to Egor Kosov.
\begin{lem}
\label{lem:Inequality on mements}Let $\rho$ be an even log-concave
density on $\R$. Then for any $k,m\in\N$: 
\[
\left(\int_{-\infty}^{\infty}t^{m+k}\rho(t)dt\right)^{2}\leq\left(1-\frac{(m-k)^{2}}{(m+k+1)^{2}}\right)\left(\int_{-\infty}^{\infty}t^{2m}\rho(t)dt\right)\left(\int_{-\infty}^{\infty}t^{2k}\rho(t)dt\right).
\]
\end{lem}

\begin{proof}
Without loss of generality, we can assume that $m+k$ is an even number.
We first consider the case where $\rho(t)$ is the density of the uniform
distribution on a symmetric segment $[-a,a]$. In this case, on the
one hand, we have 
\[
\left(\frac{1}{2a}\int_{-a}^{a}t^{m+k}\,dt\right)^{2}=\left(\frac{1}{a}\int_{0}^{a}t^{m+k}\,dt\right)^{2}=\frac{a^{2m+2k}}{(m+k+1)^{2}}.
\]
On the other hand, 
\[
\frac{1}{2a}\int_{-a}^{a}t^{2m}\,dt=\frac{1}{a}\int_{0}^{a}t^{2m}\,dt=\frac{a^{2m}}{2m+1},\text{ \,\,and\,\, }\frac{1}{2a}\int_{-a}^{a}t^{2k}\,dt=\frac{a^{2k}}{2k+1}.
\]
Thus 
\begin{equation}
\frac{\left(\frac{1}{2a}\int_{-a}^{a}t^{m+k}\,dt\right)^{2}}{\left(\frac{1}{2a}\int_{-a}^{a}t^{2m}\,dt\right)\left(\frac{1}{2a}\int_{-a}^{a}t^{2k}\,dt\right)}=\frac{(2m+1)(2k+1)}{(m+k+1)^{2}}=1-\frac{(m-k)^{2}}{(m+k+1)^{2}}.\label{eq:even segment}
\end{equation}
We now assume that $\rho$ is any even log-concave density, so the
level sets $A_{s}:=\{t:\rho(t)\ge s\}=[-a_{s},a_{s}]$ are symmetric
segments. Thus, 
\begin{align*}
\int_{-\infty}^{\infty}t^{m+k}\rho(t)\,dt & =\int_{-\infty}^{\infty}t^{m+k}\int_{0}^{\infty}1_{\{\rho(t)\ge s\}}(s,t)\,ds\,dt=\int_{0}^{\infty}\int_{-\infty}^{\infty}t^{m+k}1_{A_{s}}(t)\,dt\,ds\\
 & =\int_{0}^{\infty}\left(\frac{1}{2a_{s}}\int_{-a_{s}}^{a_{s}}t^{m+k}\,dt\right)2a_{s}\,ds.\tag{\ensuremath{\star}}
\end{align*}
By (\ref{eq:even segment}), by the Cauchy--Schwarz inequality, and by
Fubini's theorem, 
\begin{align*}
(\star) & =\left(1-\frac{(m-k)^{2}}{(m+k+1)^{2}}\right)^{1/2}\int_{0}^{\infty}\left(\frac{1}{2a_{s}}\int_{-a_{s}}^{a_{s}}t^{2m}dt\right)^{\frac{1}{2}}\left(\frac{1}{2a_{s}}\int_{-a_{s}}^{a_{s}}t^{2k}\,dt\right)^{\frac{1}{2}}2a_{s}\,ds\\
 & =\left(1-\frac{(m-k)^{2}}{(m+k+1)^{2}}\right)^{1/2}\int_{0}^{\infty}\left(\int_{-\infty}^{\infty}t^{2m}1_{A_{s}}(t)\,dt\right)^{\frac{1}{2}}\left(\int_{-\infty}^{\infty}t^{2k}1_{A_{s}}(t)\,dt\right)^{\frac{1}{2}}\,ds\\
 & \le\left(1-\frac{(m-k)^{2}}{(m+k+1)^{2}}\right)^{1/2}\left(\int_{0}^{\infty}\int_{-\infty}^{\infty}t^{2m}1_{\{\rho(t)\ge s\}}(s,t)\,dt\,ds\right)^{\frac{1}{2}}\left(\int_{0}^{\infty}\int_{-\infty}^{\infty}t^{2k}1_{\{\rho(t)\ge s\}}(s,t)\,dt\,ds\right)^{\frac{1}{2}}\\
 & =\left(1-\frac{(m-k)^{2}}{(m+k+1)^{2}}\right)^{1/2}\left(\int_{-\infty}^{\infty}t^{2m}\rho(t)\,dt\right)^{\frac{1}{2}}\left(\int_{-\infty}^{\infty}t^{2m}\rho(t)\,dt\right)^{\frac{1}{2}}.
\end{align*}
The lemma is proved. 
\end{proof}
\begin{lem}
\label{lem:mixed moments}For every $H_{n}$-invariant, log-concave,
isotropic measure $\eta$ on $\R^{n}$: 
\[
\alpha_{(2,2,2)}=\alpha_{(2,2)}=1+O(n^{-1})\text{ and }\alpha_{(4,2)}=\alpha_{4}+O(n^{-1}).
\]
\end{lem}

\begin{proof}
Since $\eta$ is isotropic, the thin-shell theorem \cite[Theorem 1.1]{klartag2025thin}
gives an a priori bound on $\mathrm{Var}_{\eta}(\left\Vert x\right\Vert _{2}^{2})$,
so 
\begin{equation}
n(n-1)\alpha_{(2,2)}+n\cdot\alpha_{4}=\E_{\eta}\left[\left\Vert x\right\Vert _{2}^{4}\right]=\mathrm{Var}_{\eta}(\left\Vert x\right\Vert _{2}^{2})+n^{2}=n^{2}+O(n).\label{eq:thin shell}
\end{equation}
Similarly, by the reverse  inequality for log-concave measures
(See \cite{carbery2001distributional} for example), there is an absolute
constant $C>0$ such that: 
\[
\E_{\eta}[\left\Vert x\right\Vert _{2}^{6}]-3n\E_{\eta}[\left\Vert x\right\Vert _{2}^{4}]+2n^{3}=\E_{\eta}\left[\left(\left\Vert x\right\Vert _{2}^{2}-n\right)^{3}\right]\leq C\E_{\eta}\left[\left(\left\Vert x\right\Vert _{2}^{2}-n\right)^{2}\right]^{\frac{3}{2}}=O(n^{\frac{3}{2}}).
\]
Hence, 
\[
\E_{\eta}[\left\Vert x\right\Vert _{2}^{6}]-n^{3}=3n\E_{\eta}[\left\Vert x\right\Vert _{2}^{4}]-3n^{3}+O(n^{\frac{3}{2}})=O(n^{2}).
\]
Since $\E_{\eta}[\left\Vert x\right\Vert _{2}^{6}]=n^{3}\alpha_{(2,2,2)}+O(n^{2})$,
we deduce that $\alpha_{(2,2,2)}=1+O(n^{-1})$. Furthermore, a thin-shell estimate for unconditional log-concave measures (see e.g.~\cite[Lemma 12.4.8]{BGVV14}) also applies to $\|x\|_4^4$,
\begin{equation}
\mathrm{Var}_{\eta}(\left\Vert x\right\Vert _{4}^{4})=O(n).\label{eq:variance of 4-norm}
\end{equation}
We now turn to estimate $\alpha_{(4,2)}$. By the Cauchy-Schwarz inequality,
and by (\ref{eq:thin shell}) and (\ref{eq:variance of 4-norm}),
\begin{align*}
\left|n(n-1)\alpha_{4,2}+n\alpha_{6}-n^{2}\alpha_{4}\right|  &=\left|\E_{\eta}\left[(\left\Vert x\right\Vert _{4}^{4}-n\alpha_{4})(\left\Vert x\right\Vert _{2}^{2}-n)\right]\right|\\
&
\leq\left(\mathrm{Var}_{\eta}(\left\Vert x\right\Vert _{4}^{4})\mathrm{Var}_{\eta}(\left\Vert x\right\Vert _{2}^{2})\right)^{\frac{1}{2}}=O(n).
\end{align*}
Dividing by $n^{2}$, we deduce that $\alpha_{4,2}=\alpha_{4}+O(n^{-1})$,
and we are done. 
\end{proof}
\begin{prop}
\label{prop:quadratic}Let $\eta$ be an $H_{n}$-invariant, log-concave,
isotropic measure on $\R^{n}$, and let $f\in\mathcal{P}_{2}(\R^{n})$
with $\mathrm{coeff}_2(f)=1$. Then 
\[
\E_{\eta}[f^{2}]\geq\frac{4}{5}+O(n^{-1}) > 0,
\]
and 
\[
\mathrm{Var}_{\eta}(f^{2})\geq\min\left(\mathrm{Var}_\eta\left(\frac{1}{\sqrt{n}}\|x\|_2^2\right),\frac{4}{5}+O(n^{-1})\right).
\]
\end{prop}

\begin{proof}
Firstly, we may assume that $n\gg1$. Indeed for each fixed $n\in\N$,
we can bound the density of $\eta$ by the indicator function of a
small Euclidean ball in $\R^{n}$, up to some constant $C(n)$ (see
e.g. \cite[Theorem 5.14]{LV07}). For isotropic balls the proposition
is known by \cite[Theorem 2]{GM22}, which implies the proposition
for every fixed $n$. 

Secondly, by Lemma \ref{lem:The-minimal-eigenvalue}, it is enough
to show that 
\begin{equation}
\min\left\{ \lambda_{1},\lambda_{2},\lambda_{3}\right\} =\min\left\{ \alpha_{4}+(n-1)\alpha_{(2,2)},\alpha_{4}-\alpha_{(2,2)},\alpha_{(2,2)}\right\} \geq\frac{4}{5}+O(n^{-1}).\label{eq:enough to show}
\end{equation}
Applying Lemma \ref{lem:Inequality on mements} with $m=2$ and $k=0$
gives $\alpha_{4}\geq\frac{9}{5}$. (\ref{eq:enough to show}) now
follows from Lemma \ref{lem:mixed moments}.
Finally, the variance bound follows by Proposition \ref{prop: quadratic case}, as the only difference between $\mathcal{V}$ and $\mathcal{E}$ lies in the eigenspace corresponding to the eigenvalue
$\delta_1 = \mathrm{Var}_\eta\left(\frac{1}{\sqrt{n}}\|x\|_2^2\right).$
\end{proof}
\begin{prop}
Let $\eta$ be an $H_{n}$-invariant, log-concave, isotropic measure
on $\R^{n}$, and let $f\in\mathcal{P}_{3}(\R^{n})$ with $\mathrm{coeff}_3(f)=1$.
Then 
\[
\E_{\eta}[f^{2}]\geq\frac{108}{175}+O(n^{-1}) >0.
\]
\end{prop}

\begin{proof}
As in the case of $d=2$, we may assume that $n\gg1$. By Lemma \ref{lem:The-minimal-eigenvalue} and Proposition \ref{prop:cubic case},
it is enough to show that all eigenvalues $\lambda_{1},\lambda_{2},\lambda_{3},\lambda_{4}$
are bounded from below by $\frac{108}{175}+O(n^{-1})$. By Lemma \ref{lem:mixed moments}
and the proof of Proposition \ref{prop:quadratic}, $\lambda_{1}=1+O(n^{-1})$
and $\lambda_{2}\geq\frac{4}{5}+O(n^{-1})$, so it is left to bound
$\lambda_{3},\lambda_{4}$. Recall that 
\[
\lambda_{3}=\frac{(n-1)\alpha_{(4,2)}}{a_{+}}+\alpha_{6}\text{ and }\lambda_{4}=\frac{(n-1)\alpha_{(4,2)}}{a_{-}}+\alpha_{6},
\]
with 
\[
a_{\pm}=\frac{-\beta\pm\sqrt{\beta^{2}+4(n-1)}}{2},\text{ where }\beta=\frac{(n-2)\alpha_{(2,2,2)}-\alpha_{6}}{\alpha_{(4,2)}}+1.
\]
Since $\sqrt{1+x}<1+x/2$ for $x>0$, we have 
\[
a_{+}=\frac{-\beta+\beta\sqrt{1+\frac{4(n-1)}{\beta^{2}}}}{2}\leq\frac{-\beta+\beta(1+\frac{2(n-1)}{\beta^{2}})}{2}=\frac{n-1}{\beta},
\]
and since $\alpha_{(2,2,2)}=\E_{\eta}[x_{1}^{2}x_{2}^{2}x_{3}^{2}]>\frac{1}{2}$
for $n\gg1$, we have: 
\[
\lambda_{3}=\frac{(n-1)\alpha_{(4,2)}}{a_{+}}+\alpha_{6}\geq\beta\alpha_{(4,2)}+\alpha_{6}=(n-2)\alpha_{(2,2,2)}+\alpha_{(4,2)}>\frac{n-2}{2}.
\]
for $n\gg1$. It is left to bound $\lambda_{4}$. Now we have 
\begin{equation}
a_{-}=\frac{-\beta-\beta\sqrt{1+\frac{4(n-1)}{\beta^{2}}}}{2}\leq-\frac{2\beta}{2}=-\beta.\label{eq:ineq}
\end{equation}
Hence, by (\ref{eq:ineq}) and by Lemma \ref{lem:mixed moments},
we deduce that
\begin{align*}
\lambda_{4} & =\alpha_{6}+\frac{(n-1)\alpha_{(4,2)}}{a_{-}}\geq\alpha_{6}-\frac{(n-1)\alpha_{(4,2)}}{\beta}=\alpha_{6}-\frac{(n-1)\alpha_{(4,2)}^{2}}{(n-2)\alpha_{(2,2,2)}-\alpha_{6}+\alpha_{(4,2)}}\\
 & =\alpha_{6}-\frac{\alpha_{(4,2)}^{2}}{\alpha_{(2,2,2)}}+O(n^{-1})=\alpha_{6}-\alpha_{4}^{2}+O(n^{-1}).
\end{align*}
Applying Lemma \ref{lem:Inequality on mements} with $m=3$ and $k=1$
we deduce that $\alpha_{6}\geq\frac{25}{21}\alpha_{4}^{2}$. We have
seen that $\alpha_{4}\geq\frac{9}{5}$, so $\alpha_{6}-\alpha_{4}^{2}\geq\frac{4}{21}\alpha_{4}^{2}\geq\frac{108}{175}$,
which concludes the proof. 
\end{proof}
\bibliographystyle{alpha}
\bibliography{bibfile}

\end{document}